\documentclass[a4paper,12pt]{article} 

\usepackage{graphicx}

\usepackage{amssymb,latexsym}
\usepackage{bm, amsmath, amssymb, amsthm}
\usepackage[right=3.0cm, left=3.0cm, top=3.0cm, bottom=3.0cm]{geometry}

\newtheorem{theorem}{Theorem}[section]
\newtheorem{proposition}[theorem]{Proposition}
\newtheorem{corollary}[theorem]{Corollary} 
\newtheorem{lemma}[theorem]{Lemma}

\theoremstyle{definition}

\newtheorem{example}[theorem]{Example}
\newtheorem{remark}[theorem]{Remark}

\newcommand\eps{\varepsilon}

\newcommand\mD{{\cal D}}

\newcommand\mF{{\cal F}}

\newcommand\mH{{\cal H}}

\newcommand\mV{{\cal V}}

\newcommand\mN{{\cal N}}

\newcommand\mA{{\cal A}}

\newcommand\Riem{\operatorname{Riem}}

\newcommand\dt{\partial_t}
\newcommand\ddt{\partial^2_{tt}}

\newcommand\ee{{\cal E}}

\begin{document}

\title{Non-canonical variations of Riemannian submersions with totally geodesic fibers}

\author{ Tomasz Zawadzki\footnote{Faculty of Mathematics and Computer Science, University of \L\'{o}d\'{z}, ul. Banacha 22, 90-238 \L\'{o}d\'{z}, Poland;
e-mail: {\tt tomasz.zawadzki@wmii.uni.lodz.pl}  }}

\date{}

\maketitle 

\begin{abstract}
Using variations of Riemannian metric that preserve a given Riemannian submersion, keep its fibers totally geodesic and the metric restricted to the fibers fixed, but change the horizontal distribution, we examine changes of sectional curvatures in horizontal and vertical directions. 
We obtain conditions, in terms of a 1-form defining a variation, to locally make  all sectional curvatures positive on the product of a manifold with positive curvature and a circle, while preserving the Riemannian submersion with geodesic fibers defined by the projection from the product. 
We examine conditions for obtaining weak contact metric structures from K-contact structures.
We demonstrate existence of fat Riemannian submersions with totally geodesic fibers and vertizontal (i.e., spanned by a horizontal and a vertical vector) curvatures non-constant along a fiber. 
For a Riemannian submersion defined by an isometric group action, with totally geodesic fibers of dimension higher than one, we find variations that preserve the isometric action, while changing the horizontal distribution and its curvatures.

\noindent
\textbf{Keywords}: Riemannian submersion, distribution, variation, curvature, homogeneous submersion, contact metric structure

\noindent
\textbf{Mathematics Subject Classifications (2020)} 
53C12; 
53C15
\end{abstract}

\section{Introduction}\label{sec1}

Deformations of metric are a useful tool in Riemannian geometry and find many applications, e.g., in 
modifying curvature of a manifold
\cite{Bourgignon, PetersenWilhelm2, Wilking, Ziller}, searching for critical points of functionals \cite{Besse}, and distinguishing families of metrics that share some properties \cite{Cheeger, Goertsches2008}. 
Deformations defining one-parameter families of metric 
are called variations.

Many interesting examples of Riemannian metrics, particularly those of positive or non-negative curvature, are related to Riemannian submersions with totally geodesic fibers \cite{Wilking, Ziller}.  
The most common approach to deforming metric while preserving a Riemannian submersion with totally geodesic fibers is to change the metric only on the fibers, keeping the horizontal distribution and the metric on it fixed  \cite{PetersenWilhelm}.
Particular methods to do that 
include canonical variations \cite{Besse}, where metric on fibers is rescaled by a constant, 
and Cheeger deformations  \cite{Cheeger}, that use a bi-invariant metric on a Lie group, whose isometric action defines a Riemannian submersion.

In this paper 
a different, and, in a way, complementary kind of deformations are examined: the ones that preserve the metric on fibers, while changing the horizontal distribution itself and the metric on it, in order to keep it isometric to the one on the image of the submersion. 
As proved in \cite{TZarxiv}, certain families of such deformations keep the fibers totally geodesic - they are defined by vector fields, whose restrictions to fibers are Killing.
These vector fields, paired with horizontal lifts of linearly independent vectors from the base of
submersion, locally describe parameters controlling the ODE system, whose solution is a deformed metric. 

We examine these particular deformations in 
detail. After fixing notation in Section \ref{secnotation}, and recalling from \cite{TZarxiv} necessary technical results about variations preserving Riemannian submersions with totally geodesic fibers in Section \ref{sectionvarpreserving},
we investigate in Section \ref{secvariationsofcurvature} how these variations change horizontal and vertizontal \cite{ZillerFatness} curvatures, i.e., sectional curvatures in directions naturally related to the submersion.
As the variations 
preserve the Riemannian submersion with totally geodesic fibers
by design, known formulas for such submersions \cite{Gray, O'Neill, Tondeur} can be readily applied.
Apart from making some computations easier, this approach allows a complete description 
of all non-vanishing derivatives of these curvatures with respect to the parameter of variation, in case when
vector fields defining the variation do not change with the parameter,
which for general variations of metric is more difficult to achieve \cite{Bourgignon}.

These results are then applied in Section \ref{sectioncirclebundle}, where we examine 
deformations of metric on a circle bundle, with unit Killing vector field $U$, which defines a Riemannian submersion with fibers being its flowlines. In this case, variations can be described 
by a single basic $1$-form, whose properties determine derivatives of sectional curvatures.
In particular, we prove that the second derivative of all vertizontal
curvatures is positive where the $1$-form has non-degenerate exterior differential.

In Section \ref{subsectrivialbundles}, we prove that 
a product metric on $N \times S^1$ can be deformed 
to make all sectional 
curvatures positive on fibers over points where $N$ has positive curvature, differential of the $1$-form is non-degenerate and its covariant derivative satisfies a certain condition. 
This illustrates difficulties with 
increasing all
curvatures at every point of a compact manifold, but may help to construct metrics with curvatures positive outside particular sets.

Having a complete description of all possible deformations preserving a 
Riemannian submersion, its geodesic fibers and the metric on every fiber,
we can apply it to other geometric structures that may coexist with such submersion.
In Section \ref{subsecKcontact}, we recover some known results about deformations of $K$-contact and Sasaki structures \cite{Boyer, Goertsches2008} defined by closed and basic forms, which are in fact induced by diffeomorphisms of the domain of the submersion.
Then, in Section \ref{subsetweakcontact}, we examine conditions to obtain weak contact metric structures \cite{RovenskiAMPA} through variations preserving Riemannian submersion on a $K$-contact or Sasaki manifold.

To define variations preserving Riemannian submersion with totally geodesic fibers, in general we use vector fields, which restrict to Killing fields on fibers.
If 
fibers have dimension higher than one, as considered in Section \ref{sechighdim}, such fields are not necessarily Killing fields on the manifold. 
Using this observation, in Section \ref{subsecfatbundles}, we 
construct fat bundles with vertizontal curvatures non-constant along fibers. This construction can be in particular applied to 
the Hopf fibration of $S^{4n+3}$ by $3$-spheres (Example \ref{exampleHopf}).

In Section \ref{subsechomogenous}, we 
examine a more general
setting of an isometric Lie group action on a manifold, with principal orbits and non-integrable horizontal distribution.
The action defines so called   
fundamental vector fields, which are Killing on the manifold, and correspond to right
invariant vector fields on the group; but  
vertical parts of Lie brackets
of basic fields also restrict to Killing fields on fibers, and correspond to left-invariant vector fields on the group \cite{GromollWalschap}. Variations
induced by each of these families of vector fields have different properties, the second
family is defined only locally, but 
preserves the original isometric action on the manifold - 
yielding new families of metrics with  
non-constant vertizontal curvature, preserving the structure of homogeneous Riemannian submersion.

On the other hand, for isometric action of a non-abelian group, variations defined by fundamental fields make the fundamental fields stop being Killing. In Section \ref{subsecsu2actions}, in a particular case of a
$3$-Sasaki manifold, we prove that if those variations preserve an almost contact metric $3$-structure \cite{Blair}, then they are induced by diffeomorphisms of the domain.

\section{Notation and definitions} \label{secnotation}

In what follows, $M$ and $N$ are smooth, connected and oriented manifolds, and all functions and tensor fields on them are always assumed to be smooth.

For a manifold $M$, $\mathfrak{X}_M$ denotes the set of vector fields on $M$. For a submersion $\pi : M \rightarrow N$, let $\mathfrak{X}_{\mV}$ denote the set of vertical vector fields, i.e., vector fields tangent to the fibers of $\pi$. The vertical distribution, spanned by all vertical fields, will be denoted by $\mV$. For a distribution $\mD$ we write $\mD_x = \mD \cap T_x M$.

Let $(M,g_t)$ be a Riemannian manifold and let $\pi : (M,g_t) \rightarrow (N, g_N)$ be a submersion. We say that $X \in \mathfrak{X}_M$ is a $g_t$-horizontal vector field if $g_t(X, \xi)=0$ for all $\xi \in \mathfrak{X}_{\mV}$; we write then $X \in \mathfrak{X}_{\mH(t)}$ and we denote the $g_t$-horizontal distribution, spanned by all $g_t$-horizontal fields, by $\mH(t)$.

A submersion $\pi : (M,g_t) \rightarrow (N, g_N)$ is called a Riemannian submersion if for all $X,Y \in \mathfrak{X}_{\mH(t)}$ we have
\[
g_t(X,Y) = g_N (\pi_* X, \pi_* Y).
\]

We will use the notation $\dim M = n+p$ and $\dim N = p$, then $\dim \mV = n$.
We say that a vector field $X \in \mathfrak{X}_M$ is 
projectable if $\pi_* X$ is well defined; $X$ is 
projectable if and only if $[X, V] \in \mathfrak{X}_{\mV}$ for all $V \in \mathfrak{X}_{\mV}$. A projectable and $g_t$-horizontal vector field is called a $g_t$-basic vector field \cite{O'Neill}.

For every $W \in \mathfrak{X}_N$ there exists a unique $g_t$-basic lift of $W$, i.e., a unique $g_t$-horizontal vector field on $M$ whose image in $\pi_*$ is $W$ \cite{O'Neill}. 
A $1$-form $\omega$ on $M$ is called basic if for all $V \in \mathfrak{X}_{\mV}$ we have $\iota_V \omega=0$ and $\iota_V {\rm d} \omega=0$, then there exists a $1$-form $\theta$ on $N$ such that $\omega=\pi^* \theta$  \cite{Tondeur}. 
For all $x \in N$, let $\mF_x = \pi^{-1}(\{x\})$ be the fiber of $\pi$ over $x$.

Let $\pi : (M, g_0 ) \rightarrow (N,g_N)$ be a Riemannian submersion. 
By $\Riem(M, {\mV}, g_0)$ we denote the set of all Riemannian metrics $g$ on $M$, such that $g(X,Y) = g_0(X,Y)$ for all $X,Y \in \mathfrak{X}_{{\mV}}$. We will consider one-parameter families of metrics $\{ g_t , t \in (-\epsilon, \epsilon) \}$ 
in $\Riem(M, {\mV}, g_0)$, 
which are called variations of $g_0$, and are always assumed to smoothly depend on the parameter $t$. The first and the second derivative with respect to $t$ will be denoted by $\dt$ and $\ddt$, respectively. For every variation $g_t$, let $B_t = \dt g_t$ and let $B_t^\sharp$ be the tensor field defined by $g_t(B_t^\sharp X, Y) = B_t(X,Y)$ for all $X,Y \in \mathfrak{X}_M$. We say that a vector field $X \in \mathfrak{X}_M$ is bounded on $(M,g)$ if $g(X,X)$ is a bounded function on $M$. For $W \in \mathfrak{X}_N$, we will denote by $\pi^* W$ the $g_0$-basic lift of $W$, 
we have $\pi_* (\pi^* W) =W$. 

We say that a variation $\{ g_t , t \in (-\epsilon, \epsilon) \subset \Riem(M, {\mV}, g_0)$ preserves a Riemannian submersion $\pi$ if $\pi : (M,g_t) \rightarrow (N,g_N)$ is a Riemannian submersion for all $t \in (-\epsilon, \epsilon)$.

Let $P_{\mV}^t$ and $P_{\mH}^t $ denote the $g_t$-orthogonal projections onto the distributions $\mV$ and $\mH(t)$, respectively. We shall consider sets of linearly independent
vector fields $\{ E_1, \ldots, E_n , \allowbreak \ee_{n+1} , \ldots \ee_{n+p} \}$ defined on an open subset of $M$, where $E_a \in \mathfrak{X}_{\mV}$ for all $a \in \{1, \ldots, n\}$; we will call them local adapted frames and denote such a frame shortly by $\{E_a, \ee_i\}$. If vector fields in a frame $\{E_a, \ee_i\}$ are $g$-orthonormal, we will call such frame $g$-orthonormal. For indices we use the convention $a,b,c \in \{1, \ldots, n\}$, $i,j,k \in \{n+1, \ldots, n+p\}$ and $\mu, \nu \in \{1, \ldots, n+p\}$.

Let $\nabla^t$ be the Levi-Civita connection of $g_t$, let $R_t(X,Y)Z = \nabla^t_X \nabla^t_Y Z - \nabla^t_Y \nabla^t_X Z - \nabla^t_{[X,Y]}Z$ be its curvature tensor, and let $\nabla^N$ be the Levi-Civita connection of $g_N$. Let $\sec_{(M,g_t)}(X,Y)$ denote the sectional curvature of the plane spanned by $X$ and $Y$ on $(M,g_t)$; if $X,Y \in \mH(t)$, we call such sectional curvature horizontal; if $X \in \mH(t)$ and $Y \in \mV$, we call it vertizontal \cite{ZillerFatness}.  
We say that the fibers of $\pi : (M, g_t) \rightarrow (N, g_N)$ 
are totally geodesic if for all $X,Y \in \mathfrak{X}_{\mV}$ we have $P_{\mH}^t  \nabla^t_X Y =0$. For a Riemannian submersion $\pi : (M, g_t) \rightarrow (N, g_N)$ the geometry of distributions $\mV$ and $\mH(t)$ is described by the following tensor \cite{O'Neill}
\begin{equation} \label{definitionAtensor}
\mA^t_X Y = P_{\mV}^t \nabla^t_{ P_{\mH}^t X} P_{\mH}^t Y + P_{\mH}^t \nabla^t_{ P_{\mH}^t X } P_{\mV}^t Y,
\end{equation}
for all $X,Y \in \mathfrak{X}_M$.
For a vector field $X$ on $(M,g_t)$, by $X^\flat$ we denote its dual $1$-form, i.e., $X^\flat (Y) = g_t(X,Y)$ for all $Y \in \mathfrak{X}_M$. The inner product of tensors defined by $g_t$ restricted to $\mH(t)$ will be denoted by $\langle \cdot , \cdot \rangle_{\mH(t)}$, in particular for $1$-forms $\omega, \eta$ and a local adapted frame $\{E_a, \ee_i\}$ we have
\[
\langle \omega , \eta \rangle_{\mH(t)} = \sum\nolimits_{i=n+1}^{n+p} \omega(\ee_i) \eta(\ee_i).
\]
Similarly, we denote $\langle \omega , \eta \rangle_N = \sum\nolimits_{i=1}^{p} \omega(e_i) \eta(e_i)$ for a local $g_N$-orthonormal basis $\{e_1 ,\ldots ,e_p\}$ on $N$.

\section{Variations preserving a Riemannian submersion with totally geodesic fibers} \label{sectionvarpreserving} 

In this section we recall some necessary technical results from \cite{TZarxiv}, formulating them in a way that will fit further results and providing their shortened proofs for readers' convenience. 
For more details about varying extrinsic geometry of fibers while preserving a Riemannian submersion, and applications to some variational problems, see \cite{TZarxiv}.

\begin{theorem} \label{corbsharpglobal}
Let  $\pi : (M, g_0) \rightarrow (N, g_N)$ be a Riemannian submersion.
For every pair of sets: $\{ V_{n+1} , \ldots , V_{n+p} \}$ of vertical vector fields, 
bounded 
on $(M,g_0)$, and $\{ W_{n+1} , \ldots , W_{n+p} \}$ of linearly independent vector fields, 
bounded 
on $(N,g_N)$, there exists a unique 
variation $\{ g_t, t \in (-\epsilon, \epsilon) \} \subset \Riem(M, \mV , g_0)$ that preserves the Riemannian submersion $\pi$, such that for every $X \in \mathfrak{X}_{\mV}$
\begin{equation} \label{bsharpinv}
B_t^\sharp X = \sum\nolimits_{i=n+1}^{n+p} g_0(V_i , X) P_{\mH}^t  \pi^* W_i ,
\end{equation}
where $\pi^* W_i$ is the unique $g_0$-horizontal lift of $W_i$, for all $i \in \{n+1, \ldots , n+p\}$.
Moreover, if for all $x \in N$ restrictions of vector fields $V_i$ to fiber $\mF_x$ are Killing vector fields on $(\mF_x , g_0 \vert_{\mF_x})$, then the fibers of $\pi : (M ,g_t ) \rightarrow (N,g_N)$ are totally geodesic for all $t \in (-\epsilon , \epsilon)$.
\end{theorem}
\begin{proof}
Let $\{ E_a , \ee_i \}$ be a local adapted 
$g_0$-orthonormal frame on an open set $U' \subset M$, and let $U$ be a relatively compact, open set, such that $\overline{U} \subset U'$. 
Then there exist functions $\{ \lambda_{ai}  : U' 
 \rightarrow \mathbb{R} \}$ such that
\begin{equation} \label{deflambdaai}
\sum\nolimits_{i=n+1}^{n+p} g_0(V_i , E_a) \pi^*(W_i) = \sum\nolimits_{i=n+1}^{n+p} \lambda_{ai} \ee_i .
\end{equation}
For these $\{\lambda_{ai}\}$, let
$\{ g_{\mu \nu}(t) , \mu , \nu \in \{1 , \ldots, n+p \} \}$
be the solution of the following system of 
equations in $U$:
\begin{eqnarray} \label{bsharp1}
&& g_{ab} = \delta_{ab} , \\ \label{bsharp2}
&& \dt g_{ai} = \sum\nolimits_{j=n+1}^{n+p} \left( \lambda_{aj} g_{ij} - \sum\nolimits_{b=1}^n \lambda_{aj} g_{jb} g_{bi}  \right) \\
\label{bsharp3}
&& \dt g_{ij} = \sum\nolimits_{a=1}^n \sum\nolimits_{k=n+1}^{n+p} \big(
 g_{ik} g_{ja} \lambda_{ak} + g_{jk} g_{ia} \lambda_{ak} \nonumber \\
&& - \sum\nolimits_{c=1}^n \left( g_{jc} g_{ia} g_{kc} \lambda_{ak} + g_{ic} g_{ja} g_{kc} \lambda_{ak} \right) \big),
\end{eqnarray}
with $g_{\mu \nu}(0) = \delta_{\mu \nu}$, 
for all $a,b \in \{1, \ldots, n\}$ and all $i,j \in \{ n+1 , \ldots, n+p \}$.

Tensor field defined by $g_t(E_a, E_b) = g_{ab}(t)$, $g_t(E_a, \ee_i) = g_{ai}(t)$ and $g_t(\ee_i, \ee_j) = g_{ij}(t)$ 
does not depend on the choice of adapted $g_0$-orthonormal frame $\{E_a, \ee_i\}$.  
Indeed, any two such frames $\{E_a, \ee_i\}$ and $\{ E'_c , \ee'_k \}$ are related by a $g_0$-orthogonal transformation preserving each of distributions $\mV$ and $\mH(0)$, so $\ee'_k = \sum\nolimits_{i=n+1}^{n+p} M_{ik} \ee_i$, $E'_c = \sum\nolimits_{a=1}^{n} M_{a c} E_a$, with $M_{\mu \nu}$ being coefficients of a matrix $M \in O(n) \times O(p)$. 
Hence, e.g., 
$g'_{kc} = \sum\nolimits_{i=n+1}^{n+p} \sum\nolimits_{a=1}^{n} M_{ik} M_{ac} g_{ai}$. 
Functions $\{ \lambda_{ai} \}$ defined by \eqref{deflambdaai} 
transform with the frame change as follows
\begin{equation} \label{lambdatransform}
\lambda'_{ck} = \sum\nolimits_{i=n+1}^{n+p} \sum\nolimits_{a=1}^{n} \lambda_{ai} M_{ik} M_{ac}. 
\end{equation} 
Comparing the system 
\eqref{bsharp1}-\eqref{bsharp3} 
written 
in frames $\{E_a, \ee_i\}$ and $\{ E'_c , \ee'_k \}$, we get $\dt g'_{kc} = \sum\nolimits_{i=n+1}^{n+p} \sum\nolimits_{a=1}^{n} \dt (M_{ik} M_{ac} g_{ai})$ and $\dt g'_{kl} = \sum\nolimits_{i=n+1}^{n+p} \sum\nolimits_{j=n+1}^{n+p} \dt ( M_{ik} M_{jl} g_{ij})$, so from the uniqueness of solution of ODE system, we have $g'_{kc}(t) = \sum\nolimits_{i=n+1}^{n+p} \sum\nolimits_{a=1}^{n} M_{ik} M_{ac} g_{ai}(t)$ and $g'_{kl}(t) = \sum\nolimits_{i=n+1}^{n+p} \sum\nolimits_{j=n+1}^{n+p} M_{ik} M_{jl} g_{ij}(t)$ for all $t$, for which the solution exists.
Hence, solution of \eqref{bsharp1}-\eqref{bsharp3} defines a tensor field $g_t$ on $U$. 

Since $g_0$ is a metric tensor, $g_t$ will also be a metric tensor for small enough values of $t$. 
Equation \eqref{bsharp2} 
defines the coefficients $\{b_{ai}\}$ of $B_t = \dt g_t$. 
From \eqref{bsharp1} we obtain that $B_t^\sharp E_a$ are $g_t$-horizontal, hence
\begin{equation} \label{Bsharptildelambda}
B^\sharp_t E_a = \sum\nolimits_{i=n+1}^{n+p} \left(  {\tilde \lambda}_{ai}(t) \ee_i - \sum\nolimits_{c=1}^n g_{ic}(t) {\tilde \lambda}_{ai}(t) E_c \right)
\end{equation}
for some $\{ \tilde \lambda_{ai} : U \times \mathbb{R} \rightarrow \mathbb{R} \}$. But obtaining $b_{ai} = B_t(E_a, \ee_i) = g_t(B^\sharp_t E_a , \ee_i)$ from \eqref{Bsharptildelambda} and comparing with \eqref{bsharp2}, we get for all $a \in \{1, \ldots ,n \}$ and all $i \in \{n+1 , \ldots, n+p\}$
\begin{equation} \label{lambdaunique}
\sum\nolimits_{j=n+1}^{n+p} ( {\tilde \lambda}_{aj}(t) - \lambda_{aj} ) ( g_{ij}(t) - \sum\nolimits_{b=1}^n g_{jb}(t) g_{bi}(t) ) = 0 .
\end{equation}
There exists $\epsilon' >0$ such that for all $t \in (-\epsilon', \epsilon')$
we have 
$\det \mathcal{G}(t) \neq 0$ where $\mathcal{G}(t)$ is the matrix with entries 
$\mathcal{G}(t)_{ij} = g_{ij}(t) - \sum\nolimits_{b=1}^n g_{jb}(t) g_{bi}(t)$, because $\mathcal{G}(0)$ is identity matrix, so
from \eqref{lambdaunique} it follows that ${\tilde \lambda}_{ai} = \lambda_{ai}$. Hence, 
\begin{equation} \label{bsharplambda} 
B^\sharp_t E_a = \sum\nolimits_{i=n+1}^{n+p}  \lambda_{ai} \left(  \ee_i - \sum\nolimits_{c=1}^n g_{ic}(t) E_c \right)
\end{equation} 
holds and, by \eqref{deflambdaai}, yields \eqref{bsharpinv}.

From \eqref{bsharp1} it follows that $g_t \in \Riem(U, \mV , g_0)$, i.e.,
\begin{equation} \label{BVVzero}
B_t(U,V) = 0 \quad {\rm for \; all } \; U,V \in \mathfrak{X}_{\mV} .
\end{equation} 
From \eqref{bsharp3} it follows that the variation $g_t$ preserves the Riemannian submersion $\pi : (U , g_0) \rightarrow (\pi(U), g_N)$. Indeed, using $P_{\mV}^t Z + P_{\mH}^t Z = Z$ and $\pi_* P_{\mV}^t  Z =0$ for all $Z \in \mathfrak{X}_M$ and all $t \in (-\epsilon, \epsilon)$, we obtain that 
$\pi : (M, g_t) \rightarrow (N, g_N)$ is a Riemannian submersion if and only if for all $X,Y \in \mathfrak{X}_M$ we have
\begin{eqnarray*}
g_t(P_{\mH}^t X, P_{\mH}^t Y) &=& g_N (\pi_* P_{\mH}^t  X, \pi_*  P_{\mH}^t  Y ) = g_N (\pi_* X, \pi_* Y ) \\ 
&=& g_N(\pi_* P_{\mH}^0  X, \pi_* P_{\mH}^0  Y) = g_0(P_{\mH}^0 X, P_{\mH}^0 Y). 
\end{eqnarray*}
The above holds for $\{ g_t, t \in (-\epsilon, \epsilon) \} \subset \Riem(M, \mV , g_0)$ if and only if for all $x \in M$, $X,Y \in T_xM$ and every $g_0$-orthonormal -- and hence $g_t$-orthonormal for all $t \in  (-\epsilon , \epsilon)$, because $\{ g_t, t \in (-\epsilon, \epsilon) \} \subset \Riem(M, \mV , g_0)$ -- basis 
$\{E_1 ,  \ldots , E_n\}$ of ${\mV}_x$ we have
\begin{eqnarray} \label{BXY}
0 &=&  \dt g_t( P_{\mH}^t  X, P_{\mH}^t Y ) = \dt \big( g_t (X - \sum\nolimits_{a=1}^n g_t(X, E_a)E_a \, , Y- \sum\nolimits_{a=1}^n g_t(Y, E_b)E_b  ) \big) \nonumber \\
&=& B_t(X,Y) - \dt \sum\nolimits_{a=1}^n g_t(X,E_a) g_t (E_a,Y) - \dt \sum\nolimits_{a=1}^n g_t(Y, E_b) g_t (E_b, X) \nonumber \\
&& + \dt \sum\nolimits_{a,b=1}^n g_t(X, E_a) g_t(E_b , Y) g_t(E_a, E_b)  \nonumber \\
&=& B_t(X,Y) - \sum\nolimits_{a=1}^n B_t(X,E_a) g_t(Y,E_a) - \sum\nolimits_{a=1}^n B_t(Y,E_a) g_t(X,E_a),
\end{eqnarray}
where we used $g_t(E_a, E_b) = g_0(E_a, E_b) = \delta_{ab}$. 
Equations \eqref{BVVzero} and \eqref{BXY} in the local frame $\{E_a, \ee_i\}$ are equivalent to the system
\begin{eqnarray*}
&& g_{ab} = \delta_{ab} , \\ \label{bsharp2a}
&& \dt g_{ai} = b_{ai}(t), \\
\label{bsharp3a}
&& \dt g_{ij} = \sum\nolimits_{a=1}^n \left( b_{ai}(t) g_{aj}(t) +  b_{aj}(t) g_{ai}(t) \right),
\end{eqnarray*}
which coincides with \eqref{bsharp1}-\eqref{bsharp3} with $b_{ai}$ given by the right hand side of \eqref{bsharp2}.

Now we prove that there exists $\epsilon>0$ such that at every $x \in M$ the solution of \eqref{bsharp1}-\eqref{bsharp3} exists for all $t \in (-\epsilon, \epsilon)$ for every adapted orthonormal frame at $x$. Indeed, \eqref{bsharp1}-\eqref{bsharp3} are equations for coefficients of $g_t$ in a $g_0$-orthonormal frame, so the initial condition for these equations for all $x \in M$ is the same: $g_{ai}(0)=0$ and $g_{ij}(0)=\delta_{ij}$ for $a \in \{1, \ldots, n\}$ and $i,j \in \{n+1, \ldots, n+p\}$. Since $\{ V_i, W_i \}$ are bounded, from \eqref{deflambdaai} it follows that there exists a compact set $K \subset \mathbb{R}^{np}$, such that $\{ \lambda_{ai} \} \in K$ for all $x \in M$,
and for all local $g_0$-orthonormal frames $\{E_a, \ee_i\}$. 
For every $\{ \lambda_{ai} \} \in K$ there exist $\epsilon' >0$ and an open neighbourhood $\mN \subset \mathbb{R}^{np}$ such that for all $\{ \lambda_{ai} \} \in \mN$ the solution of \eqref{bsharp1}-\eqref{bsharp3} with initial condition $g_{ai}(0)=0 , g_{ij}(0)=\delta_{ij}$ exists and, since the solutions depend smoothly on parameters $\{ \lambda_{ai} \}$ of the equations, 
$\det ( g_{ij}(t) - \sum\nolimits_{b=1}^n g_{jb}(t) g_{bi}(t) ) \neq 0$ for all $t \in (-\epsilon', \epsilon')$. Covering $K$ with these neighbourhoods and finding the smallest $\epsilon'$ for some finite subcover, we obtain an interval $(-\epsilon, \epsilon)$ on which $g_t$ is defined for all $x \in M$ as the metric whose components in a local frame $\{E_a , \ee_i \}$ are the solution of \eqref{bsharp1}-\eqref{bsharp3} with $\{ \lambda_{ai} \}$ defined by \eqref{deflambdaai}.

The condition for keeping the fibers of $\pi$ totally geodesic on $(M,g_t)$ follows from the next two formulas, \eqref{dtnablaXYZ} and \eqref{dtPhdtnablaXYforVi}, which we will prove as lemmas further below. For a variation $g_t \in \Riem(M, \mV, g_0)$ preserving the Riemannian submersion $\pi$ we have for all $X,Y \in \mathfrak{X}_{\mV}$ and $Z \in \mathfrak{X}_{\mH(t)}$
\begin{eqnarray} \label{dtnablaXYZ}
2 g_t (\dt \nabla^t_X Y , Z) 
&=& g_t([X, B_t^\sharp Y] , Z) 
 + g_t( [Y, B_t^\sharp X ], Z) \nonumber\\
&& - B_t( Z, P_{\mV}^0 ( \nabla^0_{X} {Y} + \nabla^0_{Y} {X}  ) )
\end{eqnarray}
and for a variation $g_t \in \Riem(M, \mV, g_0)$ satisfying \eqref{bsharpinv}, we have for all $X,Y \in \mathfrak{X}_{\mV}$
\begin{eqnarray} \label{dtPhdtnablaXYforVi}
&& 2 P_{\mH}^t \dt \nabla^t_{X} Y = \sum\nolimits_{i=n+1}^{n+p} \big( g_0(\nabla^0_X V_i , Y) + g_0(\nabla^0_Y V_i , X) \big) P_{\mH}^t  \pi^* W_i . 
\end{eqnarray}
Hence, if all $V_i$ are vertical fields, whose restrictions to every fiber $\mF_x$ are Killing fields on $(\mF_x, g_0 \vert_{\mF_x})$, we obtain $P_{\mH}^t \dt \nabla^t_{X} Y =0$ for all $X,Y \in \mathfrak{X}_{\mV}$. 

From $g_t \in \Riem(M, \mV, g_0)$ and Koszul formula for $X,Y,Z \in \mathfrak{X}_{\mV}$ we obtain $\dt P_\mV^t \nabla^t_X Y =0$ (see \eqref{dtnablaXYZproof4} below for the proof of this statement), and hence
\begin{eqnarray*}
\dt g_t( P_{\mH}^t \nabla^t_{X} Y ,  P_{\mH}^t \nabla^t_{X} Y ) &=& B_t( P_{\mH}^t \nabla^t_{X} Y , P_{\mH}^t \nabla^t_{X} Y ) + 2g_t( \dt P_{\mH}^t \nabla^t_{X} Y , P_{\mH}^t \nabla^t_{X} Y ) \\
&=& 2g_t( P_{\mH}^t \dt \nabla^t_{X} Y , P_{\mH}^t \nabla^t_{X} Y ) - 2g_t( P_{\mH}^t \dt P_{\mV}^t \nabla^t_{X} Y , P_{\mH}^t \nabla^t_{X} Y ) \\
&=& 2g_t( P_{\mH}^t \dt \nabla^t_{X} Y , P_{\mH}^t \nabla^t_{X} Y ) .
\end{eqnarray*}
Hence, if $P_{\mH}^t \dt \nabla^t_{X} Y =0$ and $P_{\mH}^0 \nabla^0_{X} Y=0$, then $P_{\mH}^t \nabla^t_{X} Y=0$ and so fibers remain totally geodesic. 
\end{proof}

\begin{lemma}
For a variation $g_t \in \Riem(M, \mV, g_0)$ preserving the Riemannian submersion $\pi$ 
for all $X,Y \in \mathfrak{X}_{\mV}$ and $Z \in \mathfrak{X}_{\mH(t)}$ \eqref{dtnablaXYZ} holds.
\end{lemma} 
\begin{proof} 
From the Koszul formula and $B_t(X,Y)=0$, due to \eqref{BVVzero}, we obtain
\begin{eqnarray} \label{dtnablaXYZproof1}
&& 2 g_t (\dt \nabla^t_X Y , Z) 
= 2\dt (g_t (\nabla^t_X Y , Z))  - 2B_t(\nabla^t_X Y , Z) \nonumber \\
&&= X (B_t(Y,Z)) + Y (B_t(X,Z)) - Z (B_t(X,Y)) \nonumber\\
&&\quad + B_t ( [X,Y] , Z ) + B_t([Z,X],Y)  + B_t([Z,Y],X) - 2B_t( \nabla^t_X Y , Z ) \nonumber\\
&&= X (g_t(B_t^\sharp Y,Z)) + Y (g_t(B_t^\sharp X,Z)) + B_t ( \nabla^t_X Y - \nabla^t_Y X , Z ) \nonumber\\
&&\quad + B_t ( \nabla^t_Z X - \nabla^t_X Z , Y ) + B_t ( \nabla^t_Z Y - \nabla^t_Y Z , X ) - 2B_t( \nabla^t_X Y , Z )\nonumber \\
&&= g_t( \nabla^t_X B_t^\sharp Y,Z ) + g_t( \nabla^t_X Z , B_t^\sharp Y ) + g_t( \nabla^t_Y B_t^\sharp X,Z )  + g_t( \nabla^t_Y Z , B_t^\sharp X ) \nonumber\\
&&\quad -B_t( \nabla^t_X Y + \nabla^t_Y X , Z)  + g_t ( \nabla^t_Z X - \nabla^t_X Z , B_t^\sharp Y ) + g_t ( \nabla^t_Z Y - \nabla^t_Y Z , B_t^\sharp X ) \nonumber\\
&&= g_t( [ X ,  B_t^\sharp Y ] ,Z ) + g_t( \nabla^t_{B_t^\sharp Y} X , Z )
+ g_t( [ Y ,  B_t^\sharp X ] ,Z ) + g_t( \nabla^t_{B_t^\sharp X} Y , Z ) \nonumber\\
&&\quad  -g_t( \nabla^t_X Y + \nabla^t_Y X , B_t^\sharp Z)  + g_t ( \nabla^t_Z X , B_t^\sharp Y ) + g_t ( \nabla^t_Z Y , B_t^\sharp X ).
\end{eqnarray}
Using $B_t^\sharp X , B_t^\sharp Y \in \mathfrak{X}_{\mH(t)}$ and $\pi : (M,g_t) \rightarrow (N,g_N)$ being a Riemannian submersion for all $t \in (-\epsilon, \epsilon)$, we obtain 
\begin{eqnarray} \label{dtnablaXYZproof2}
&& g_t( \nabla^t_{B_t^\sharp Y} X , Z ) + g_t( \nabla^t_{B_t^\sharp X} Y , Z ) + g_t ( \nabla^t_Z X , B_t^\sharp Y ) + g_t ( \nabla^t_Z Y , B_t^\sharp X ) \nonumber\\
&& = - g_t(  X , \nabla^t_{B_t^\sharp Y} Z ) - g_t(  Y , \nabla^t_{B_t^\sharp X} Z ) - g_t (  X , \nabla^t_Z B_t^\sharp Y ) - g_t (  Y , \nabla^t_Z B_t^\sharp X )  \nonumber\\
&&= - g_t(  X , \nabla^t_{B_t^\sharp Y} Z + \nabla^t_Z B_t^\sharp Y ) - g_t(  Y , \nabla^t_{B_t^\sharp X} Z + \nabla^t_Z B_t^\sharp X ) =0.
\end{eqnarray}
We have $Z \in \mathfrak{X}_{\mH(t)}$ and hence, by \eqref{BXY}, $B_t^\sharp Z \in \mathfrak{\mV}$. We obtain
\begin{eqnarray} \label{dtnablaXYZproof3}
g_t( \nabla^t_X Y + \nabla^t_Y X , B_t^\sharp Z) &=& g_t( P_{\mV}^t( \nabla^t_X Y + \nabla^t_Y X ) , B_t^\sharp Z) \nonumber\\
&=& B_t( P_{\mV}^t( \nabla^t_X Y + \nabla^t_Y X ) , Z) .
\end{eqnarray}
Using a local vertical orthonormal frame $\{E_a\}$,  $g_t \in \Riem(M,\mV,g_0)$ and the Koszul formula it follows that
\begin{eqnarray} \label{dtnablaXYZproof4}
&& P_{\mV}^t( \nabla^t_X Y + \nabla^t_Y X ) = 
\sum\nolimits_{a=1}^n g_t( \nabla^t_X Y + \nabla^t_Y X , E_a ) E_a = \nonumber\\
&&\sum\nolimits_{a=1}^n g_0( \nabla^0_X Y + \nabla^0_Y X , E_a ) E_a =
P_{\mV}^0( \nabla^0_X Y + \nabla^0_Y X ).
\end{eqnarray}
From \eqref{dtnablaXYZproof1}-\eqref{dtnablaXYZproof4} we obtain \eqref{dtnablaXYZ}.
\end{proof}

\begin{lemma}
For a variation $g_t \in \Riem(M, \mV, g_0)$ satisfying \eqref{bsharpinv}, 
for all $X,Y \in \mathfrak{X}_{\mV}$ \eqref{dtPhdtnablaXYforVi} holds.
\end{lemma}
\begin{proof} 
For all $X \in \mathfrak{X}_{\mV}$ and all $i \in \{n+1,\ldots,n+p\}$, since $\pi^* W_i$ is projectable, we have $P_{\mH}^t [X,  \pi^* W_i ]=0$ and since $\mV$ is integrable, we have $P_{\mH}^t [X, P_{\mV}^t  \pi^* W_i ]=0$, so
\begin{equation} \label{lieXP2piW}
P_{\mH}^t  [X, P_{\mH}^t  \pi^* W_i] = P_{\mH}^t [X,  \pi^* W_i ] - P_{\mH}^t [X, P_{\mV}^t  \pi^* W_i ] = 0.
\end{equation}
From \eqref{dtnablaXYZ} we obtain 
\[
2 P_{\mH}^t \dt \nabla^t_{X} Y = P_{\mH}^t \left( [X, B_t^\sharp Y] + [Y, B_t^\sharp X ] - B_t^\sharp P_{\mV}^0 ( \nabla^0_{X} {Y} + \nabla^0_{Y} {X}  ) \right).
\]
Using \eqref{bsharpinv} in the above, by \eqref{lieXP2piW} we obtain 
for all $X,Y \in \mathfrak{X}_{\mV}$
\begin{eqnarray} 
&& 2 P_{\mH}^t \dt \nabla^t_{X} Y = P_{\mH}^t  \big( 
[X,  \sum\nolimits_{i=n+1}^{n+p} g_0(V_i , Y) P_{\mH}^t  \pi^* W_i]  \nonumber \\
&&\quad + [Y,  \sum\nolimits_{i=n+1}^{n+p} g_0(V_i , X) P_{\mH}^t  \pi^* W_i ] - B_t^\sharp P_{\mV}^0 (\nabla^0_{X}Y +  \nabla^0_{Y}X ) \big) \nonumber\\
&&= \sum\nolimits_{i=n+1}^{n+p} X(  g_0(V_i , Y) ) P_{\mH}^t  \pi^* W_i  + \sum\nolimits_{i=n+1}^{n+p} Y(  g_0(V_i , X) ) P_{\mH}^t  \pi^* W_i \nonumber\\
&&\quad - \sum\nolimits_{i=n+1}^{n+p} g_0(V_i , \nabla^0_{X}Y +  \nabla^0_{Y}X) P_{\mH}^t  \pi^* W_i \nonumber\\
&&= \sum\nolimits_{i=n+1}^{n+p} \big( g_0(\nabla^0_X V_i , Y) + g_0(\nabla^0_Y V_i , X) \big) P_{\mH}^t  \pi^* W_i . \nonumber 
\end{eqnarray}
\end{proof} 

\begin{lemma} 
Let $\{ g_t, t \in (-\epsilon, \epsilon) \} \subset \Riem(M, \mV, g_0)$ be a 
variation such that $\pi : (M, g_t) \rightarrow (N, g_N)$ is a Riemannian submersion for all $t \in (-\epsilon, \epsilon)$ and \eqref{bsharpinv} holds, where $\{ W_{n+1}, \ldots , W_{n+p} \}$ are $g_N$-orthonormal. Then for all $t \in (-\epsilon, \epsilon)$ the vector fields $\{ P_{\mH}^t \pi^* W_{n+1} , \ldots , P_{\mH}^t \pi^* W_{n+p} \}$ are $g_t$-orthonormal and $g_t$-horizontal, 
and
\begin{equation} \label{dtPHWi}
\dt P_{\mH}^t \pi^* W_i = -V_i 
\end{equation}
for all $i \in \{ n+1 , \ldots, n+p\}$.
\end{lemma}
\begin{proof}
The first statement follows from 
\begin{equation} \label{gNWiWj}
g_t(P_{\mH}^t \pi^* W_i , P_{\mH}^t \pi^* W_j ) = g_N(W_i, W_j) = \delta_{ij}.
\end{equation}
To prove the last one, we have for all $X \in \mathfrak{X}_M$:
\begin{eqnarray} \label{dtPVX}
\dt P_{\mV}^t X &=& \dt \sum\nolimits_{a=1}^{n}  g_t(X , E_a) E_a = \sum\nolimits_{a=1}^{n} \big( B_t(X , E_a) E_a + g_t(\dt X , E_a) E_a \big) \nonumber \\
&=& \sum\nolimits_{a=1}^{n} g_t( B_t^\sharp E_a , X) E_a + P_{\mV}^t (\dt X) \nonumber\\
&=& \sum\nolimits_{a=1}^{n} \sum\nolimits_{i=n+1}^{n+p} g_0(V_i , E_a) g_t( P_{\mH}^t \pi^* W_i , X) E_a + P_{\mV}^t (\dt X)  \nonumber\\
&=& \sum\nolimits_{i=n+1}^{n+p} g_t( P_{\mH}^t \pi^* W_i , X) V_i + P_{\mV}^t (\dt X) .
\end{eqnarray}
In particular,
\begin{eqnarray*}
\dt P_{\mH}^t \pi^* W_i &=& - \dt P_{\mV}^t \pi^* W_i
= -\sum\nolimits_{j=n+1}^{n+p} g_t( P_{\mH}^t \pi^* W_i , \pi^* W_j) V_j \\
&=& - \sum\nolimits_{i=n+1}^{n+p} g_N( W_i , W_j)  V_j = -V_i.
\end{eqnarray*}
\end{proof}

\begin{remark} \label{remVit}
We considered $V_i$ independent of $t$ in Theorem \ref{corbsharpglobal}, to ensure existence of solutions of the system \eqref{bsharp1}-\eqref{bsharp3} in a fixed interval $(-\epsilon, \epsilon)$ at all points of $M$. However, if we assume that there exists a variation $\{ g_t, t \in (-\epsilon, \epsilon) \} \subset \Riem(M, \mV , g_0)$ that preserves the Riemannian submersion $\pi$, then it satisfies \eqref{bsharpinv}, with some $t$-dependent vertical vector fields $V_i(t)$. Moreover, fibers of $\pi$ remain totally geodesic for all $\{ g_t, t \in (-\epsilon, \epsilon) \}$ if and only if restrictions of all $V_i(t)$ to every fiber $\mF_x$ are Killing fields on $(\mF_x, g_0 \vert_{\mF_x})$, for all $t \in (-\epsilon, \epsilon)$. Therefore, some further formulas will be obtained in the more general setting, where \eqref{bsharpinv} is satisfied with $t$-dependent vector fields $V_i(t)$.
\end{remark}


\section{Variational formulas for curvatures} \label{secvariationsofcurvature}

Let $\{ g_t, t \in (-\epsilon, \epsilon) \} \subset \Riem(M, \mV, g_0)$ be a 
variation such that $\pi : (M, g_t) \rightarrow (N, g_N)$ is a Riemannian submersion for all $t \in (-\epsilon, \epsilon)$, with totally geodesic fibers.
For $X,Y \in \mH(t)$ and $U \in \mV$, from \eqref{definitionAtensor} we obtain $\mA^t_X Y = \frac{1}{2}P_{\mV}^t [X,Y]$ and $g_t( \mA^t_X U , Y) = - g_t( \mA^t_X Y , U  )$, we also have the following curvature formulas
\cite{O'Neill}
\begin{eqnarray} \label{secAXY}
&& \sec_{(M,g_t)}(X,Y) = \sec_{(N, g_N)}(\pi_* X, \pi_* Y) - 3g_t( \mA^t_X Y ,  \mA^t_X Y) ,
\\
\label{secAXU}
&& \sec_{(M,g_t)}(X,U) = g_t( \mA^t_X U ,  \mA^t_X U) .
\end{eqnarray}

\begin{lemma} \label{lemmaAWiWjU}
Let $\{ g_t, t \in (-\epsilon, \epsilon) \} \subset \Riem(M, \mV, g_0)$ be a 
variation such that $\pi : (M, g_t) \rightarrow (N, g_N)$ is a Riemannian submersion for all $t \in (-\epsilon, \epsilon)$ and \eqref{bsharpinv} holds, where 
$\{W_{n+1}, \ldots , W_{n+p}\}$ are $g_N$-orthonormal and 
$\{V_{n+1}, \ldots , V_{n+p}\}$ are bounded vertical vector fields whose restrictions to every fiber $\mF_x$ are Killing fields on $(\mF_x ,g_0 \vert_{\mF_x})$. Let the fibers of $\pi : (M, g_0) \rightarrow (N, g_N)$ be totally geodesic. 
Then 
\begin{eqnarray} \label{dtAWiWj}
&& 2\dt \mA_{P_{\mH}^t W_i}  P_{\mH}^t W_j  = \sum\nolimits_{k=n+1}^{n+p} g_N([  W_i ,  W_j  ] , W_k ) V_k \nonumber\\
&&\quad - P_{\mV}^t \sum\nolimits_{a=1}^n ([ V_i ,  P_{\mH}^t W_j  ] + [ P_{\mH}^t W_i ,  V_j  ] ) 
\end{eqnarray}
and for all 
$U, U_1, U_2 \in \mathfrak{X}_{\mV}$ we have
\begin{eqnarray} \label{dtAWiU}
&& 2\dt \mA_{P_{\mH}^t W_i} U  =  -\sum\nolimits_{j,k=n+1}^{n+p} g_N([  W_i ,  W_j  ] , W_k ) g_0( V_k , U)  P_{\mH}^t W_j \nonumber\\
&&\quad + \sum\nolimits_{j=n+1}^{n+p} g_t( [ V_i ,  P_{\mH}^t W_j  ] + [ P_{\mH}^t W_i ,  V_j  ] , U )  P_{\mH}^t W_j \nonumber\\
&&\quad + \sum\nolimits_{j=n+1}^{n+p} g_t( [ P_{\mH}^t W_i , P_{\mH}^t W_j ] , U) V_j ,
\end{eqnarray}
and
\begin{eqnarray} \label{dtAWiU1WjU2}
&& 4 \dt g_t( \mA^t_{P_{\mH}^t W_i } U_1 , \mA^t_{P_{\mH}^t W_l} U_2 ) \nonumber\\
&&= \sum\nolimits_{j=n+1}^{n+p} g_t( P_{\mV}^t [P_{\mH}^t W_l, P_{\mH}^t W_j] ,U_2 ) \cdot \big( 
\sum\nolimits_{k=n+1}^{n+p} g_N([W_i, W_j] , W_k)g_0( V_k, U_1 ) \nonumber\\
&&\quad - g_t( [ V_i , P_{\mH}^t W_j ] + [ P_{\mH}^t W_i, V_j ] , U_1 ) \big) \nonumber\\
&&\quad + \sum\nolimits_{j=n+1}^{n+p} g_t( P_{\mV}^t [P_{\mH}^t W_i, P_{\mH}^t W_j] ,U_1) \cdot \big( 
\sum\nolimits_{k=n+1}^{n+p} g_N([W_l, W_j] , W_k)g_0( V_l, U_2 ) \nonumber\\
&&\quad - g_t( [ V_l , P_{\mH}^t W_j ] + [ P_{\mH}^t W_l, V_j ] , U_2 ) \big) .
\end{eqnarray}
\end{lemma}
\begin{proof} 
For all $i \in \{n+1, \ldots, n+p\}$, from \eqref{lieXP2piW} it follows that $P_{\mH}^t \pi^* W_i$ is 
projectable, so $P_{\mH}^t \pi^* W_i$ is the $g_t$-horizontal lift of $W_i$. By the Jacobi identity, the Lie bracket of projectable fields is also projectable. Hence, 
\begin{eqnarray} \label{liebracketWiWj}
&&g_t( [ P_{\mH}^t \pi^* W_i , P_{\mH}^t \pi^* W_j  ] , P_{\mH}^t \pi^* W_k ) 
= g_N( \pi_* [ P_{\mH}^t \pi^* W_i , P_{\mH}^t \pi^* W_j  ] , W_k ) \nonumber \\
&&=g_N( [ \pi_* P_{\mH}^t \pi^* W_i , \pi_* P_{\mH}^t \pi^* W_j  ] , W_k ) = g_N( [ W_i , W_j ] , W_k ) .
\end{eqnarray}

For all $U \in \mathfrak{X}_{\mV}$ we have, by 
\eqref{BVVzero}, \eqref{dtPVX} and \eqref{liebracketWiWj}
\begin{eqnarray} \label{dtAWiWjU}
&& 2\dt g_t( \mA^t_{ P_{\mH}^t \pi^* W_i }  P_{\mH}^t \pi^* W_j , U ) =
\dt g_t (  P_{\mV}^t [ P_{\mH}^t \pi^* W_i , P_{\mH}^t \pi^* W_j  ] , U ) \nonumber \\
&&= g_t ( \dt P_{\mV}^t [ P_{\mH}^t \pi^* W_i , P_{\mH}^t \pi^* W_j  ] , U ) \nonumber \\
&&= \sum\nolimits_{k=n+1}^{n+p} g_t( [ P_{\mH}^t \pi^* W_i , P_{\mH}^t \pi^* W_j  ] , P_{\mH}^t \pi^* W_k ) g_t( V_k , U ) \nonumber \\
&& \quad - g_t( [V_i , P_{\mH}^t \pi^* W_j  ] + [ P_{\mH}^t \pi^* W_i , V_j ] , U ) \nonumber \\ 
&&= \sum\nolimits_{k=n+1}^{n+p} g_N( [ W_i , W_j ] , W_k ) g_0( V_k , U ) \nonumber \\
&& \quad - g_t( [V_i , P_{\mH}^t \pi^* W_j  ] + [ P_{\mH}^t \pi^* W_i , V_j ] , U ) .
\end{eqnarray}
Using \eqref{liebracketWiWj} and \eqref{dtAWiWjU} we obtain \eqref{dtAWiWj} as
\begin{eqnarray*}
&& 2\dt \mA_{P_{\mH}^t W_i}  P_{\mH}^t W_j  = \dt P_{\mV}^t [ P_{\mH}^t W_i ,  P_{\mH}^t W_j  ] \\
&& = \dt \sum\nolimits_{a=1}^n g_t( [ P_{\mH}^t W_i ,  P_{\mH}^t W_j  ] , E_a ) E_a \\
&&= \sum\nolimits_{a=1}^n B_t([ P_{\mH}^t W_i ,  P_{\mH}^t W_j  ] , E_a ) E_a \\
&&\quad - \sum\nolimits_{a=1}^n g_t([ V_i ,  P_{\mH}^t W_j  ] + [ P_{\mH}^t W_i ,  V_j  ] , E_a ) E_a \\
&&= \sum\nolimits_{a=1}^n ([ P_{\mH}^t W_i ,  P_{\mH}^t W_j  ] , B_t^\sharp E_a ) E_a \\
&&\quad - \sum\nolimits_{a=1}^n g_t([ V_i ,  P_{\mH}^t W_j  ] + [ P_{\mH}^t W_i ,  V_j  ] , E_a ) E_a \\
&&= \sum\nolimits_{a=1}^n \sum\nolimits_{k=n+1}^{n+p} ([ P_{\mH}^t W_i ,  P_{\mH}^t W_j  ] , P_{\mH}^t W_k )g_0(V_k, E_a) E_a\\
&&\quad - \sum\nolimits_{a=1}^n g_t([ V_i ,  P_{\mH}^t W_j  ] + [ P_{\mH}^t W_i ,  V_j  ] , E_a ) E_a \\
&&= \sum\nolimits_{k=n+1}^{n+p} g_N([  W_i ,  W_j  ] , W_k ) V_k \\
&&\quad - P_{\mV}^t \sum\nolimits_{a=1}^n ([ V_i ,  P_{\mH}^t W_j  ] + [ P_{\mH}^t W_i ,  V_j  ] ) .
\end{eqnarray*}
Similarly, \eqref{dtAWiU} follows from
\begin{eqnarray*}
&& 2\dt \mA_{P_{\mH}^t W_i} U  = 2\dt \sum\nolimits_{j=n+1}^{n+p} g_t( \mA_{P_{\mH}^t W_i} U , P_{\mH}^t W_j)  P_{\mH}^t W_j \\
&&= -2 \dt \sum\nolimits_{j=n+1}^{n+p} g_t( \mA_{P_{\mH}^t W_i} P_{\mH}^t W_j , U)  P_{\mH}^t W_j \\
&&= -\sum\nolimits_{j,k=n+1}^{n+p} g_N([  W_i ,  W_j  ] , W_k ) g_0( V_k , U)  P_{\mH}^t W_j \\
&&\quad + \sum\nolimits_{j=n+1}^{n+p} g_t( [ V_i ,  P_{\mH}^t W_j  ] + [ P_{\mH}^t W_i ,  V_j  ] , U )  P_{\mH}^t W_j \\
&&\quad + 2\sum\nolimits_{j=n+1}^{n+p} g_t( \mA_{P_{\mH}^t W_i} P_{\mH}^t W_j , U) V_j \\
&&= -\sum\nolimits_{j,k=n+1}^{n+p} g_N([  W_i ,  W_j  ] , W_k ) g_0( V_k , U)  P_{\mH}^t W_j \\
&&\quad + \sum\nolimits_{j=n+1}^{n+p} g_t( [ V_i ,  P_{\mH}^t W_j  ] + [ P_{\mH}^t W_i ,  V_j  ] , U )  P_{\mH}^t W_j \\
&&\quad + \sum\nolimits_{j=n+1}^{n+p} g_t( [ P_{\mH}^t W_i , P_{\mH}^t W_j ] , U) V_j .
\end{eqnarray*}
and \eqref{dtAWiU1WjU2} follows from 
\begin{eqnarray*}
&& 4 \dt g_t( \mA^t_{P_{\mH}^t W_i } U_1 , \mA^t_{P_{\mH}^t W_l} U_2 ) \\
&&  = \dt \sum\nolimits_{j=n+1}^{n+p} g_t( P_{\mV}^t [P_{\mH}^t W_i, P_{\mH}^t W_j] , U_1 ) g_t( P_{\mV}^t [P_{\mH}^t W_l, P_{\mH}^t W_j] ,U_2 ) ,
\end{eqnarray*}
\eqref{dtAWiWjU} and \eqref{BVVzero}.
\end{proof}

\begin{theorem} \label{thd2tsecXUgeq0}
Let $\{ g_t, t \in (-\epsilon, \epsilon) \} \subset \Riem(M, \mV, g_0)$ be a 
variation such that $\pi : (M, g_t) \rightarrow (N, g_N)$ is a Riemannian submersion for all $t \in (-\epsilon, \epsilon)$ and \eqref{bsharpinv} holds, where 
$\{W_{n+1}, \ldots , W_{n+p}\}$ are $g_N$-orthonormal and 
$\{V_{n+1}, \ldots , V_{n+p}\}$ are bounded vertical vector fields whose restrictions to every fiber $\mF_x$ are Killing fields on $(\mF_x ,g_0 \vert_{\mF_x})$. Let the fibers of $\pi : (M, g_0) \rightarrow (N, g_N)$ be totally geodesic. 
Then for all $U \in \mathfrak{X}_{\mV}$ we have
\begin{eqnarray} \label{dtsecXURSGF}
\dt \sec_M (P_{\mH}^t \pi^* W_i,U) &=& \frac{1}{2} \sum\nolimits_{j=n+1}^{n+p} g_t( [P_{\mH}^t \pi^* W_i, P_{\mH}^t \pi^* W_j], U) \cdot \nonumber\\
&& \cdot \big( 
\sum\nolimits_{k=n+1}^{n+p} g_N( [ W_i , W_j ] , W_k ) g_0( V_k , U ) \nonumber\\
&&  - g_t( [V_i , P_{\mH}^t \pi^* W_j  ] + [ P_{\mH}^t \pi^* W_i , V_j ] , U ) 
\big) .
\end{eqnarray}
and
\begin{eqnarray} \label{ddtsecXURSGFdtV}
&& \ddt \sec_M( P_{\mH}^t \pi^* W_i , U) \nonumber\\ 
&&=\frac{1}{2} 
\sum\nolimits_{j=n+1}^{n+p} 
\big( 
\sum\nolimits_{l=n+1}^{n+p} g_N( [ W_i , W_j ] , W_l ) g_0( V_l , U ) 
- g_t( [V_i , P_{\mH}^t \pi^* W_j  ] \nonumber\\
&& + [ P_{\mH}^t \pi^* W_i , V_j ] , U ) 
\big)^2 
+ \frac{1}{2} \sum\nolimits_{j=n+1}^{n+p} g_t( [P_{\mH}^t \pi^* W_i, P_{\mH}^t \pi^* W_j], U) \cdot \nonumber \\
&& \cdot \big( \sum\nolimits_{l=n+1}^{n+p} g_N( [ W_i , W_j ] , W_l ) g_0( \dt V_l , U ) - g_t( [\dt V_i , P_{\mH}^t \pi^* W_j  ] , U) \nonumber \\
&& - g_t( [ P_{\mH}^t \pi^* W_i , \dt V_j  ] , U) + 2 g_t([V_i, V_j] ,U) \big) .
\end{eqnarray}
\end{theorem}
\begin{proof}
We note that $\mA^t_{P_{\mH}^t W_i} P_{\mH}^t W_j \in \mV$ and $\mA^t_{P_{\mH}^t W_i} U \in \mH(t)$, so from \eqref{BVVzero} and \eqref{BXY} it follows that $B_t( \mA^t_{P_{\mH}^t W_i} P_{\mH}^t W_j , \mA^t_{P_{\mH}^t W_i} P_{\mH}^t W_j ) =0$ and $B_t( \mA^t_{P_{\mH}^t W_i} U , \mA^t_{P_{\mH}^t W_i} U ) =0$.
From the above, \eqref{secAXU} and \eqref{dtAWiU} 
we obtain \eqref{dtsecXURSGF}.

Using \eqref{lieXP2piW}, \eqref{BVVzero} and \eqref{dtPHWi} we obtain
\begin{eqnarray} \label{dtgtViPHWjU}
&& \dt g_t( [V_i , P_{\mH}^t \pi^* W_j  ] , U) \nonumber\\
&& = B_t ( [V_i , P_{\mH}^t \pi^* W_j  ] , U ) + g_t([\dt V_i , P_{\mH}^t \pi^* W_j  ] , U ) -  g_t( [V_i , V_j  ] , U) \nonumber\\
&& = g_t([\dt V_i , P_{\mH}^t \pi^* W_j  ] , U ) -  g_t( [V_i , V_j  ] , U) .
\end{eqnarray}
Using \eqref{dtgtViPHWjU} in \eqref{dtsecXURSGF}, we obtain \eqref{ddtsecXURSGFdtV}.
\end{proof}

\begin{theorem} \label{thd2tTgeq0}
Let $\{ g_t, t \in (-\epsilon, \epsilon) \} \subset \Riem(M, \mV, g_0)$ be a 
variation such that $\pi : (M, g_t) \rightarrow (N, g_N)$ is a Riemannian submersion for all $t \in (-\epsilon, \epsilon)$ and \eqref{bsharpinv} holds, where $\{ W_{n+1}, \ldots , W_{n+p} \}$ are $g_N$-orthonormal vector fields on $N$. Let $\sec_M(X,Y)$ be the sectional curvature of the plane field 
spanned by linearly independent $X,Y \in \mathfrak{X}_M$.
Then
\begin{eqnarray} \label{dtnormTWiWjsec}
&& \dt \sec_M( P_{\mH}^t \pi^* W_i , P_{\mH}^t \pi^* W_j ) \nonumber \\
&&= -\frac{3}{2} \sum\nolimits_{k=n+1}^{n+p} g_N( [ W_i , W_j ] , W_k ) g_t( V_k ,   P_{\mV}^t [ P_{\mH}^t \pi^* W_i , P_{\mH}^t \pi^* W_j  ] ) \nonumber\\
&& \quad + \frac{3}{2}  g_t( [V_i , P_{\mH}^t \pi^* W_j  ] + [ P_{\mH}^t \pi^* W_i , V_j ] , P_{\mV}^t [ P_{\mH}^t \pi^* W_i , P_{\mH}^t \pi^* W_j  ] ) .
\end{eqnarray}
and
\begin{eqnarray} \label{d2tnormTWiWjsec}
&& \ddt \sec_M( P_{\mH}^t \pi^* W_i , P_{\mH}^t \pi^* W_j ) \nonumber \\
&&= -\frac{3}{2} \sum\nolimits_{k=n+1}^{n+p} g_N( [ W_i , W_j ] , W_k ) g_t( \dt V_k ,  P_{\mV}^t [ P_{\mH}^t \pi^* W_i , P_{\mH}^t \pi^* W_j  ] ) \nonumber \\
&&\quad - \frac{3}{2} \sum\nolimits_{k,l=n+1}^{n+p} g_N( [ W_i , W_j ] , W_k )  g_N( [ W_i , W_j ] , W_l ) g_t( V_k , V_l ) \nonumber\\
&&\quad +3 \sum\nolimits_{k=n+1}^{n+p} g_N( [ W_i , W_j ] , W_k ) g_t( V_k , [ V_i , P_{\mH}^t \pi^* W_j  ] + [ P_{\mH}^t \pi^* W_i , V_j ] ) \nonumber\\
&&\quad + \frac{3}{2} g_t( [\dt V_i , P_{\mH}^t \pi^* W_j  ] + [ P_{\mH}^t \pi^* W_i , \dt V_j ] , P_{\mV}^t [ P_{\mH}^t \pi^* W_i , P_{\mH}^t \pi^* W_j  ] )\nonumber\\
&&\quad - 3 g_t( [ V_i , V_j  ]  , P_{\mV}^t [ P_{\mH}^t \pi^* W_i , P_{\mH}^t \pi^* W_j  ] ) \nonumber\\
&&\quad - \frac{3}{2} g_t( [ V_i , P_{\mH}^t \pi^* W_j  ] + [ P_{\mH}^t \pi^* W_i ,  V_j ] , [ V_i , P_{\mH}^t \pi^* W_j ] + [ P_{\mH}^t \pi^* W_i ,  V_j ] ) .
\end{eqnarray}
\end{theorem}
\begin{proof} 
The proof is similar to that of Theorem \ref{thd2tsecXUgeq0} and follows from \eqref{dtAWiWj}.
\end{proof} 

For variations $g_t$ preserving Riemannian submersion with totally geodesic fibers, if vector fields $V_i$ are independent of $t$, we can compute all non-vanishing derivatives of curvatures in vertical and horizontal directions.

\begin{theorem} \label{thconstVidtsec}
Let $\{ g_t, t \in (-\epsilon, \epsilon) \} \subset \Riem(M, \mV, g_0)$ be a 
variation such that $\pi : (M, g_t) \rightarrow (N, g_N)$ is a Riemannian submersion for all $t \in (-\epsilon, \epsilon)$ and \eqref{bsharpinv} holds, where 
$\{W_{n+1}, \ldots , W_{n+p}\}$ are $g_N$-orthonormal and 
$\{V_{n+1}, \ldots , V_{n+p}\}$ are bounded vertical vector fields whose restrictions to every fiber $\mF_x$ are Killing fields on $(\mF_x ,g_0 \vert_{\mF_x})$. Let the fibers of $\pi : (M, g_0) \rightarrow (N, g_N)$ be totally geodesic. Suppose that $\dt V_i = 0$ for all $i,j \in \{n+1, \ldots, n+p\}$. 
Then for all $U \in \mathfrak{X}_{\mV}$ we have
\begin{eqnarray} \label{ddtsecXURSGF}
&& \ddt \sec_M( P_{\mH}^t \pi^* W_i , U) \nonumber\\ 
&&=\frac{1}{2} 
\sum\nolimits_{j=n+1}^{n+p} 
\big( 
\sum\nolimits_{l=n+1}^{n+p} g_N( [ W_i , W_j ] , W_l ) g_0( V_l , U ) 
- g_t( [V_i , P_{\mH}^t \pi^* W_j  ] \nonumber\\
&& + [ P_{\mH}^t \pi^* W_i , V_j ] , U ) 
\big)^2 
+ \sum\nolimits_{j=n+1}^{n+p} g_t( [P_{\mH}^t \pi^* W_i, P_{\mH}^t \pi^* W_j], U) \cdot g_t([V_i, V_j] ,U) .  \nonumber\\
\end{eqnarray}
\begin{eqnarray} \label{d3tsecXURSGF}
&& \frac{\partial^3}{\partial t^3} \sec_M( P_{\mH}^t \pi^* W_i , U) \nonumber\\ 
&&= 
3 \sum\nolimits_{j=n+1}^{n+p} 
\big( 
\sum\nolimits_{l=n+1}^{n+p} g_N( [ W_i , W_j ] , W_l ) g_0( V_l , U ) 
- g_t( [V_i , P_{\mH}^t \pi^* W_j  ] \nonumber\\
&& + [ P_{\mH}^t \pi^* W_i , V_j ] , U ) 
\big) \cdot g_t( [V_i , V_j  ] , U ) ,
\end{eqnarray}
\begin{eqnarray} \label{d4tsecXURSGF}
\frac{\partial^4}{\partial t^4} \sec_M( P_{\mH}^t \pi^* W_i , U) 
= 
6 \sum\nolimits_{j=n+1}^{n+p} g_t( [V_i , V_j  ] , U )^2 
\end{eqnarray}
and $\frac{\partial^5}{\partial t^5} \sec_M( P_{\mH}^t \pi^* W_i , U) =0$.

Moreover,
\begin{eqnarray} \label{d2tnormTWiWjsecVconst}
&& \ddt \sec_M( P_{\mH}^t \pi^* W_i , P_{\mH}^t \pi^* W_j ) \nonumber \\
&&= 
- \frac{3}{2} \sum\nolimits_{k,l=n+1}^{n+p} g_N( [ W_i , W_j ] , W_k )  g_N( [ W_i , W_j ] , W_l ) g_0( V_k , V_l ) \nonumber\\
&&\quad +3 \sum\nolimits_{k=n+1}^{n+p} g_N( [ W_i , W_j ] , W_k ) g_t( V_k , [ V_i , P_{\mH}^t \pi^* W_j  ] + [ P_{\mH}^t \pi^* W_i , V_j ] ) \nonumber\\
&&\quad - 3 g_t( [ V_i , V_j  ]  , P_{\mV}^t [ P_{\mH}^t \pi^* W_i , P_{\mH}^t \pi^* W_j  ] ) \nonumber\\
&&\quad - \frac{3}{2} g_t( [ V_i , P_{\mH}^t \pi^* W_j  ] + [ P_{\mH}^t \pi^* W_i ,  V_j ] , [ V_i , P_{\mH}^t \pi^* W_j ] + [ P_{\mH}^t \pi^* W_i ,  V_j ] ) ,
\end{eqnarray}
\begin{eqnarray} \label{d3tnormTWiWjsec1}
&& \frac{\partial^3 }{\partial t^3} \sec_M( P_{\mH}^t \pi^* W_i , P_{\mH}^t \pi^* W_j ) \nonumber \\
&&=
-9 \sum\nolimits_{k=n+1}^{n+p} g_N( [ W_i , W_j ] , W_k ) g_0( V_k , [ V_i , V_j  ] ) \nonumber\\
&& \quad + 9g_t( [V_i , P_{\mH}^t \pi^* W_j  ] + [ P_{\mH}^t \pi^* W_i , V_j ] , [ V_i , V_j  ] ) ,
\end{eqnarray}
\begin{eqnarray} \label{d4tnormTWiWjsec1}
\frac{\partial^4 }{\partial t^4} \sec_M( P_{\mH}^t \pi^* W_i , P_{\mH}^t \pi^* W_j ) = - 18 g_t( [V_i , V_j  ] , [ V_i , V_j  ] ) 
\end{eqnarray}
and $\frac{\partial^5 }{\partial t^5} \sec_M( P_{\mH}^t \pi^* W_i , P_{\mH}^t \pi^* W_j )=0$.

We also have
\begin{eqnarray} \label{d2tAWiU1AWlU2}
&& 4 \ddt g_t( \mA^t_{P_{\mH}^t W_i } U_1 , \mA^t_{P_{\mH}^t W_l} U_2 ) \nonumber\\
&&=
2 \sum\nolimits_{j=n+1}^{n+p} \big( 
\sum\nolimits_{k=n+1}^{n+p} g_N([W_l, W_j] , W_k)g_0( V_k, U_2 ) \nonumber\\
&&\quad - g_t( [ V_l , P_{\mH}^t W_j ] + [ P_{\mH}^t W_l, V_j ] , U_2 ) \big)
 \cdot \big( 
\sum\nolimits_{k=n+1}^{n+p} g_N([W_i, W_j] , W_k)g_0( V_k, U_1 ) \nonumber\\
&&\quad - g_t( [ V_i , P_{\mH}^t W_j ] + [ P_{\mH}^t W_i, V_j ] , U_1 ) \big) \nonumber\\
&&\quad + 
2 \sum\nolimits_{j=n+1}^{n+p} g_t( P_{\mV}^t [P_{\mH}^t W_l, P_{\mH}^t W_j] ,U_2 ) \cdot g_t( [ V_i , V_j ] , U_1 ) \nonumber\\
&&\quad + 
2 \sum\nolimits_{j=n+1}^{n+p} g_t( P_{\mV}^t [P_{\mH}^t W_i, P_{\mH}^t W_j] ,U_1 ) \cdot g_t( [ V_l , V_j ] , U_2 )  ,
\end{eqnarray}
\begin{eqnarray} \label{d3tAWiU1AWlU2}
&& 4 \frac{\partial^3}{\partial t^3} g_t( \mA^t_{P_{\mH}^t W_i } U_1 , \mA^t_{P_{\mH}^t W_l} U_2 ) \nonumber\\
&&= 4 \sum\nolimits_{j=n+1}^{n+p} g_t([V_l, V_j], U_2)
 \cdot \big( 
\sum\nolimits_{k=n+1}^{n+p} g_N([W_i, W_j] , W_k)g_0( V_k, U_1 ) \nonumber\\
&&\quad - g_t( [ V_i , P_{\mH}^t W_j ] + [ P_{\mH}^t W_i, V_j ] , U_1 ) \big) \nonumber\\
&&\quad + 4 \sum\nolimits_{j=n+1}^{n+p} g_t( [ V_i , V_j ] , U_1 ) \big( 
\sum\nolimits_{k=n+1}^{n+p} g_N([W_l, W_j] , W_k)g_0( V_k, U_2 ) \nonumber\\
&&\quad - g_t( [ V_l , P_{\mH}^t W_j ] + [ P_{\mH}^t W_l, V_j ] , U_2 ) \big)
\nonumber\\
&&\quad + 
2 \sum\nolimits_{j=n+1}^{n+p}  
\big(  \sum\nolimits_{k=n+1}^{n+p} g_N([W_l, W_j] , W_k)g_0( V_l, U_2 ) \nonumber\\
&&\quad - g_t( [ V_l , P_{\mH}^t W_j ] + [ P_{\mH}^t W_l, V_j ] , U_2 )  \big)
\cdot g_t( [ V_i , V_j ] , U_1 ) \nonumber \\
%
%
&&\quad + 
2 \sum\nolimits_{j=n+1}^{n+p}  
\big(  \sum\nolimits_{k=n+1}^{n+p} g_N([W_i, W_j] , W_k)g_0( V_l, U_1 ) \nonumber\\
&&\quad - g_t( [ V_i , P_{\mH}^t W_j ] + [ P_{\mH}^t W_i, V_j ] , U_1 )  \big)
\cdot g_t( [ V_l , V_j ] , U_2 ) .
\end{eqnarray}
and
\begin{eqnarray} \label{d4tAWiU1AWlU2}
&& 4 \frac{\partial^4}{\partial t^4} g_t( \mA^t_{P_{\mH}^t W_i } U_1 , \mA^t_{P_{\mH}^t W_l} U_2 ) \nonumber\\
%
&&= 24 \sum\nolimits_{j=n+1}^{n+p} g_t([V_l, V_j], U_2) g_t( [ V_i , V_j ]  , U_1 ) .
\end{eqnarray}
and $\frac{\partial^5}{\partial t^5} g_t( \mA^t_{P_{\mH}^t W_i } U_1 , \mA^t_{P_{\mH}^t W_l} U_2 ) =0$.
\end{theorem}
\begin{proof}
Using 
\eqref{dtPHWi}, 
\eqref{dtgtViPHWjU} and $\dt V_i =0$, from \eqref{ddtsecXURSGF} we obtain \eqref{d3tsecXURSGF}. 
Since $\mV$ is integrable, from \eqref{BVVzero} we obtain 
\begin{equation} \label{BtViVjbracket}
B_t([V_i, V_j], U) = 0.
\end{equation}
Using \eqref{BtViVjbracket} in \eqref{d3tsecXURSGF} we obtain \eqref{d4tsecXURSGF}, which does not depend on $t$ by $\dt V_i =0$ and \eqref{BtViVjbracket}.
The proof of other formulas is similar.
\end{proof}

Note that our variations are defined by vector fields $W_i \in \mathfrak{X}_N$ and vertical fields $V_i \in \mathfrak{X}_{\mV}$ paired with them. The formulas for variations of sectional curvatures obtained 
in Theorem \ref{thconstVidtsec} are for directions defined by those vector fields $W_i$, but 
we can obtain similar formulas for arbitrary horizontal directions.

\begin{lemma} \label{lemmaKZU}
Let $\{ g_t, t \in (-\epsilon, \epsilon) \} \subset \Riem(M, \mV, g_0)$ be a 
variation such that $\pi : (M, g_t) \rightarrow (N, g_N)$ is a Riemannian submersion for all $t \in (-\epsilon, \epsilon)$ and \eqref{bsharpinv} holds, where $\{ W_{n+1}, \ldots , W_{n+p} \}$ are $g_N$-orthonormal vector fields on $N$. 

Let $\{ Z_{n+1}, \ldots , Z_{n+p} \}$ be $g_N$-orthonormal vector fields on $N$ such that we have $Z_i  = \sum\nolimits_{j=n+1}^{n+p} a_{ij} W_j$ and $\sum\nolimits_{j,k=n+1}^{n+p} a_{ij}a_{jk} = \delta_{ik}$  for all $i,k \in \{n+1, \ldots, n+p\}$ . 
Then variation $g_t$ satisfies
\begin{equation} \label{bsharpinvalt}
B_t^\sharp X = \sum\nolimits_{i=n+1}^{n+p} g_0(\xi_i , X) P_{\mH}^t  \pi^* Z_i ,
\end{equation}
for all $X \in \mV$, where $\xi_i =  \sum\nolimits_{j=n+1}^{n+p} a_{ij} V_j$.
Moreover, if $\dt V_i =0$ for all $i \in \{n+1,\ldots, n+p\}$, then
\begin{eqnarray} \label{dtsecZU}
&& \dt \sec_{(M,g_t)}(P_{\mH}^t \pi^* Z_i, U) = 2\sum\nolimits_{j,k,l=n+1}^{n+p} a_{ik}a_{il}
g_t([ P_{\mH}^t \pi^* W_k , P_{\mH}^t \pi^* W_j ] , U) \nonumber\\
&&\cdot \big(
\sum\nolimits_{m=n+1}^{n+p} g_N( [ W_l , W_j ] , W_m  )g_0(V_m ,U)  - g_t( [V_l , P_{\mH}^t \pi^* W_j ] + [P_{\mH}^t \pi^* W_l, V_j ] , U )
 \big). \nonumber\\
\end{eqnarray}
and
\begin{eqnarray} \label{ddtsecZU}
&&\ddt \sec_{(M,g_t)}(P_{\mH}^t \pi^* Z_i, U) = 2\sum\nolimits_{j,k,l=n+1}^{n+p} a_{ik}a_{il} \nonumber\\
&& \cdot
 \big(
\sum\nolimits_{m=n+1}^{n+p} g_N( [ W_l , W_j ] , W_m  )g_0(V_m ,U) - g_t( [V_l , P_{\mH}^t \pi^* W_j ] + [P_{\mH}^t \pi^* W_l, V_j ] , U )
 \big) \nonumber\\
&& \cdot  \big(
\sum\nolimits_{m=n+1}^{n+p} g_N( [ W_k , W_j ] , W_m  )g_0(V_m ,U)  - g_t( [V_k , P_{\mH}^t \pi^* W_j ] + [P_{\mH}^t \pi^* W_k , V_j ] , U )
 \big) \nonumber\\
&&\quad + 4\sum\nolimits_{j,k,l=n+1}^{n+p} a_{ik}a_{il} g_t( [P_{\mH}^t \pi^* W_k , P_{\mH}^t \pi^* W_j] , U ) g_0([V_l , V_j] , U).
\end{eqnarray}
Similarly,
\begin{eqnarray} \label{dtsecZiZj}
&& \dt \sec_{(M,g_t)}(P_{\mH}^t \pi^* Z_i,  P_{\mH}^t \pi^* Z_j) = \sum\nolimits_{k,l,r,s=n+1}^{n+p} a_{ik}a_{jl}a_{ir}a_{js} \nonumber\\
&& \cdot \big( 
-\frac{3}{2} \sum\nolimits_{q=n+1}^{n+p} g_N( [ W_k , W_l ] , W_q ) g_t(V_q , [ P_{\mH}^t \pi^* W_r , P_{\mH}^t \pi^* W_s ]) \nonumber\\
&& + \frac{3}{2} g_t( [ V_k , P_{\mH}^t \pi^* W_l ] + [P_{\mH}^t \pi^* W_k, V_l] , P_{\mV}^t [ P_{\mH}^t \pi^* W_r , P_{\mH}^t \pi^* W_s ] )
\big)
\end{eqnarray}
and
\begin{eqnarray} \label{ddtsecZiZj}
&& \ddt \sec_{(M,g_t)}(P_{\mH}^t \pi^* Z_i, P_{\mH}^t \pi^* Z_j) = \sum\nolimits_{k,l,r,s=n+1}^{n+p} a_{ik}a_{jl}a_{ir}a_{js} \nonumber\\
&& \cdot \big( 
-\frac{3}{2} \sum\nolimits_{q=n+1}^{n+p} g_N( [ W_k , W_l ] , W_q ) g_t(\dt V_q , [ P_{\mH}^t \pi^* W_r , P_{\mH}^t \pi^* W_s ]) \nonumber\\
&&
+ \frac{3}{2} g_t( [ \dt V_k , P_{\mH}^t \pi^* W_l ] + [P_{\mH}^t \pi^* W_k,  \dt V_l] , P_{\mV}^t [ P_{\mH}^t \pi^* W_r , P_{\mH}^t \pi^* W_s ] ) \nonumber\\
&&
-3 g_t( [V_k , V_l] , [  P_{\mH}^t \pi^* W_r , P_{\mH}^t \pi^* W_s ] ) \nonumber\\
&&
-\frac{3}{2} \sum\nolimits_{q,m=n+1}^{n+p} g_N( [ W_k , W_l ] , W_q ) 
g_N( [ W_r , W_s ] , W_m ) g_0(V_q, V_m) \nonumber\\
&&
+3 \sum\nolimits_{q=n+1}^{n+p} g_N( [ W_k , W_l ] , W_q ) g_t( V_q , [V_r, P_{\mH}^t \pi^* W_s ] + [P_{\mH}^t \pi^* W_r , V_s] ) \nonumber\\
&&
- \frac{3}{2} g_t( [ V_k , P_{\mH}^t \pi^* W_l ] + [P_{\mH}^t \pi^* W_k, V_l] , [ V_r , P_{\mH}^t \pi^* W_s ] + [P_{\mH}^t \pi^* W_r, V_s] )
\big) .
\end{eqnarray}

\end{lemma}
\begin{proof}
Equation \eqref{bsharpinvalt} follows from 
\begin{eqnarray*}
\sum\nolimits_{i=n+1}^{n+p} g_0(\xi_i , E_a) \pi^*(Z_i)
&=& \sum\nolimits_{i=n+1}^{n+p} \sum\nolimits_{j,k=n+1}^{n+p} a_{ik} a_{ij} g_0(V_j , E_a) \pi^*(W_k) \nonumber\\
&=& \sum\nolimits_{j,k=n+1}^{n+p} \delta_{jk} g_0(V_j , E_a) \pi^*(W_k) \nonumber\\
&=& \sum\nolimits_{i=n+1}^{n+p} g_0(V_i , E_a) \pi^*(W_i),
\end{eqnarray*}
\eqref{deflambdaai} and uniqueness of the solution of the system \eqref{bsharp1}-\eqref{bsharp3}. Equations \eqref{dtsecZU}, \eqref{ddtsecZU} follow from relations between $Z_i, \xi_i$ and $W_i, V_i$, and equations \eqref{dtsecXURSGF}, \eqref{d2tnormTWiWjsec}. Equations \eqref{dtsecZiZj} and \eqref{ddtsecZiZj} are obtained similarly. We note that we can also use $Z_i  = \sum\nolimits_{j=n+1}^{n+p} a_{ij} W_j$ and $\xi_i =  \sum\nolimits_{j=n+1}^{n+p} a_{ij} V_j$ in Lemma \ref{lemmaAWiWjU} and curvature formulas \eqref{secAXU} and \eqref{secAXY} to get the same results.
\end{proof} 

\begin{theorem} \label{thKillingremains}
Let $\{ g_t, t \in (-\epsilon, \epsilon) \} \subset \Riem(M, \mV, g_0)$ be a 
variation such that $\pi : (M, g_t) \rightarrow (N, g_N)$ is a Riemannian submersion for all $t \in (-\epsilon, \epsilon)$ and \eqref{bsharpinv} holds. Let $K$ be a vertical Killing vector field on $(M,g_0)$, then $K$ is a Killing vector field for all $\{ g_t, t \in (-\epsilon, \epsilon) \}$ if and only if for all $\{ V_{n+1} , \ldots , V_{n+p} \}$ in \eqref{bsharpinv} we have $[V_i , K] =0$.
\end{theorem}
\begin{proof}
Since $K$ is a Killing 
field on $(M,g_0)$ and 
$\{ g_t, t \in (-\epsilon, \epsilon) \} \subset \Riem(M, \mV, g_0)$,  
we have for all $t \in (-\epsilon, \epsilon)$ and all vertical fields $U_1, U_2$ 
\[
(\mathcal{L}_K g_t)(U_1, U_2) =  (\mathcal{L}_K g_0)(U_1, U_2)= 0.
\]
Since $K$ is vertical and 
$\pi : (M, g_t) \rightarrow (N, g_N)$ is a Riemannian submersion for all $t \in (-\epsilon, \epsilon)$, we have for all $t \in (-\epsilon, \epsilon)$ and all $g_t$-horizontal fields $Z_1, Z_2$:
\[
(\mathcal{L}_K g_t)(Z_1, Z_2) =0.
\]

Let $i \in \{n+1, \ldots , n+p\}$ and let $U$ be a vertical vector field. Then, using \eqref{bsharpinv}, integrability of $\mathcal{V}$,  \eqref{BVVzero} and \eqref{dtPHWi} we obtain
\begin{eqnarray} \label{dtLKg}
&& \dt ( (\mathcal{L}_K g_t)(P_{\mH}^t \pi^* W_i , U) ) =
\dt \big(  K (g_t( P_{\mH}^t \pi^* W_i , U )) - g_t([K , P_{\mH}^t \pi^* W_i ] , U) \nonumber \\
&& - g_t([K,U] , P_{\mH}^t \pi^* W_i )  \big) \nonumber \\
&& = - B_t ([K , P_{\mH}^t \pi^* W_i ] , U) - g_t ([K , \dt P_{\mH}^t \pi^* W_i ] , U) \nonumber \\
&&= g_t ([K , V_i ] , U).
\end{eqnarray}
\end{proof}

\begin{remark} \label{remd2tLKg}
From \eqref{dtLKg}, integrability of $\mV$ and \eqref{BVVzero}, it follows that if $\dt V_i= 0$ for all $i \in \{n+1 , \ldots, n+p\}$, then $\ddt (\mathcal{L}_K g_t) =0$.
\end{remark}

For local properties of variations, we can use their parametrization by sets of vector fields
$\{W_{n+1}, \ldots , W_{n+p}\}$, which are $g_N$-orthonormal and 
$\{V_{n+1}(t), \ldots , V_{n+p}(t)\}$, which are vertical. However, such vector fields $W_i$ usually cannot be defined on entire manifold $N$. 
With globally defined, $g_0$-orthonormal vertical Killing fields (e.g., fundamental fields induced by an isometric group action considered in Section \ref{subsechomogenous}), we can instead describe variations preserving a Riemannian submersion with totally geodesic fibers in terms of $1$-forms on $M$.

\begin{theorem} \label{thBomega}
Let $\{ g_t, t \in (-\epsilon, \epsilon) \} \subset \Riem(M, \mV, g_0)$ be a 
variation such that $\pi : (M, g_t) \rightarrow (N, g_N)$ is a Riemannian submersion with totally geodesic fibers for all $t \in (-\epsilon, \epsilon)$. Let the fibers of $\pi : (M, g_0) \rightarrow (N, g_N)$ be 
spanned by orthonormal vertical Killing vector fields $E_1, \ldots,  E_n$ on $(M,g_0)$. Then there exist $1$-forms $\{ \omega^1_t, \ldots , \omega^n_t \}$ on $M$ such that for all $X \in \mathfrak{X}_M$ and for all $a \in \{1, \ldots, n\}$ we have $\omega^a_t(X) = \omega^a_t( P_{\mH}^t X)$ and
\begin{equation} \label{Btomega}
B_t(E_a, X) = \omega^a_t( X ) .
\end{equation}
Moreover, for any $g_N$-orthonormal vector fields $\{W_{n+1}, \ldots , W_{n+p}\}$ and vertical fields equation \eqref{bsharpinv} holds with
\begin{equation} \label{Viomega}
V_i = \sum\nolimits_{a=1}^n \omega^a_t(\pi^* W_i) E_a .
\end{equation}
\end{theorem}
\begin{proof}
Since $\{ g_t, t \in (-\epsilon, \epsilon) \} \subset \Riem(M, \mV, g_0)$, 
we have $B_t(E_a,U) =0$ for all vertical $U$. It follows that for all $X \in \mathfrak{X}_M$ we have $B_t(E_a, X) = B_t(E_a, P_{\mH}^t X)$. On the other hand, by linearity of $B_t$, there exists a $1$-form $\omega^a_t$ on $M$ such that $\omega^a_t(X) = B_t(E_a , X) = \omega^a_t(P_{\mH}^t X)$. Let $\{W_{n+1} , \ldots, W_{n+p}\}$ be a local $g_N$-orthonormal frame.
Then for all vertical $U$
\begin{eqnarray} \label{Bsharpomega}
B_t^\sharp U &=& \sum\nolimits_{i=n+1}^{n+p} g_t(B_t^\sharp U, P_{\mH}^t \pi^* W_i) P_{\mH}^t \pi^* W_i \nonumber\\
&=&\sum\nolimits_{i=n+1}^{n+p} B_t(U, P_{\mH}^t \pi^* W_i) P_{\mH}^t \pi^* W_i \nonumber\\
&=&
\sum\nolimits_{a=1}^n \sum\nolimits_{i=n+1}^{n+p} g_0(E_a, U) B_t(E_a, P_{\mH}^t \pi^* W_i) P_{\mH}^t \pi^* W_i \nonumber\\
&=& \sum\nolimits_{i=n+1}^{n+p} g_0(U ,  \sum\nolimits_{a=1}^n \omega^a_t(\pi^* W_i) E_a  ) P_{\mH}^t \pi^* W_i .
\end{eqnarray}
Comparing \eqref{Bsharpomega} and \eqref{bsharpinv}, we obtain \eqref{Viomega}.
\end{proof}

We note that while $\omega^a_t(X) = \omega^a_t( P_{\mH}^t X)$, the forms $\omega^a_t$ in Theorem \ref{thBomega} do not need to be basic for $n>1$ (example of such variations will be given in Section \ref{subsechomogenous}).

We will say that a variation $g_t$ is induced by diffeomorphisms of $M$, if $g_t = f_t^* g_0$ for a one-parameter family of diffeomorphisms $f_t : M \rightarrow M$. 
A sufficient condition for a variation $g_t$ not to be induced by diffeomorphisms of a closed (i.e., compact, without boundary) manifold $M$ is  $\delta_{0} B_0 =0$ \cite{Besse}, where for all $X \in \mathfrak{X}_M$
\begin{equation} \label{defdeltaB0}
\delta_{0} B_0(X) = \sum\nolimits_{a=1}^n (\nabla^0_{E_a} B_0 )(E_a,X) + \sum\nolimits_{i=n+1}^{n+p} (\nabla^0_{\pi^* W_i} B_0 )(\pi^* W_i,X),
\end{equation}
for any $g_N$-orthonormal local basis $\{ W_{n+1}, \ldots, W_{n+p}\}$ and $g_0$-orthonormal basis of the fiber $\{ E_1, \ldots , E_n \}$; for any $1$-form $\omega$ on $N$, let us also define $\delta_N \omega$ by the following formula: 
\[
\delta_N \omega = \sum\nolimits_{i=n+1}^{n+p} (\nabla^N_{W_i} \omega)(W_i).
\]
For variations satisfying \eqref{Btomega} with basic forms $\omega^a_0 = \pi^* \alpha^a_0$, we have the following formulation of condition $\delta_{0} B_0 =0$.  
\begin{proposition} \label{propnondiffeo}
Let $\{ g_t, t \in (-\epsilon, \epsilon) \} \subset \Riem(M, \mV, g_0)$ be a 
variation such that $\pi : (M, g_t) \rightarrow (N, g_N)$ is a Riemannian submersion for all $t \in (-\epsilon, \epsilon)$ with totally geodesic fibers spanned by $g_0$-orthonormal vertical fields $\{ E_1, \ldots, E_n \}$, which are Killing on $(M, g_0)$. 
If \eqref{Btomega} holds with $\omega^a_0 = \pi^* \alpha^a_0$, where $\alpha^a_0$ is a $1$-form on $N$ for all $a \in \{1, \ldots, n\}$, then $\delta_0 B_0=0$ is equivalent to
\begin{eqnarray} \label{deltaomegaa}
&& \sum\nolimits_{a=1}^n \omega^a_0(\mA^0_X E_a) =0, \\ \label{deltaNeta}
&& \delta_N \alpha^a_t =0
\end{eqnarray}
for all $a \in \{1, \ldots, n\}$ and all $X \in \mH(0)$.
\end{proposition}
\begin{proof}
Let $\{ W_{n+1}, \ldots, W_{n+p}\}$ be a local $g_N$-orthonormal basis. 
First we consider \eqref{defdeltaB0} for $X=E_b$.
From $\mV$ being integrable and totally geodesic follows 
\begin{equation} \label{PHnablaEaEb}
P_{\mH}^0 \nabla^0_{E_a}E_b=0 
\end{equation}
and hence, by \eqref{BVVzero}, 
$\sum\nolimits_{a=1}^n (\nabla^0_{E_a} B_0)(E_a,E_b)=0$ for all $a,b \in \{1, \ldots,n\}$.
On the other hand, 
\begin{eqnarray}
&& \sum\nolimits_{i=n+1}^{n+p} (\nabla^0_{\pi^* W_i} B_0 )(\pi^* W_i, E_b) 
=  \sum\nolimits_{i=n+1}^{n+p} \big( \pi^* W_i( B_0 (\pi^* W_i , E_b)  ) \nonumber\\
&& - 
B_0(\nabla^0_{ \pi^* W_i } \pi^* W_i , E_b) - B_0( \pi^* W_i , \nabla^0_{ \pi^* W_i } E_b) \big) \nonumber\\
&&= \sum\nolimits_{i=n+1}^{n+p} \big( \pi^* W_i( \pi^* \alpha^b_0(\pi^* W_i ) )
- \pi^* \alpha^b_0(\nabla^0_{ \pi^* W_i } \pi^* W_i) \big)  \nonumber\\
&&= \delta_N \pi^* \alpha^b_0, \nonumber
\end{eqnarray}
where we used $B_0( \pi^* W_i , \nabla^0_{ \pi^* W_i } E_b) = B_0( \pi^* W_i , P_{\mV}^0 \nabla^0_{ \pi^* W_i } E_b)$ by \eqref{BVVzero}, and for all $U \in \mathfrak{X}_{\mV}$:
\[
g_0( \nabla^0_{ \pi^* W_i } E_b , U ) = (\mathcal{L}_{E_b}g_0)( \pi^* W_i , U) - g_0( \nabla^0_{ U } E_b , \pi^* W_i ) =0 
\]
by $E_b$ being Killing and \eqref{PHnablaEaEb}. 
Hence, $\delta_0 B_0(X)=0$ for vertical $X$ yields \eqref{deltaNeta}.

Now we examine \eqref{defdeltaB0} for a projectable $X \in \mathfrak{X}_{\mH(0)}$. We have
\begin{eqnarray} \label{deltaB01}
&& \sum\nolimits_{a=1}^n (\nabla^0_{E_a} B_0)(E_a,X)
= \sum\nolimits_{a=1}^n \big( E_a (\omega^a_0(X)) \nonumber\\
&& - B_0(\nabla^0_{E_a} E_a , X) - \omega^a_0(\nabla^0_{E_a}X) \big) =  - \sum\nolimits_{a=1}^n  \omega^a_0(\nabla^0_{E_a}X) ,
\end{eqnarray}
because $\omega^a_0(X)$ is constant along fiber for projectable $X$ and basic $\omega^a_0$, and by \eqref{BVVzero} we have $B_0(\nabla^0_{E_a} E_a , X) = B_0( P_{\mV}^0 \nabla^0_{E_a} E_a , X)$, and for all $U \in \mathfrak{X}_{\mV}$ we have 
\[
g_0( \nabla^0_{E_a} E_a , U ) = (\mathcal{L}_{E_a}g_0)(E_a, U) - g_0( \nabla^0_{U } E_a , E_a ) = - \frac{1}{2} U(g_0(E_a,E_a)) =0, 
\]
by $E_a$ being a unit Killing field.
On the other hand,
\begin{eqnarray}
&& \sum\nolimits_{i=n+1}^{n+p} (\nabla^0_{\pi^* W_i} B_0 )(\pi^* W_i, X) 
=  \sum\nolimits_{i=n+1}^{n+p} \big( \pi^* W_i( B_0 (\pi^* W_i , X)  ) \nonumber\\
&& - 
B_0(\nabla^0_{ \pi^* W_i } \pi^* W_i , X) - B_0( \pi^* W_i , \nabla^0_{ \pi^* W_i } X) \big). \nonumber
\end{eqnarray}
By \eqref{BXY} and $X, \pi^* W_i \in \mathfrak{X}_{\mH(0)}$ we obtain $ B_0 (\pi^* W_i , X)=0$ and
\begin{eqnarray}
&& \sum\nolimits_{i=n+1}^{n+p} (\nabla^0_{\pi^* W_i} B_0 )(\pi^* W_i, X)  \nonumber\\
&&=  \sum\nolimits_{i=n+1}^{n+p} \big(  - B_0(P_{\mV}^0  \nabla^0_{ \pi^* W_i } \pi^* W_i , X) - B_0( \pi^* W_i , P_{\mV}^0  \nabla^0_{ \pi^* W_i } X) \big). \nonumber
\end{eqnarray}
Since $\pi$ is a Riemannian submersion, $\mH(0)$ is totally geodesic and so $P_{\mV}^0  \nabla^0_{ \pi^* W_i } \pi^* W_i =0$. Hence
\begin{eqnarray}
&&\sum\nolimits_{i=n+1}^{n+p} (\nabla^0_{\pi^* W_i} B_0 )(\pi^* W_i, X)  \nonumber\\
&& = -B_0( \pi^* W_i , P_{\mV}^0  \nabla^0_{ \pi^* W_i } X) = 
-\sum\nolimits_{a=1}^n g_0(E_a, \nabla^0_{ \pi^* W_i } X) B_0(E_a, \pi^* W_i) \nonumber\\
&&= -\sum\nolimits_{a=1}^n g_0(E_a, \mA^0_{ \pi^* W_i } X) \omega^a_0( \pi^* W_i) 
= \sum\nolimits_{a=1}^n g_0(E_a, \mA^0_{X} \pi^* W_i) \omega^a_0( \pi^* W_i) 
\nonumber\\
&&= -\sum\nolimits_{a=1}^n g_0(\mA^0_{X} E_a, \pi^* W_i) \omega^a_0( \pi^* W_i) = -\sum\nolimits_{a=1}^n \omega^a_0 ( \mA^0_{X} E_a ). \nonumber
\end{eqnarray}
Hence, for projectable $X \in \mathfrak{X}_{\mH(0)}$ we have
\begin{eqnarray}
&& \delta_0 B_0( X ) =  - \sum\nolimits_{a=1}^n ( \omega^a_0(\nabla^0_{E_a}X)
+ \omega^a_0 ( \mA^0_{X} E_a ) ) = -2 \sum\nolimits_{a=1}^n \omega^a_0 ( \mA^0_{X} E_a ), \nonumber 
\end{eqnarray}
because $[E_a, X] \in \mathfrak{X}_\mV$ and hence
\begin{eqnarray}
&& \omega^a_0(\nabla^0_{E_a}X) = \omega^a_0(P_{\mH}^0 \nabla^0_{E_a}X) = \omega^a_0(P_{\mH}^0 \nabla^0_{X} E_a ) + \omega^a_0(P_{\mH}^0 [E_a,X]) \nonumber \\
&& = \omega^a_0(P_{\mH}^0 \nabla^0_{X} E_a ) = \omega^a_0(\mA^0_{X} E_a ) . \nonumber
\end{eqnarray}
It follows that $\delta_0 B_0(X)=0$ for horizontal $X$ yields \eqref{deltaomegaa}.
\end{proof}


\section{Circle bundles} \label{sectioncirclebundle}

For a circle bundle with $n=1$, that defines a Riemannian submersion $\pi : (M,g_0) \rightarrow (N,g_N)$, there is only one (up to rescaling) Killing vector field along the fibers. It follows that the only $1$-form $\omega_t$ in \eqref{Viomega} is then a pullback of a form 
from $N$.

\begin{theorem} \label{thS1bundleetaform}
Let $\{ g_t, t \in (-\epsilon, \epsilon) \} \subset \Riem(M, \mV, g_0)$ be a 
variation such that $\pi : (M, g_t) \rightarrow (N, g_N)$ is a Riemannian submersion with totally geodesic fibers for all $t \in (-\epsilon, \epsilon)$. Let the fibers of $\pi : (M, g_0) \rightarrow (N, g_N)$ be 
spanned by a unit vertical field $U$, which is Killing on $(M,g_0)$. Then there exists a $1$-form $\alpha_t$ on $N$ such that for all $X \in \mathfrak{X}_M$ we have
\begin{equation} \label{Bteta}
B_t(U, X) = \alpha_t( \pi_* X ) = \pi^* \alpha_t( X ) .
\end{equation}
Moreover, for any $g_N$-orthonormal vector fields $\{W_{n+1}, \ldots , W_{n+p}\}$ 
equation \eqref{bsharpinv} holds with
\begin{equation} \label{Vieta}
V_i = ( \alpha_t(W_i) \circ \pi) \cdot U 
\end{equation}
and we have
\begin{equation} \label{dtAXYeta}
\dt g_t(\mA^t_{ P_{\mH}^t \pi^* W_l } P_{\mH}^t \pi^* W_j , U ) = - {\rm d} \alpha_t(W_l , W_j) .
\end{equation}
\end{theorem}
\begin{proof}
For one-dimensional fibers, from \eqref{Viomega} we obtain $V_i = \omega_t(W_i) U$ for some $1$-form $\omega_t$. 
We will show that $\omega_t$ is a basic form, i.e., $\iota_U \omega_t =0$ and $\iota_U {\rm d} \omega_t =0$. The first condition follows from $\omega_t(U) = B_t(U,U) =0$; 
for the second, we have for all $i \in \{ n+1 , \ldots , n+p \}$:
\[
2{\rm d}\omega_t( U, \pi^* W_i) = U( \omega_t(\pi^* W_i) ) - \pi^* W_i (\omega_t (U)) - \omega_t( [U , \pi^* W_i] ) = U( \omega_t(\pi^* W_i) ),
\]
because $[U , \pi^* W_i] \in \mathfrak{X}_{\mV}$. But since variation $\{ g_t, t \in (-\epsilon, \epsilon) \} \subset \Riem(M, \mV, g_0)$ keeps the fibers geodesics, $\{ V_{n+1} , \ldots, V_{n+p} \}$ must be Killing fields when restricted to a fiber by Remark \ref{remVit}, and we have 
\[
U ( g_0(V_i , U) ) = \frac{1}{2} (\mathcal{L}_{V_i} g_0)(U,U) +g_0(V_i, \nabla^0_U U) =0,
\]
and hence $U( \omega_t (\pi^* W_i) ) =0$. It follows that $\omega_t$ is a basic form on $M$, and hence there exists a $1$-form $\alpha_t$ on $N$ such that $\omega_t = \pi^* \alpha_t$.

We obtain \eqref{dtAXYeta} from \eqref{dtAWiWjU} and the following computation 
\begin{eqnarray*} 
&& \sum\nolimits_{m=n+1}^{n+p} g_N( [ W_l , W_j ] , W_m  )g_0(V_m ,U)  - g_t( [V_l , P_{\mH}^t \pi^* W_j ] + [P_{\mH}^t \pi^* W_l, V_j ] , U ) \nonumber \\
&& = \sum\nolimits_{m=n+1}^{n+p} g_N( [ W_l , W_j ] , W_m  ) \alpha_t(W_m) +
P_{\mH}^t \pi^* W_j ( \alpha_t( W_l ) ) - P_{\mH}^t \pi^* W_l ( \alpha_t( W_j )) ) \nonumber\\
&&= \alpha_t( [ W_l , W_j ] ) - P_{\mH}^t \pi^* W_j ( \pi^* \alpha_t( \pi^* W_l ) ) - P_{\mH}^t \pi^* W_l (  \pi^* \alpha_t( \pi^* W_j ) ) \nonumber\\
&&= \pi^*\alpha_t( [ P_{\mH}^t \pi^* W_l , P_{\mH}^t \pi^* W_j ] ) - P_{\mH}^t \pi^* W_j ( \pi^* \alpha_t( P_{\mH}^t \pi^* W_l ) ) \nonumber\\
&& - P_{\mH}^t \pi^* W_l (  \pi^* \alpha_t( P_{\mH}^t \pi^* W_j ) ) \nonumber \\
&&= - 2 {\rm d} \pi^* \alpha_t( P_{\mH}^t \pi^* W_l , P_{\mH}^t \pi^* W_j ) = - 2 {\rm d} \alpha_t(W_l , W_j) .
\end{eqnarray*}
\end{proof}

For a Riemannian submersion $\pi :(M,g_t) \rightarrow (N,g_N)$, with one-dimensional geodesic fibers spanned by a unit vector field $U$, let $\mathfrak{X}_{P(t)}$ denote the set of projectable vector fields with fiberwise constant norm on $(M,g_t)$. In other words, $X \in \mathfrak{X}_{P(t)}$ if and only if $[X, U] \in \mathfrak{X}_{\mV}$ and $U(g_t(X,X))=0$. 
Vector fields from $\mathfrak{X}_{P(t)}$ are linear combinations of vector fields from an adapted frame $\{U, P_{\mH}^t \pi^* W_{n+1}, \ldots , P_{\mH}^t \pi^*  W_{n+p} \}$ with coefficients constant 
along fibers of $\pi$, where $\{  W_{n+1}, \ldots , W_{n+p} \}$ form a local $g_N$-orthonormal frame on $(N, g_N)$.
For all $x \in M$ every vector $X \in T_x M$ can be locally extended to a vector field from $\mathfrak{X}_{P(t)}$. 

\begin{lemma} \label{lemnablaXY}
Let $\pi :(M,g_t) \rightarrow (N,g_N)$ be a Riemannian submersion with one-dimensional geodesic fibers spanned by a unit vector field $U$. 
Then for all $X, Y ,Z \in \mathfrak{X}_{P(t)}$ we have $U (g_t( \nabla^t_{X} Y , Z) ) =0$.
\end{lemma}
\begin{proof}
From fibers being geodesics follows 
\[
U(g_t(\nabla^t_{U} P_{\mH}^t \pi^*X , U)) = - U( g_t(\nabla^t_{U} U, P_{\mH}^t \pi^*X)) =0.
\]
From $\pi$ being Riemannian submersion we obtain $U(g_t( \nabla^t_{P_{\mH}^t \pi^*X } P_{\mH}^t \pi^*Y , P_{\mH}^t \pi^*Z ) ) = U( g_N(\nabla^N_X Y ,Z) \circ \pi) =0$. From $P_{\mH}^t [ P_{\mH}^t \pi^*X, U] = 0$ follows 
\[
P_{\mH}^t \nabla^t_{U} P_{\mH}^t \pi^* X = P_{\mH}^t \nabla^t_{ P_{\mH}^t \pi^* X }U = \mA^t_{ P_{\mH}^t \pi^* X }U,
\]
and hence
\[
U(g_t( \nabla^t_U { P_{\mH}^t \pi^* X } , P_{\mH}^t \pi^* Y )  ) = U(g_t( \nabla^t_{P_{\mH}^t \pi^* X } U , P_{\mH}^t \pi^* Y )  ) = - U(g_t( \nabla^t_{P_{\mH}^t \pi^* X } P_{\mH}^t \pi^* Y  , U)).
\]
For all Riemannian submersions with one-dimensional, totally geodesic fibers we have by \cite{Tondeur} (last formula on page 52):
\begin{equation} \label{TondeurnablaUAXYU}
g_t( ( \nabla^t_{ U } \mA^t )_{ P_{\mH}^t \pi^* X } P_{\mH}^t \pi^* Y , U )=0.
\end{equation}
Using \eqref{TondeurnablaUAXYU} we obtain
\begin{eqnarray} \label{UgnablaXYU}
&& U(\mA^t_{P_{\mH}^t \pi^*X} P_{\mH}^t \pi^*Y ,U) =  U(\nabla^t_{ P_{\mH}^t \pi^*X  } P_{\mH}^t \pi^*Y , U) = g_t( (\nabla^t_{U} \mA^t)_{ P_{\mH}^t \pi^*X} P_{\mH}^t \pi^*Y ,U ) \nonumber\\
&&+ g_t(\mA^t_{ \nabla^t_{U} P_{\mH}^t \pi^*X} P_{\mH}^t \pi^*Y ,U  )  + g_t(\mA^t_{ P_{\mH}^t \pi^*X} \nabla^t_{U} P_{\mH}^t \pi^*Y ,U  ) \nonumber\\
&& = g_t(\mA^t_{ P_{\mH}^t \nabla^t_{U} P_{\mH}^t \pi^*X} P_{\mH}^t \pi^*Y ,U  ) + g_t(\mA^t_{ P_{\mH}^t \pi^*X} P_{\mH}^t \nabla^t_{U} P_{\mH}^t \pi^*Y ,U  ) \nonumber\\
&& = g_t(\mA^t_{ \mA^t_{P_{\mH}^t \pi^*X } U } P_{\mH}^t \pi^*Y ,U  ) + g_t(\mA^t_{P_{\mH}^t \pi^*X}  \mA^t_{P_{\mH}^t \pi^*Y} U  ,U  ) \nonumber\\
&& = \sum\nolimits_{i=n+1}^{n+p} g_t(\mA^t_{ P_{\mH}^t \pi^* W_i } P_{\mH}^t \pi^*Y ,U  ) g_t(\mA^t_{P_{\mH}^t \pi^*X } U, P_{\mH}^t \pi^* W_i) \nonumber\\
&&\quad + g_t(\mA^t_{P_{\mH}^t \pi^*X}  \mA^t_{P_{\mH}^t \pi^*Y} U  ,U  ) \nonumber\\
&& = \sum\nolimits_{i=n+1}^{n+p} g_t(\mA^t_{P_{\mH}^t \pi^*Y} U , P_{\mH}^t \pi^* W_i ) g_t(\mA^t_{P_{\mH}^t \pi^*X } U, P_{\mH}^t \pi^* W_i) \nonumber\\
&&\quad - g_t(\mA^t_{P_{\mH}^t \pi^*Y} U  , \mA^t_{P_{\mH}^t \pi^*X}  U  ) \nonumber\\
&& = g_t(\mA^t_{P_{\mH}^t \pi^*Y} U, \mA^t_{P_{\mH}^t \pi^*X } U ) - g_t(\mA^t_{P_{\mH}^t \pi^*Y} U  , \mA^t_{P_{\mH}^t \pi^*X}  U  ) =0.
\end{eqnarray}
It follows that for $X,Y,Z \in \mathfrak{X}_{P(t)}$ we have $U(g_t(\nabla^t_X Y,Z))=0$.
\end{proof}

\begin{proposition} \label{propAUprojectable}
Let $\{ g_t, t \in (-\epsilon, \epsilon) \} \subset \Riem(M, \mV, g_0)$ be a 
variation such that $\pi : (M, g_t) \rightarrow (N, g_N)$ is a Riemannian submersion with totally geodesic fibers for all $t \in (-\epsilon, \epsilon)$. Let the fibers of $\pi : (M, g_0) \rightarrow (N, g_N)$ be 
spanned by a unit vertical field $U$ which is Killing on $(M,g_0)$. Let $Z \in \mathfrak{X}_N$. 
Then $\mA^t_{\pi^* Z} U$ is projectable for all $t \in (-\epsilon, \epsilon)$.
\end{proposition}
\begin{proof}
$\mA^t_{\pi^* Z} U = \mA^t_{P_{\mH}^t \pi^* Z} U \in \mathfrak{X}_{\mH(t)}$ is a linear combination of local $g_t$-basic fields, with coefficients constant along the fibers by Lemma \ref{lemnablaXY}, so $[U, \mA^t_{P_{\mH}^t \pi^* Z} U] \in \mathfrak{X}_{\mV}$. 
\end{proof}

\begin{proposition} \label{propS1oversymplectic}
Let $\{ g_t, t \in (-\epsilon, \epsilon) \} \subset \Riem(M, \mV, g_0)$ be a 
variation such that $\pi : (M, g_t) \rightarrow (N, g_N)$ is a Riemannian submersion for all $t \in (-\epsilon, \epsilon)$ with geodesic fibers spanned by a unit vertical field $U$.
Then for all $Z \in \mathfrak{X}_{N}$ we have
\begin{eqnarray} \label{dtsecZUdform}
\dt \sec_{(M,g_t)}(P_{\mH}^t \pi^*Z, U) &=& 
-4 \langle \iota_{ P_{\mH}^t \pi^* Z } {\rm d} U^\flat , \iota_{ P_{\mH}^t \pi^* Z } \pi^* {\rm d} \alpha_t \rangle_{\mH(t)},
\end{eqnarray}
and
\begin{eqnarray} \label{d2tsecZUdform}
\ddt \sec_{(M,g_t)}(P_{\mH}^t \pi^* Z, U) = 8 \| \iota_{Z} {\rm d}\alpha_0 \|^2_N  .
\end{eqnarray}
\end{proposition} 
\begin{proof}
We note that since $U$ is unit and $\pi :(M,g_0) \rightarrow (N,g_N)$ is a Riemannian submersion, $U$ is a Killing field on $(M, g_0)$ \cite{Tondeur}. 
Let $\{W_{n+1}, \ldots , W_{n+p}\}$ be a local $g_N$-orthonormal frame and let vertical fields $\{V_{n+1}, \ldots , V_{n+p}\}$ be as in \eqref{bsharpinv}.
Let $\alpha_t$ be a $1$-form on $N$, such that 
$V_i = (\alpha_t(W_i) \circ \pi) \cdot  U = \pi^*\alpha_t (\pi^* W_i) \cdot  U$ for all $i \in \{n+1, \ldots, n+p\}$. Then $U$ remains a Killing field on $(M,g_t)$ by Theorem \ref{thKillingremains}, and we have
\begin{eqnarray*}
g_t( [ U , P_{\mH}^t \pi^* W_j ] , U) &=& (\mathcal{L}_U g_t)( P_{\mH}^t \pi^* W_j, U) - U( g_t(P_{\mH}^t \pi^* W_j  , U) ) \\
&& - g_t([U,U], P_{\mH}^t \pi^* W_j) =0.
\end{eqnarray*}
Without loss of generality, we can assume that $Z=Z_i$ is one of vector fields of a local orthonormal system $\{Z_{n+1}, \ldots, Z_{n+p}\} \subset \mathfrak{X}_N$, then
from \eqref{dtAXYeta} 
and Lemma \ref{lemmaKZU} we obtain 
\begin{eqnarray*} 
\dt \sec_{(M,g_t)}(P_{\mH}^t \pi^* Z_i, U) &=& 2\sum\nolimits_{j,k,l=n+1}^{n+p} a_{ik}a_{il} \cdot 
2 {\rm d} U^\flat(P_{\mH}^t \pi^* W_k , P_{\mH}^t \pi^* W_j ) \nonumber\\
&&\cdot \big(  - 2 {\rm d} \pi^* \alpha_t( P_{\mH}^t \pi^* W_l , P_{\mH}^t \pi^* W_j ) 
 \big)  \nonumber\\
&&= -4 \sum\nolimits_{k,l=n+1}^{n+p} a_{ik}a_{il} \langle \iota_{ P_{\mH}^t \pi^* W_k } {\rm d} U^\flat , \iota_{ P_{\mH}^t \pi^* W_l } \pi^* {\rm d} \alpha_t \rangle_{\mH(t)} \nonumber\\
&&= -4 \langle \iota_{ P_{\mH}^t \pi^* Z_i } {\rm d} U^\flat , \iota_{ P_{\mH}^t \pi^* Z_i } \pi^* {\rm d} \alpha_t \rangle_{\mH(t)},
\end{eqnarray*}
and
\begin{eqnarray*} 
\ddt \sec_{(M,g_t)}(P_{\mH}^t \pi^* Z_i, U) 
%
%
&& = 2\sum\nolimits_{j,k,l=n+1}^{n+p} a_{ik}a_{il} \cdot 2 {\rm d}\pi^*\alpha_t ( P_{\mH}^t \pi^* W_l , P_{\mH}^t \pi^* W_j ) \nonumber\\
&& \quad \cdot 2 {\rm d}\pi^*\alpha_t ( P_{\mH}^t \pi^* W_k , P_{\mH}^t \pi^* W_j ) \nonumber\\
&&= 8 \sum\nolimits_{k,l=n+1}^{n+p} a_{ik}a_{il} \langle \iota_{ P_{\mH}^t \pi^* W_l } {\rm d}\pi^*\alpha_t , \iota_{ P_{\mH}^t \pi^* W_k } {\rm d}\pi^*\alpha_t  \rangle_{\mH(t)} \nonumber\\
&&= 8 \langle \iota_{ P_{\mH}^t \pi^* Z_i } {\rm d}\pi^*\alpha_t , \iota_{ P_{\mH}^t \pi^* Z_i } {\rm d}\pi^*\alpha_t  \rangle_{\mH(t)} .
\end{eqnarray*}
%
\end{proof}

\subsection{Trivial bundles} \label{subsectrivialbundles}

A trivial bundle $\pi : (N \times S^1,g_0) \rightarrow (N, g_N)$ with the product metric $g_0 = g_N + g_{1}$, where $g_1$ is a metric on $S^1$, has vanishing vertizontal curvatures. In this section we examine conditions to increase them, while preserving the Riemannian submersion defined by 
the canonical projection from the Cartesian product.

\begin{corollary} \label{corS1oversymplectic}
Let $\{ g_t, t \in (-\epsilon, \epsilon) \} \subset \Riem(M, \mV, g_0)$ be a variation of the product metric $g_0$ on $M = N \times S^1$ preserving the Riemannian submersion $\pi : (M, g_0) \rightarrow (N,g_N)$, which is canonical projection from $N \times S^1$ onto $N$, and keeping fibers of $\pi$ geodesic. Let $x \in N$, then there exists $t_x >0$ such that all vertizontal curvatures of $(M,g_t)$ at every $y \in \pi^{-1}(\{x\})$ are positive for all $t \in (0, t_x)$ if and only if
${\rm d} \alpha_0$ in \eqref{Bteta} is non-degenerate at $x$.
\end{corollary}
\begin{proof}
For a product metric $g_0$ we have $\sec_{(M,g_0)}( \pi^* Z , U)=0$, where $U$ is a unit vertical vector field.
From integrability of $\mH(0)$ and \eqref{dtsecZUdform} we obtain 
\[
\dt \sec_{(M,g_t)}( P_{\mH}^t \pi^* Z , U) \vert_{t=0} =0.
\] 
By \eqref{d3tsecXURSGF}, \eqref{Vieta} and $\alpha_t$ being basic, we obtain 
$\frac{\partial^3}{\partial t^3} \sec_M( P_{\mH}^t \pi^* Z , U)=0$. Hence, there exists $t_x>0$ such that $\sec_{(M,g_t)}( P_{\mH}^t \pi^* Z , U)>0$ for all $t \in (0,t_x)$ if and only if $\ddt \sec_{(M,g_t)}(P_{\mH}^t \pi^* Z, U) \vert_{t=0} = \| \iota_{Z} {\rm d}\alpha_0 \|^2_N >0$. We have $\| \iota_{Z} {\rm d}\alpha_0 \|^2_N > 0$ for every $Z \in T_xN$ if and only if ${\rm d}\alpha_0$ is non-degenerate at $x$.
\end{proof}

In particular, we can increase all vertizontal curvatures on a trivial circle bundle if there exists a bounded and exact symplectic form on $N$ (hence, $N$ cannot be a closed manifold). For a contact (possibly closed) manifold $N$, we can increase vertizontal curvatures on $N \times S^1$ in all directions except those that project to the Reeb field on $N$.

\begin{corollary} \label{corS1overcontact} 
If there exists a contact form on $N$ with Reeb field $\xi$, then there exists a variation $\{ g_t, t \in (-\epsilon, \epsilon) \} \subset \Riem(M, \mV, g_0)$ of the product metric $g_0$ on $N \times S^1$ preserving Riemannian submersion $\pi : (N \times S^1, g_0) \rightarrow (N,g_N)$ and keeping fibers geodesic, with 
$\sec_M( P_{\mH}^t \pi^* Z , U)>0$ for all $Z \in \mathfrak{X}_{N}$ linearly independent of $\xi$ and all $t \in (-\epsilon, 0) \cup (0, \epsilon)$, where $U$ is a unit vertical vector field.
\end{corollary} 
\begin{proof}
The proof is analogous to that of Corollary \ref{corS1oversymplectic}, only now as $\alpha_t=\alpha_0$ we take the contact form on $N$. Since $\| \iota_Z {\rm d} \alpha_0 \|_N >0$ for all $Z$ linearly independent of the Reeb field $\xi$ on $N$, the proof follows from \eqref{d2tsecZUdform}.
\end{proof}

Conditions for making all sectional curvatures positive on fibers of $\pi : N \times S^1 \rightarrow N$, while keeping it a Riemannian submersion with geodesic fibers, are the following.

\begin{theorem} \label{thcurvatureincrease}
Let $\{ g_t, t \in (-\epsilon, \epsilon) \} \subset \Riem(M, \mV, g_0)$ be a 
variation such that $\pi : (M, g_t) \rightarrow (N, g_N)$ is a Riemannian submersion for all $t \in (-\epsilon, \epsilon)$ with geodesic fibers spanned by a unit vertical field $U$, which is Killing on $(M, g_0)$.
Let $x \in N$ be such that all sectional curvatures of $(N,g_N)$ at $x$ are positive.
Then there exists $t_x>0$ such that all sectional curvatures of $(M,g_t)$ at every $y \in \pi^{-1}(\{x\})$ are positive for all $t \in (0, t_x)$ if and only if $\alpha_t$ in Theorem \ref{thS1bundleetaform} satisfies the following conditions: ${\rm d}\alpha_0$ is non-degenerate at $x$, and there exists $\tau_x>0$ such that $( \nabla^N_X \alpha_t)(X,Y)=0$ for all $X,Y \in T_x N$ and all $t \in (0,\tau_x)$.
\end{theorem} 
\begin{proof} 
Since $\mH(t)$ is a codimension one distribution for all $t \in (-\epsilon, \epsilon)$, for all $x \in M$ every plane in $T_x M$ is spanned by vectors $X$ and $Y \cos \theta + U \sin \theta$, where $\theta \in [0,2\pi)$, $X,Y \in \mH(t)_x$ are $g_t$-orthonormal, and $U \in \mV_x$ is unit. It follows that
\begin{eqnarray} \label{secXYcosUsin}
\sec_{(M, g_t)}(X, Y\cos \theta + U \sin \theta)  &=& \sec_{(M, g_t)}(X,U)  \sin^2 \theta + \sec_{(M, g_t)}(X,Y) \cos^2 \theta \nonumber\\
&& - 2 g_t( R_t(X,Y)X ,U) \sin \theta \cos \theta,
\end{eqnarray}
where $R_t(X,Y)X = \nabla^t_X \nabla^t_Y X - \nabla^t_Y \nabla^t_X X - \nabla^t_{[X,Y]}X$. Since $\pi : (M,g_t) \rightarrow (N,g_N)$ is a Riemannian submersion with totally geodesic fibers, we have for $X,Y \in \mH(t)$ and $U \in \mV$ \cite{Tondeur}:
\begin{equation} \label{RXYXUandA}
g_t( R_t(X,Y)X ,U) = -g_t( ( \nabla^t_{ X } \mA )_{ X } Y , U) .
\end{equation}
We obtain from \eqref{dtsecZiZj} and \eqref{ddtsecZiZj}, respectively:
\begin{eqnarray} \label{dtsecZiZjetat}
\dt \sec_{(M,g_t)}(P_{\mH}^t \pi^* Z_i, P_{\mH}^t \pi^* Z_j) = 3 g_t( U , [ P_{\mH}^t \pi^* Z_i , P_{\mH}^t \pi^* Z_j ] ) {\rm d}\alpha_t( Z_i , Z_j )
\end{eqnarray}
and
\begin{eqnarray} \label{ddtsecZiZjetat}
\ddt \sec_{(M,g_t)}(P_{\mH}^t \pi^* Z_i, P_{\mH}^t \pi^* Z_j) &=& 3 g_t( U , [ P_{\mH}^t \pi^* Z_i , P_{\mH}^t \pi^* Z_j ])  {\rm d} \dt \alpha_t( Z_i , Z_j )  \nonumber\\
&&
- 6 ( {\rm d}\alpha_t( Z_i , Z_j ) )^2.
\end{eqnarray}
Also, from 
\eqref{dtAXYeta} it follows that for all $g_N$-orthonormal $Z_i, Z_j$
\[
\dt g_t( \mA_{P_{\mH}^t \pi^* Z_i} P_{\mH}^t \pi^* Z_j , U) = - {\rm d} \alpha_t(Z_i, Z_j) .
\]
We have \cite{O'Neill} 
\begin{equation} \label{nablaXYZprojects}
\pi_* P_{\mH}^t \nabla^t_{ P_{\mH}^t \pi^* Z_i } P_{\mH}^t \pi^* Z_j =  (\nabla^N_{ \pi_* P_{\mH}^t \pi^* Z_i }  \pi_* P_{\mH}^t \pi^* Z_j) \circ \pi = ( \nabla^N_{ Z_i } Z_j ) \circ \pi .
\end{equation} 
Hence,
$\dt g_t( \nabla^t_{ P_{\mH}^t \pi^* Z_i } P_{\mH}^t \pi^* Z_j , P_{\mH}^t \pi^* Z_k ) = \dt g_N( \nabla^N_{ Z_i } Z_j ,  Z_k  ) =0$ and using 
\[
g_t(  \mA^t_{ P_{\mH}^t \pi^* Z_j } P_{\mH}^t \pi^* Z_k , \nabla^t_{ P_{\mH}^t \pi^* Z_i } U ) =0,
\]
with \eqref{dtPHWi}, \eqref{Vieta}, 
and \eqref{dtAXYeta}, we obtain
\begin{eqnarray} \label{dtnablaAXYZ}
&& \dt g_t( ( \nabla^t_{ P_{\mH}^t \pi^* Z_i } \mA^t )_{ P_{\mH}^t \pi^* Z_j } P_{\mH}^t \pi^* Z_k , U ) \nonumber\\
&& = P_{\mH}^t \pi^* Z_i ( \dt g_t(  \mA^t_{ P_{\mH}^t \pi^* Z_j } P_{\mH}^t \pi^* Z_k , U )  ) + (\dt P_{\mH}^t \pi^* Z_i)( g_t( \mA^t_{ P_{\mH}^t \pi^* Z_j } P_{\mH}^t \pi^* Z_k ,U ) ) \nonumber\\
&& - \sum\nolimits_{l=n+1}^{n+p} g_t( \nabla^t_{ P_{\mH}^t \pi^* Z_i } P_{\mH}^t \pi^* Z_j , P_{\mH}^t \pi^* Z_l  ) \dt g_t( \mA^t_{ P_{\mH}^t \pi^* Z_l } P_{\mH}^t \pi^* Z_k , U  ) \nonumber\\
&& - \sum\nolimits_{l=n+1}^{n+p}  g_t( \mA^t_{ P_{\mH}^t \pi^* Z_l } P_{\mH}^t \pi^* Z_k , U  ) \dt  g_t( \nabla^t_{ P_{\mH}^t \pi^* Z_i } P_{\mH}^t \pi^* Z_j , P_{\mH}^t \pi^* Z_l  ) \nonumber\\
&& - \sum\nolimits_{l=n+1}^{n+p} g_t( \nabla^t_{ P_{\mH}^t \pi^* Z_i } P_{\mH}^t \pi^* Z_k , P_{\mH}^t \pi^* Z_l  ) \dt g_t( \mA^t_{ P_{\mH}^t \pi^* Z_j } P_{\mH}^t \pi^* Z_l , U  ) \nonumber\\
&& - \sum\nolimits_{l=n+1}^{n+p}  g_t( \mA^t_{ P_{\mH}^t \pi^* Z_j } P_{\mH}^t \pi^* Z_l , U  ) \dt  g_t( \nabla^t_{ P_{\mH}^t \pi^* Z_i } P_{\mH}^t \pi^* Z_k , P_{\mH}^t \pi^* Z_l  ) \nonumber\\
&&= - P_{\mH}^t \pi^* Z_i ({\rm d} \pi^* \alpha_t ( P_{\mH}^t Z_j, P_{\mH}^t Z_k ) )
- \alpha_t(W_i) \cdot U( g_t( \mA^t_{ P_{\mH}^t \pi^* Z_j } P_{\mH}^t \pi^* Z_k ,U ) ) \nonumber\\
&& + \sum\nolimits_{l=n+1}^{n+p} g_t( \nabla^t_{ P_{\mH}^t \pi^* Z_i } P_{\mH}^t \pi^* Z_j , P_{\mH}^t \pi^* Z_l  ) {\rm d} \pi^* \alpha_t ( P_{\mH}^t Z_l, P_{\mH}^t Z_k )  \nonumber\\
&&
+ \sum\nolimits_{l=n+1}^{n+p} g_t( \nabla^t_{ P_{\mH}^t \pi^* Z_i } P_{\mH}^t \pi^* Z_k , P_{\mH}^t \pi^* Z_l  ) {\rm d} \pi^* \alpha_t (P_{\mH}^t Z_j, P_{\mH}^t Z_l )  \nonumber\\
&&= -(\nabla^t_{ P_{\mH}^t \pi^* Z_i } {\rm d} \pi^* \alpha_t  )( P_{\mH}^t \pi^* Z_j , P_{\mH}^t \pi^* Z_k ) \nonumber\\
&& -  g_t( (\nabla^t_U \mA^t )_{ P_{\mH}^t \pi^* Z_j } P_{\mH}^t \pi^* Z_k , U ) \pi^* \alpha_t( P_{\mH}^t \pi^* Z_i ) \nonumber \\
&&= -(\nabla^t_{ P_{\mH}^t \pi^* Z_i } {\rm d} \pi^* \alpha_t  )( P_{\mH}^t \pi^* Z_j , P_{\mH}^t \pi^* Z_k ),
\end{eqnarray}
where we used \eqref{TondeurnablaUAXYU} to obtain the last line.
From \eqref{nablaXYZprojects} 
and $\pi^* \alpha_t$ being basic,
we obtain 
\begin{eqnarray} \label{nablaNAXYZU}
&& (\nabla^t_{ P_{\mH}^t \pi^* Z_i } {\rm d} \pi^* \alpha_t  )( P_{\mH}^t \pi^* Z_i , P_{\mH}^t \pi^* Z_j ) =  P_{\mH}^t \pi^* Z_i ( {\rm d} \pi^* \alpha_t ( P_{\mH}^t \pi^* Z_i , P_{\mH}^t \pi^* Z_j ) ) \nonumber\\
&& - {\rm d} \pi^* \alpha_t ( \nabla^t_{ P_{\mH}^t \pi^* Z_i } P_{\mH}^t \pi^* Z_i , P_{\mH}^t \pi^* Z_j )
- {\rm d} \pi^* \alpha_t ( P_{\mH}^t \pi^* Z_i ,  \nabla^t_{ P_{\mH}^t \pi^* Z_i } P_{\mH}^t \pi^* Z_j ) \nonumber\\
&&= P_{\mH}^t \pi^* Z_i ( {\rm d} \alpha_t ( Z_i , Z_j ) \circ \pi )  \nonumber\\
&& - {\rm d} \pi^* \alpha_t ( P_{\mH}^t \nabla^t_{ P_{\mH}^t \pi^* Z_i } P_{\mH}^t \pi^* Z_i , P_{\mH}^t \pi^* Z_j )
- {\rm d} \pi^* \alpha_t ( P_{\mH}^t \pi^* Z_i , P_{\mH}^t \nabla^t_{ P_{\mH}^t \pi^* Z_i } P_{\mH}^t \pi^* Z_j ) \nonumber\\
&&= \pi_* P_{\mH}^t \pi^* Z_i ( {\rm d} \alpha_t ( Z_i , Z_j ) ) \circ \pi - {\rm d} \alpha_t ( \pi_* P_{\mH}^t \nabla^t_{ P_{\mH}^t \pi^* Z_i } P_{\mH}^t \pi^* Z_i , \pi_* P_{\mH}^t \pi^* Z_j ) \circ \pi \nonumber\\
&& 
- {\rm d} \alpha_t ( \pi_* P_{\mH}^t \pi^* Z_i , \pi_* P_{\mH}^t \nabla^t_{ P_{\mH}^t \pi^* Z_i } P_{\mH}^t \pi^* Z_j ) \circ \pi \nonumber\\
&&= Z_i ( {\rm d} \alpha_t ( Z_i , Z_j ) ) \circ \pi - {\rm d} \alpha_t ( \nabla^N_{ Z_i } Z_i , Z_j ) \circ \pi 
- {\rm d} \alpha_t (  Z_i ,  \nabla^N_{ Z_i } Z_j ) \circ \pi \nonumber\\
&&= (\nabla^N_{ Z_i } {\rm d} \alpha_t)(Z_i, Z_j) \circ \pi
\end{eqnarray}
%
and hence, assuming that for all $X,Y \in \mathfrak{X}_N$ we have $(\nabla^N_X \alpha_t)(X, Y) =0$, 
\begin{eqnarray*}
&&\sec_{(M, g_t)}(P_{\mH}^t \pi^* Z_i, P_{\mH}^t \pi^* Z_j\cos \theta + U \sin \theta) \nonumber\\ 
&&= 4 \| \iota_{Z_i} {\rm d}\alpha_0 \|^2_N t^2 \sin^2 \theta + ( \sec_{(N, g_N)}(Z_i, Z_j) - 3 ( {\rm d}\alpha_0( Z_i , Z_j ) )^2 t^2  ) \cos^2 \theta + O(t^3). \nonumber\\
\end{eqnarray*}
\end{proof} 

\begin{remark}
Assumptions of Theorem \ref{thcurvatureincrease} are satisfied in particular at points where ${\rm d}\alpha_t$ is a K\"ahler form on $(N,g_N)$.
\end{remark}

We give an example of variation by family of metrics with integrable horizontal distributions.

\begin{example} \label{exproduct}
Let $(N,g_N)$ be a closed manifold. Let $\pi : (N \times S^1,g_0) \rightarrow (N,g_N)$ be the projection and let $g_0$ be the product metric. Let $\omega_t = \pi^* \alpha_0$ for all $t \in (-\epsilon, \epsilon)$, where $\alpha_0$ is a harmonic $1$-form on $N$,
and let $\{ g_t, t \in (-\epsilon, \epsilon) \} \subset \Riem(M, \mV, g_0)$ be a variation preserving Riemannian submersion $\pi$, such that \eqref{Bteta} holds. Then  $\mH(t)$ is integrable for all $t \in (-\epsilon, \epsilon)$ by \eqref{dtAXYeta}, and $\pi : (M,g_t) \rightarrow (N,g_N)$ remains a Riemannian submersion with totally geodesic fibers, so every $g_t$ is locally a product metric on $S^1 \times H(t)$, where $H(t)$ is the integral manifold of $\mH(t)$, with $g_t \vert_{H(t)} = \pi^* g_N$. 
Also, 
$(N \times S^1,g_t)$ are not globally diffeomorphic to $(N \times S^1,g_0)$, because \eqref{deltaomegaa} and \eqref{deltaNeta} are satisfied. 
\end{example}


\subsection{K-contact and Sasaki structures} \label{subsecKcontact}

Recall \cite{Blair} that $(\phi_t, U, \eta_t, g_t)$, where $\phi_t$ is a $(1,1)$-tensor field, $U$ is a vector field, $\eta_t$ is a $1$-form and $g_t$ is a Riemannian metric, is called a contact metric structure on a manifold $M$ if the following hold:
\begin{eqnarray}
\label{UReeb}
&&\eta_t(X) = g_t(U,X)  \quad {\rm for \; all \;} X \in \mathfrak{X}_M \\ 
\label{iotaReeb}
&&\iota_U {\rm d} \eta_t=0 \\
\label{detacms}
&& {\rm d}\eta_t(X,Y) = g_t(X, \phi_t Y)  \quad {\rm for \; all \;} X,Y \in \mathfrak{X}_M \\
\label{phitisomH}
&& g_t(\phi_t X, \phi_t Y) = g_t(X,Y) \quad {\rm for \; all \;} X,Y \in \mathfrak{X}_{\mH(t)} \\
\label{phitalmostcomplex}
&& \phi_t^2 = - {\rm Id}_{TM} + \eta_t \otimes U.
\end{eqnarray}
If, additionally, $U$ is a Killing field on $(M,g_t)$, then the contact metric structure is called $K$-contact.

Let $\pi : (M, g_t) \rightarrow (N, g_N)$ be a Riemannian submersion with one-dimensional fibers spanned by a unit Killing field $U$. Let $\eta_t(X) = g_t(U,X)$ for all $X \in \mathfrak{X}_M$. Let $\phi_t$ be the tensor defined by
\begin{equation} \label{definitionphit}
\phi_t(X) = -\mA^t_X U
\end{equation}
for all $X \in \mathfrak{X}_M$, where $\mA^t$ is the O'Neill tensor \eqref{definitionAtensor} of $\pi: (M,g_t) \rightarrow (N,g_N)$.

We will say that $(\pi, g_t , U )$ defines a $K$-contact structure on $M$ if for all $X,Y \in \mathfrak{X}_M$ we have
\begin{equation} \label{Aisometry}
g_t( \mA^t_X U , \mA^t_Y U ) = g_t(P_{\mH}^t X, P_{\mH}^t Y)
\end{equation}
and
\begin{equation} \label{KcontactA}
\mA^t_{  \mA^t_{X} U } U = - X + g_t(X,U)U.
\end{equation} 
Then, in convention of \cite{Blair}, $(\phi_t, U, \eta_t, g_t)$ is a $K$-contact structure on $M$. Indeed, for $U$ being a unit vector field, $\nabla^t_U U =0$ is equivalent to \eqref{iotaReeb}; equation \eqref{definitionphit} is equivalent to \eqref{detacms}; \eqref{Aisometry} is equivalent to \eqref{phitisomH} and \eqref{KcontactA} is equivalent to \eqref{phitalmostcomplex}. Thus, $(\phi_t, U, \eta_t, g_t)$ is a contact metric structure, and since $U$ is assumed to be a Killing field, it is a $K$-contact structure.

\begin{remark} \label{remRScms}
On the other hand, if  $(\phi_t, U, \eta_t, g_t)$ is a contact metric structure and $U$ defines a Riemannian foliation (i.e., $(\mathcal{L}_U g_t)(X,Y)=0$ for all $X,Y \in \mathfrak{X}_{ {\ker} \eta_t }$), then for all $X,Y \in \ker \eta_t$ we have from \eqref{detacms} and \eqref{UReeb}:
\begin{eqnarray*}
&&g_t(\phi_t X, Y) = {\rm d}\eta_t(Y,X) = -{\rm d}\eta_t(X,Y) = \frac{1}{2} \eta_t( [X,Y] ) =  g_t(U , \frac{1}{2}[X,Y])  \nonumber\\
&& = \frac{1}{2} g_t(U, \nabla^t_X Y - \nabla^t_Y X)= g_t(U, \nabla^t_X Y) - \frac{1}{2}g_t(U, \nabla^t_X Y + \nabla^t_Y X)  \nonumber\\
&& =
g_t(U, \nabla^t_X Y) + \frac{1}{2} (\mathcal{L}_U g_t)(X,Y)= g_t(U, \nabla^t_X Y) = - g_t(\nabla^t_X U , Y).
\end{eqnarray*}
From \eqref{iotaReeb} and \eqref{detacms} we have $\phi_t X \in \mathfrak{X}_{\mH(t)}$ for all $X \in \mathfrak{X}_M$ and $\phi_t U =0$, hence $\phi_t X = - P_{\mH}^t \nabla^t_{P_{\mH}^t X} U = - \mA^t_X U$ for the O'Neill tensor $\mA^t$ of a local Riemannian submersion defined by flowlines of $U$ on $(M,g_t)$.
\end{remark}

We will say that $(\pi, g_t , U )$ defines a Sasaki structure on $M$ if for all $X,Y \in \mathfrak{X}_M$ equations \eqref{Aisometry} and \eqref{KcontactA} hold, and additionally we have
\begin{equation} \label{Sasakiphi}
(\nabla^t_X \phi_t)(Y) = g_t(X,Y)U - g_t(U,Y)X.
\end{equation} 
Then, in convention of \cite{Blair}, $(\phi_t, U, \eta_t , g_t)$ is a Sasaki structure on $M$.

\begin{remark}
$\mA^t_{\pi^* Z} U$ is projectable for all $Z \in \mathfrak{X}_N$ if and only if $\mathcal{L}_U \phi_t =0$. 
\end{remark}

The following proposition rephrases in the language of Riemannian submersions and infinitesimal variations known results about ``type II deformations'' of $K$-contact and Sasaki structures \cite{Boyer, Goertsches2008}. Some of the computations in its proof will be useful in examining weak contact metric structures further below. 

\begin{proposition}
Let 
$\pi : (M, g_0) \rightarrow (N, g_N)$ be 
a Riemannian submersion with fibers 
spanned by a unit vertical field $U$, which is Killing on $(M,g_0)$. 
Let $\{ g_t, t \in (-\epsilon, \epsilon) \} \subset \Riem(M, \mV, g_0)$ be a 
variation such that $\pi : (M, g_t) \rightarrow (N, g_N)$ is a Riemannian submersion with totally geodesic fibers for all $t \in (-\epsilon, \epsilon)$. 
If $(\pi, g_0 , U )$ defines a $K$-contact (resp. Sasaki) structure on $M$, then $(\pi, g_t , U )$ defines a $K$-contact (resp. Sasaski) structure on $M$ for all $t \in (-\epsilon, \epsilon)$ if and only if \eqref{Bteta} holds with $\alpha_t$ closed for all $t \in (-\epsilon, \epsilon)$.
\end{proposition}
\begin{proof}
First we prove that if \eqref{Aisometry} holds for all $t \in (-\epsilon, \epsilon)$, then ${\rm d} \alpha_t=0$. Indeed, from \eqref{secAXU} and \eqref{d2tsecZUdform} we obtain that $\ddt g_t( \mA^t_{ P_{\mH}^t \pi^* W_i } U , \mA^t_{ P_{\mH}^t \pi^* W_i } U )=0$ if and only if ${\rm d} \alpha_t=0$.

Now we assume that ${\rm d} \alpha_t=0$ and $(\pi, g_0 , U )$ defines a $K$-contact structure, then from \eqref{dtAWiU1WjU2} it follows that
\begin{eqnarray} \label{dtAXUAYUetaclosed}
&& \dt g_t( \mA^t_{P_{\mH}^t \pi^* W_i } U , \mA^t_{P_{\mH}^t \pi^* W_l } U )
= \sum\nolimits_{j=n+1}^{n+p} g_t( \mA^t_{ P_{\mH}^t \pi^* W_l }  P_{\mH}^t \pi^* W_j , U ) {\rm d} \alpha_t( W_i , W_j ) \nonumber\\
&& + \sum\nolimits_{j=n+1}^{n+p} g_t( \mA^t_{ P_{\mH}^t \pi^* W_i }  P_{\mH}^t \pi^* W_j , U ) {\rm d} \alpha_t( W_l , W_j ) =0.
\end{eqnarray}
Since $g_0( \mA^0_{ \pi^* W_i } U , \mA^0_{ \pi^* W_l } U ) = \delta_{il}$, 
we see that \eqref{Aisometry} holds for all $t \in (-\epsilon, \epsilon)$. Moreover, it follows from \eqref{dtAXUAYUetaclosed} that $\{ \mA^t_{P_{\mH}^t \pi^* W_i } U \}$, where $i \in \{n+1, \ldots,  n+p\}$, form a $g_t$-orthonormal frame of $\mH(t)$.
Hence, we have
\begin{eqnarray} \label{AAXUUt}
\mA^t_{  \mA^t_{X} U } U &=& \sum\nolimits_{i=n+1}^{n+p} g_t(\mA^t_{X} U , P_{\mH}^t \pi^* W_i ) \mA^t_{ P_{\mH}^t \pi^* W_i } U \nonumber \\
&=& 
- \sum\nolimits_{i=n+1}^{n+p} g_t(\mA^t_{X} P_{\mH}^t \pi^* W_i , U ) \mA^t_{ P_{\mH}^t \pi^* W_i } U \nonumber\\
&=& \sum\nolimits_{i=n+1}^{n+p} g_t(\mA^t_{ P_{\mH}^t \pi^* W_i } X , U ) \mA^t_{ P_{\mH}^t \pi^* W_i } U \nonumber\\
&=& -  \sum\nolimits_{i=n+1}^{n+p} g_t(\mA^t_{ P_{\mH}^t \pi^* W_i } U , X ) \mA^t_{ P_{\mH}^t \pi^* W_i } U \nonumber\\
&=& -P_{\mH}^t X = -X + g_t(X,U)U.
\end{eqnarray}
It follows that if ${\rm d} \alpha_t=0$ then both \eqref{Aisometry} and \eqref{KcontactA} hold and $(\pi, g_t , U )$ defines a $K$-contact structure on $M$ for $t \in (-\epsilon,\epsilon)$.

To prove that if ${\rm d} \alpha_t=0$, then also \eqref{Sasakiphi} holds, we will write it in an equivalent form. Since $g_t(U,U)=1$ and $\nabla^t_U U =0$, we have $\nabla^t_X U = \mA^t_X U$. Hence,
\begin{eqnarray} \label{nablaphiandA}
&&( \nabla^t_X \phi_t) (Y) = \nabla^t_X \phi_t(Y) - \phi_t(\nabla^t_X Y) 
= -\nabla^t_X \mA^t_Y U + \mA^t_{  \nabla^t_X Y} U \nonumber\\
&& = -\nabla^t_X \mA^t_Y U + \mA^t_{  \nabla^t_X Y} U + \mA^t_{ Y}  \nabla^t_X U - \mA^t_{ Y} \mA^t_X U \nonumber\\
&&= -(\nabla^t_X \mA^t)_Y U  - \mA^t_{ Y} \mA^t_X U.
\end{eqnarray}
For $X=Y=U$, \eqref{Sasakiphi} becomes
\[
-(\nabla^t_U \mA^t)_U U  - \mA^t_{ U} \mA^t_U U =0 
\]
which is true for all Riemannian submersions.

For $X \in \mH(t)$ and $Y=U$, by $\mA_U =0$, \eqref{Sasakiphi} becomes
\begin{equation} \label{SasakiXU}
-(\nabla^t_X \mA^t)_U U = -X ,
\end{equation}
which holds by \eqref{AAXUUt} and \cite[5.32 Lemma]{Tondeur} $(\nabla^t_X \mA^t)_U = - \mA^t_{\mA^t_X U}$.

For $X=U$ and $Y \in \mH(t)$, by $\mA_U =0$, \eqref{Sasakiphi} becomes
\[
(\nabla^t_U \mA^t)_Y U = 0,
\]
which is true for Riemannian submersions with geodesic fibers, by $(\nabla^t_U \mA^t)_Y$ being alternating \cite{Tondeur}:
\begin{equation} \label{nablaUAYUU}
g_t( (\nabla^t_U \mA^t)_Y U , U ) = - g_t( (\nabla^t_U \mA^t)_Y U , U )
\end{equation}
and, by \cite[p.52, last formula]{Tondeur}:
\[
g_t( (\nabla^t_U \mA^t)_Y U , Z ) = - g_t( (\nabla^t_U \mA^t)_Y Z , U ) = 0
\]
for all $Z \in \mH(t)$.

Finally, for $X,Y \in \mH(t)$, \eqref{SasakiXU} becomes
\[
-(\nabla^t_X \mA^t)_Y U  - \mA^t_{ Y} \mA^t_X U = g_t(X,Y) U,
\]
which is equivalent to the system
\begin{eqnarray} \label{SasakiXYU}
&& -g_t( (\nabla^t_X \mA^t)_Y U , U)  - g_t( \mA^t_{ Y} \mA^t_X U , U) = g_t(X,Y) , \\
\label{SasakiXYZ}
&& -g_t( (\nabla^t_X \mA^t)_Y U , Z)  - g_t( \mA^t_{ Y} \mA^t_X U , Z) =0
\end{eqnarray}
for all $Z \in \mH(t)$.
Using \eqref{nablaUAYUU} and $\mA^t_Y$ being alternating, we write \eqref{SasakiXYU} as
\[
g_t( \mA^t_X U ,  \mA^t_{ Y} U) = g_t(X,Y),
\]
which holds by \eqref{Aisometry}.
To prove \eqref{SasakiXYZ}, using \eqref{dtnablaAXYZ} we obtain $-g_t( (\nabla^t_X \mA^t)_Y U , Z) = g_t((\nabla^t_X \mA^t)_Y Z , U) = ({\nabla^t_X \rm d \alpha_t})(Y,Z)$ and $g_t( \mA^t_{ Y} \mA^t_X U , Z) = -g_t( \mA^t_X U , \mA^t_{ Y} Z  )=0$ since $\mA^t_X U \in \mH(t)$ and $\mA^t_{ Y} Z \in \mV$.
It follows that \eqref{SasakiXYZ} is equivalent to vanishing of ${\nabla^t \rm d \alpha_t}$, which is a weaker condition than $\alpha_t$ being closed. Hence, if defining a $K$-contact structure by $(\pi, g_t, U)$ is preserved by variations $g_t$, then so is defining a Sasaki structure.
\end{proof}

On $K$-contact or Sasaki manifolds, the infinitesimal variations with $\alpha_t$ closed and basic integrate to ''type II deformations'' considered in \cite{Boyer,Goertsches2008}, 
which are induced by diffeomorphisms of $(M,g_0)$.
For a Riemannian submersion defined by a Killing field, we have the following result. 
\begin{theorem} \label{thdiffequivalence}
Let $\{ g_t, t \in (-\epsilon, \epsilon) \} \subset \Riem(M, \mV, g_0)$ be a 
variation such that $\pi : (M, g_t) \rightarrow (N, g_N)$ is a Riemannian submersion for all $t \in (-\epsilon, \epsilon)$ with geodesic fibers spanned by a unit vertical field $U$, which is Killing on analytic, simply connected manifold $(M, g_0)$. 
If \eqref{Bteta} holds with $\alpha_t$ basic, closed and analytic
then 
for all $t \in (-\epsilon, \epsilon)$ there exists o diffeomorphism $f_t : M \rightarrow M$ such that $g_t = f_t^* g_0$.
\end{theorem}
\begin{proof}
By the assumption ${\rm d} \alpha_t=0$ and by $\{ g_t, t \in (-\epsilon, \epsilon) \} \subset \Riem(M, \mV, g_0)$, we have that $U$ is a Killing field on $g_t$ with $g_t(U,U)=1=g_0(U,U)$.
From \cite[Theorem 7.2.]{Sharipov} for all $t \in (-\epsilon, \epsilon)$ we obtain existence of a local diffeomorphism, which on analytic and simply connected $M$ extends uniquely to a diffeomorphism $f_t : M \rightarrow M$ such that $g_t = f_t^* g_0$ \cite{KobayashiNomizu1}.
\end{proof}

We note that, by Example \ref{exproduct}, the assumption that $(M,g_0)$ is simply connected is necessary in Theorem \ref{thdiffequivalence}.

If $g_t=f_t^* g_0$ for a one-parameter family of diffeomorphisms $f_t : M \rightarrow M$, generated by a vector field $Z$, then $B_t = \dt g_t = \mathcal{L}_Z g_t$. 
We examine some properties of vector fields that define variations $\{ g_t, t \in (-\epsilon, \epsilon) \} \subset \Riem(M, \mV, g_0)$, which preserve a Riemannian submersion with geodesic fibers.

\begin{proposition} \label{propLZgdefinesvariation}
Let $\pi : (M, g_0) \rightarrow (N, g_N)$ be a Riemannian submersion for all with geodesic fibers spanned by a unit vertical field $U$, which is Killing on 
$(M, g_0)$.
Let $Z \in \mathfrak{X}_N$.
Let $\alpha_t$ be a $1$-form on $N$, such that for all $X \in \mathfrak{X}_M$
\begin{equation} \label{alphaLieZ}
\pi^* \alpha_t (X) = (\mathcal{L}_{\pi^* Z}g_0)(U,X) . 
\end{equation} 
Then there exists a variation $\{ g_t, t \in (-\epsilon, \epsilon) \} \subset \Riem(M, \mV, g_0)$ 
such that $\pi : (M, g_t) \rightarrow (N, g_N)$ is a Riemannian submersion with geodesic fibers, such that \eqref{Bteta} holds with $\alpha_t$ given by \eqref{alphaLieZ}.
\end{proposition} 
\begin{proof}
We check that 
a $1$-form $\omega$ given by $\omega(X) =(\mathcal{L}_{\pi^* Z}g_0)(U,X)$ for all $X \in \mathfrak{X}_M$ 
is a basic form for Riemannian submersion $\pi: (M,g_t) \rightarrow (N,g_N)$ for all $t \in (-\epsilon, \epsilon)$. Indeed, 
\begin{eqnarray*}
\omega(U) &=& (\mathcal{L}_{\pi^*}g_0)(U,U) = g_0(\nabla^0_U \pi^* Z, U ) \nonumber\\
&=& 
-g_t( \nabla^0_U  U , \pi^* Z ) =0, 
\end{eqnarray*}
by fibers of $\pi: (M,g_0) \rightarrow (N,g_N)$ being geodesics.
Let $X$ be a projectable vector field on $M$. Then $X = \pi^* Y + W$, where $Y \in \mathfrak{X}_N$ and $W \in \mathfrak{X}_{\mV}$. We have
\begin{eqnarray*}
&& \iota_U {\rm d} \omega (X) = U(\omega (X)) = U(\omega (\pi^* Y))
= U( (\mathcal{L}_{\pi^* Z} g_0)(U, \pi^* Y) ) \nonumber\\
&& = U (  \pi^*Z (g_0(U, \pi^* Y) ) - g_0( [\pi^*Z , \pi^* Y] , U ) - g_0([\pi^*Z , U] , \pi^* Y)    ) \nonumber\\
&&= - U ( g_0( \nabla^0_{\pi^*Z} \pi^*Y - \nabla^0_{\pi^*Y} \pi^*Z  ,U ) ) =0
\end{eqnarray*} 
by Lemma \ref{lemnablaXY}. Hence, $\omega= \pi^* \alpha$ for some $1$-form $\alpha$ on $N$.
\end{proof}

\begin{proposition}
Let $\pi : (M, g_0) \rightarrow (N, g_N)$ be a Riemannian submersion for all with geodesic fibers spanned by a unit vertical field $U$, which is Killing on 
$(M, g_0)$.
Let $Z \in \mathfrak{X}_N$, let $f_t : M \rightarrow M$ be the family of one-parameter  diffeomorphisms generated by $\pi^* Z$, and let $g_t=f_t^* g_0$. 
For all $Y \in \mathfrak{X}_N$, let $\phi_0 Y = \pi_* \mA^0_{\pi^* Y} U$.

If $\{ g_t, t \in (-\epsilon, \epsilon) \} \subset \Riem(M, \mV, g_0)$ 
and $\pi : (M, g_t) \rightarrow (N, g_N)$ is a Riemannian submersion with geodesic fibers for all $(-\epsilon, \epsilon)$, then $Z$ is a Killing field on $(N,g_N)$ and ${\rm d} (\phi_0 Z)^\flat =0$, where
$(\phi_0 Z)^\flat (Y) = g_N( \phi_0 Z ,Y)$ for all $Y \in \mathfrak{X}_N$.
\end{proposition} 
\begin{proof}
We have $B_0=\dt g_t \vert_{t=0}= \mathcal{L}_Z g_0$. If  $\pi : (M, g_t) \rightarrow (N, g_N)$ is a Riemannian submersion with geodesic fibers for all $(-\epsilon, \epsilon)$ and $\{ g_t, t \in (-\epsilon, \epsilon) \} \subset \Riem(M, \mV, g_0)$, then from \eqref{BXY} follows for all $X,Y \in \mathfrak{X}_N$:
\begin{eqnarray*} 
0 &=& (\mathcal{L}_{\pi^* Z} g_0)(\pi^* X, \pi^* Y) = g_0(\nabla^0_{\pi^* X}  \pi^* Z , \pi^* Y) + g_0(\nabla^0_{\pi^* Y}  \pi^* Z , \pi^* X) \nonumber\\
&=&  g_N(\nabla^N_{X}  Z , Y) +  g_N(\nabla^N_{Y}  Z , X) = (\mathcal{L}_{Z} g_N)(X, Y).
\end{eqnarray*}
By Proposition \ref{propAUprojectable}, $\phi_0$ is well defined.
If $\{ g_t, t \in (-\epsilon, \epsilon) \} \subset \Riem(M, \mV, g_0)$ then by Theorem \ref{thS1bundleetaform} there exists a family $\alpha_t$ of $1$-forms on $N$ satisfying \eqref{Bteta}.
For all $Y \in \mathfrak{X}_N$ we have $\pi^* \alpha_0( Y ) = - g_0( [\pi^*Z , \pi^* Y] , U ) = 2 g_0( \mA^0_{ \pi^*Z } U , \pi^* Y ) = 2g_0(\phi_0 Z, Y)$. 
\end{proof}

We note that ${\rm d} (\phi_0 Z)^\flat =0$ holds in particular when $Z$ is a Killing vector field on $(N,g_N)$ and $(N,g_N,\phi_0)$ is a K\"ahler manifold \cite[2.130]{Besse}.


\subsection{Weak contact metric structures}  \label{subsetweakcontact}

We consider 
a generalization of contact metric structures \cite{RovenskiAMPA} and find conditions for obtaining them by variations preserving Riemannian submersions with geodesic fibers. $(\phi_t, U, \eta_t, g_t)$, where $\phi_t$ is a $(1,1)$-tensor field, $U$ is a vector field, $\eta_t$ is a $1$-form and $g_t$ is a Riemannian metric, is called a weak contact metric structure on a manifold $M$ if \eqref{UReeb}-\eqref{detacms} hold, together with the following:
\begin{eqnarray}
\label{weakphitonH}
&& g_t(\phi_t X, \phi_t Y) = g_t(X, Q_t Y) - g_t(U,X) g_t(U,Y) \; \; {\rm for \; all \;} X,Y \in \mathfrak{X}_M ,  \\
\label{weakphitQ}
&& \phi_t^2 = -Q_t + \eta_t \otimes U , \\ 
\label{QUU}
&& Q_t U = U,
\end{eqnarray}
where $Q_t$ is a $(1,1)$-tensor positive definite on $\ker \eta_t$, i.e., $g_t(Q_t X,X)>0$ for all $0 \neq X \in \mathfrak{X}_{\ker \eta_t}$.

Let $\pi :(M,g_0) \rightarrow (N,g_N)$ be a Riemannian submersion with unit vertical vector field $U$ and let  $\{ g_t, t \in (-\epsilon, \epsilon) \} \subset \Riem(M, \mV, g_0)$ be a variation preserving the Riemannian submersion $\pi :(M,g_0) \rightarrow (N,g_N)$.
We say that $(\pi,g_t,U)$ defines a weak contact metric structure on $M$ \cite{RovenskiAMPA}, if for $\phi_t$ defined by \eqref{definitionphit} conditions \eqref{UReeb}-\eqref{detacms} and \eqref{weakphitonH}-\eqref{QUU} are satisfied. We note that $\ker \eta_t = \mH(t)$. 

Let ${\tilde Q}_t = Q_t - {\rm Id}_{TM}$, then the following conditions were considered in \cite{RovenskiAMPA}:
\begin{equation} \label{vrcondition1}
(\nabla^t_E {\tilde Q}_t  ) X = 0 \quad {\rm for \; all} \; E \in \mathfrak{X}_M \; {\rm and \; all} \; X \in \ker \eta_t
\end{equation}
and
\begin{equation} \label{vrcondition2}
R_t({\tilde Q}_t X,Y)Z  \in \ker \eta_t \quad  {\rm for \; all} \; X,Y,Z \in \ker \eta_t .
\end{equation}

We note that from Proposition \ref{propAUprojectable} follows that $(\nabla^t_U {\tilde Q}_t  ) Y = 0$ for all $Y \in \mathfrak{X}_{\mH(t)}$.
We show that for weak contact metric structures defined by Riemannian submersions condition \eqref{vrcondition1} can be further weakened to $P_\mH^t ( ( \nabla_X {\tilde Q}_t ) {Y}) =0$ for all $X, Y \in \mathfrak{X}_{\mH(t)}$.

\begin{proposition}
If $\pi :(M,g_t) \rightarrow (N,g_N)$ is a Riemannian submersion with unit vertical vector field $U$, then $P_\mV^t  (( \nabla^t_X {\tilde Q}_t ) Y) =0$ for all $X,Y \in \mathfrak{X}_{\mH(t)}$.
\end{proposition}
\begin{proof}
We have
\begin{eqnarray*}
&& g_t( ( \nabla^t_X {\tilde Q}_t ) Y , U) \nonumber\\
&& = g_t( \nabla^t_X {\tilde Q}_t Y , U) - g_t(  {\tilde Q}_t \nabla^t_X Y , U) \nonumber\\
&& = g_t( \nabla^t_X Q_t Y , U) - g_t( \nabla^t_X Y , U) - g_t(  Q_t \nabla^t_X Y , U) + g_t(  \nabla^t_X Y , U)  \nonumber\\
&& =  -g_t( \nabla^t_X \phi_t^2 Y , U) + g_t(\nabla^t_X (  g_t(U,Y)U  ) , U  )  \nonumber\\
&& + g_t(  \phi^2 \nabla^t_X Y , U) - g_t( U, \nabla^t_X Y  )g_t(U,U) \nonumber\\
&&= -g_t( \nabla^t_X \phi_t^2 Y , U)  + g_t(  \phi^2 \nabla^t_X Y , U) + g_t( \nabla^t_X U,  Y  ) \nonumber \\
&&= -g_t( ( \nabla^t_X \phi_t^2) Y , U ) - g_t(\phi_t X, Y),
\end{eqnarray*}
where we used \eqref{definitionphit}. 
From \eqref{iotaReeb} and \eqref{detacms} follows $\phi_t U=0$ and hence $\phi_t Z = \phi_t P_{\mH}^t Z$ for all $Z \in \mathfrak{X}_M$.
We can assume that at a point $x \in M$ for $X,Y \in \mH(t)_x$ we have $P_{\mH}^t \nabla^t_X Y =0$, then at $x$ we obtain
\begin{eqnarray*}
&& g_t( ( \nabla^t_X {\tilde Q}_t ) Y , U) \nonumber\\
&&= -g_t( \nabla^t_X \phi^2 Y , U ) + g_t(  \phi^2 P_{\mH}^t \nabla^t_X Y , U) - g_t(\phi_t X, Y) \nonumber\\
&&= g_t(\phi_t^2 Y , \nabla^t_X U ) - g_t(\phi_t X, Y) \nonumber\\
&&= -g_t(\phi_t^2 Y , \phi_t X) - g_t(\phi_t X, Y) \nonumber\\
&&= -g_t(\phi_t Y , X) -  g_t(\phi_t X, Y) =0. 
\end{eqnarray*}
where we used \eqref{phitisomH} and $g_t( \phi_t X,Y) = -g_t( \phi_t Y ,X)$, which follows from \eqref{detacms}.
\end{proof}

\begin{theorem} \label{thweakcms}
Let $\{ g_t, t \in (-\epsilon, \epsilon) \} \subset \Riem(M, \mV, g_0)$ be a 
variation such that $\pi : (M, g_t) \rightarrow (N, g_N)$ is a Riemannian submersion for all $t \in (-\epsilon, \epsilon)$ with geodesic fibers spanned by a unit vertical field $U$, which is Killing on $(M, g_0)$, such that $\mA^t_X U \neq 0$ for all $0 \neq X \in\mathfrak{X}_{\mH(t)}$.
Let ${\rm d} \theta_0$ be the only $2$-form on $N$ such that ${\rm d} \eta_0 = \pi^* ( {\rm d} \theta_0)$. Suppose that \eqref{Bteta} holds with basic $\alpha_t = \alpha$ independent of $t$. 

Then $(\pi,g_t,U)$ defines a weak contact metric structure on $M$ for all $t \in (-\epsilon, \epsilon)$. Moreover, 
\eqref{vrcondition1} holds for all $t \in (-\epsilon, \epsilon)$ if and only if for all $X,Y,Z \in \mathfrak{X}_N$ and every local orthonormal frame $\{ W_{n+1}, \ldots, W_{n+p} \}$ on $N$ we have
\begin{eqnarray}
\label{weakcms0}
&&\sum\nolimits_{i=n+1}^{n+p} \big(g_0( \mA^0_{ \pi^* W_i } Y , U ) g_0( (\nabla^0_{\pi^* Z} \mA)_{ \pi^* X } \pi^* W_i , U ) \nonumber\\
&& + g_0( \mA^0_{ \pi^* W_i } X , U ) g_0( (\nabla^0_{\pi^* Z} \mA)_{ \pi^* Y } \pi^* W_i , U ) \big)=0 , \\
\label{weakcms1}
&& \sum\nolimits_{i=n+1}^{n+p} (
g_0( (\nabla^0_{\pi^* Z} \mA)_{\pi^* W_i }  \pi^* Y , U ) {\rm d}\alpha(X, W_i) +g_0( (\nabla^0_{\pi^* Z} \mA)_{\pi^* W_i }  \pi^* X , U ) {\rm d}\alpha(Y, W_i)\nonumber\\
&&
+ ( \langle \iota_{Y} (\nabla^N_{Z} {\rm d} \alpha ) , \iota_X {\rm d} \theta_0 \rangle_N + \langle \iota_{X} (\nabla^N_{Z} {\rm d} \alpha ) , \iota_Y {\rm d} \theta_0 \rangle_N ) =0 , \\
\label{weakcms2}
&& \langle \iota_Y (\nabla^N_{Z} {\rm d}\alpha) , \iota_X {\rm d} \alpha \rangle_N + \langle \iota_X (\nabla^N_{Z} {\rm d}\alpha) , \iota_Y {\rm d} \alpha \rangle_N =0.
\end{eqnarray} 

Moreover, \eqref{vrcondition2} holds for all $t \in (-\epsilon, \epsilon)$ if and only if for all $X,Y,Z \in \mathfrak{X}_N$ and every local orthonormal frame $\{ W_{n+1}, \ldots, W_{n+p} \}$ on $N$ we have
\begin{eqnarray} \label{Rweakcmst0}
&& \sum\nolimits_{l=n+1}^{n+p}  g_0(\mA^0_{  \pi^* X } U , \mA^0_{ \pi^* W_l } U ) g_0( (\nabla^0_{ \pi^* Z  } \mA^0)_{  \pi^* W_l } \pi^* Y, U) \nonumber\\
&& - g_0((\nabla^0_{ \pi^* Z  } \mA^0)_{ \pi^* X} \pi^*Y , U)  =0 , \\
\label{Rweakcmst1}
&& -\sum\nolimits_{l=n+1}^{n+p}  g_0(\mA^0_{  \pi^* X } U , \mA^0_{ \pi^* W_l } U ) \cdot (\nabla^N_{ Z } {\rm d} \alpha )( W_l , Y ) \nonumber\\
&& + \frac{1}{2} \sum\nolimits_{l=n+1}^{n+p} (  
\langle \iota_{W_l} {\rm d} \theta_0 , \iota_{X}  {\rm d} \alpha \rangle_N + 
\langle \iota_{X} {\rm d} \theta_0 , \iota_{W_l}  {\rm d} \alpha \rangle_N )
\cdot  g_0( (\nabla^0_{ \pi^*Z  } \mA^0)_{  \pi^* W_l } \pi^*Y, U) \nonumber\\
&& + (\nabla^N_{ Z } {\rm d} \alpha )( X , Y )
 =0 , \\
\label{Rweakcmst2}
&& \sum\nolimits_{l=n+1}^{n+p} \langle \iota_{W_l} {\rm d} \alpha , \iota_{X} {\rm d} \alpha \rangle_N g_0( (\nabla^0_{ \pi^*Z  } \mA^0)_{  \pi^* W_l } \pi^*Y, U) \nonumber\\
&& -\frac{1}{2} \sum\nolimits_{l=n+1}^{n+p} ( \langle \iota_{W_l} {\rm d} \theta_0 , \iota_{X}  {\rm d} \alpha \rangle_N  + \langle \iota_{X} {\rm d} \theta_0 , \iota_{W_l}  {\rm d} \alpha \rangle_N )
(\nabla^N_{ Z } {\rm d} \alpha )( W_l , Y ) =0, \\
\label{Rweakcmst3}
&& \sum\nolimits_{l=n+1}^{n+p} \langle \iota_{W_l} {\rm d} \alpha , \iota_{X} {\rm d} \alpha \rangle_N (\nabla^N_{ Z } {\rm d} \alpha )( W_l , Y )=0.
\end{eqnarray}
\end{theorem}
\begin{proof}
From $\mA^t_X U \neq 0$ for all $0 \neq X \in\mathfrak{X}_{\mH(t)}$ follows that $g_t( \mA^t_X U, \mA^t_X U ) = g_t(\phi_t X, \phi_t Y) >0$, hence 
by \eqref{weakphitQ} we obtain that $Q_t$ is positive definite on $\ker \eta_t$ and $(\pi,g_t,U)$ defines a weak contact metric structure on $M$. 

Let $X,Y \in \mathfrak{X}_{\mH(t)}$ and let $E \in \mathfrak{X}_M$. We have
\begin{eqnarray*}
&& g_t( (\nabla^t_{E} {\tilde Q}_t ) X , Y ) =
g_t( \nabla^t_{E} {\tilde Q}_t X , Y ) - g_t( {\tilde Q}_t \nabla^t_{E} X , Y ) \nonumber\\
&&= g_t( \nabla^t_{E} {Q_t} X , Y ) - g_t( \nabla^t_{E} X , Y ) - g_t( {Q}_t \nabla^t_{E} X , Y ) + g_t( \nabla^t_{E} X , Y ) \nonumber\\
&&= -g_t( \nabla^t_{E} \phi_t^2 X , Y ) + g_t( \nabla^t_{E} (g_t(X, U)U) , Y ) \nonumber\\
&&\quad  + g_t( \phi_t^2 \nabla^t_{E} X , Y ) - g_t( g_t(U, \nabla^t_{E} X)  U,Y ) \nonumber\\
&&=-g_t( \nabla^t_E  \phi_t^2 X , Y ) + g_t( \phi_t^2 \nabla^t_{E} X , Y ).
\end{eqnarray*}
Using \eqref{definitionphit} we obtain $\phi^2_t X = \mA_{ \mA_{X} U } U$.
Assuming $P_{\mH}^t \nabla^t_E X =0$ at a point, where the formulas are computed, we obtain
\begin{eqnarray} \label{nablaQXYZ}
&& g_t( (\nabla^t_{E} {\tilde Q}_t ) X , Y ) =
-g_t( \nabla^t_E  \phi_t^2 X , Y ) + g_t( \phi_t^2 \nabla^t_{E} X , Y ) =-g_t( \nabla^t_E  \mA_{ \mA_{X} U } U , Y )  \nonumber\\
&&= -g_t( (\nabla^t_E \mA)_{\mA_{X} U} U , Y ) -g_t( \mA_{ \nabla^t_E \mA_{X} U} U , Y ) -g_t(\mA_{ \mA_{X} U} \nabla^t_E U , Y ) \nonumber\\
&&= -g_t( (\nabla^t_E \mA)_{\mA_{X} U} U , Y ) -g_t( \mA_{ (\nabla^t_E \mA)_{X} U} U , Y ) -g_t(\mA_{ \mA_{ \nabla^t_E X} U} U , Y) \nonumber\\
&&\quad 
-g_t(\mA_{ \mA_{ X} \nabla^t_E U} U , Y)
 -g_t(\mA_{ \mA_{X} U} \nabla^t_E U , Y ) \nonumber\\
&&= -g_t( (\nabla^t_E \mA)_{\mA_{X} U} U , Y ) -g_t( \mA_{ (\nabla^t_E \mA)_{X} U} U , Y ) ,
%
%
\end{eqnarray}
because by $\nabla^t_E U = P_{\mH}^t \nabla^t_E U$, $\mA_{E} = \mA_{P_{\mH}^t E}$,  $\mA_{E} : \mV^t \rightarrow \mH^t$  and $\mA_{E} : \mH^t \rightarrow \mV^t$ for all $E \in \mathfrak{X}_{M}$ we have
\begin{eqnarray*}
&&g_t(\mA_{ \mA_{ \nabla^t_E X} U} U , Y) = g_t(\mA_{ \mA_{ P_{\mH}^t \nabla^t_E X} U} U , Y) =0, \\
&&g_t(\mA_{ \mA_{ X} \nabla^t_E U} U , Y)=g_t(\mA_{ \mA_{ X} P_{\mV}^t \nabla^t_E U} U , Y)=0,\\
&& g_t(\mA_{ \mA_{X} U} \nabla^t_E U , Y ) = g_t(\mA_{ \mA_{X} U} P_{\mV}^t \nabla^t_E U , Y )=0 .
\end{eqnarray*}
If $E=U$, then from $(\nabla^t_{E_1} \mA^t)_{E_2}$ being alternating for all $E_1,E_2 \in \mathfrak{X}_M$ \cite{Tondeur} and \eqref{TondeurnablaUAXYU} follows 
$g_t( (\nabla^t_U \mA)_{\mA_{X} U} U , Y )=-g_t( (\nabla^t_U \mA)_{\mA_{X} U} Y, U ) = 0$ and $P_{\mH}^t (\nabla^t_U \mA)_{X} U=0$, hence
\begin{eqnarray} \label{nablaQXYU}
&& g_t( (\nabla^t_{U} {\tilde Q}_t ) X , Y ) =
-g_t( (\nabla^t_U \mA)_{\mA_{X} U} U , Y ) -g_t( \mA_{ (\nabla^t_U \mA)_{X} U} U , Y ) \nonumber\\
&& = - g_t( \mA_{ P_{\mH}^t (\nabla^t_U \mA)_{X} U} U , Y ) =0.
\end{eqnarray}

Let $Z \in \mathfrak{X}_{\mH(t)}$. 
Let $\{ W_{n+1}, \ldots, W_{n+p} \}$ be a local orthonormal frame on $(N,g_N)$. From \eqref{nablaQXYZ} and \cite[Lemma 5.34(i)]{Tondeur} we obtain
\begin{eqnarray*}
&& g_t( (\nabla^t_{Z} {\tilde Q}_t ) X , Y ) =  -\sum\nolimits_{i=n+1}^{n+p} \big( g_t( (\nabla^t_Z \mA)_{P_{\mH}^t \pi^* W_i } U , Y )g_t(\mA_{X} U , P_{\mH}^t \pi^* W_i) \nonumber\\
&&\quad + g_t( \mA_{P_{\mH}^t \pi^* W_i} U , Y ) g_t( (\nabla^t_Z \mA)_{X} U , P_{\mH}^t \pi^* W_i) \big) \nonumber\\
&&= - \sum\nolimits_{i=n+1}^{n+p} \big( g_t( (\nabla^t_Z \mA)_{P_{\mH}^t \pi^* W_i } Y , U ) g_t(\mA_{X} P_{\mH}^t \pi^* W_i , U) \nonumber\\
&&\quad
 +  g_t( (\nabla^t_Z \mA)_{X}  P_{\mH}^t \pi^* W_i, U) g_t( \mA_{Y} P_{\mH}^t \pi^* W_i , U ) \big) \nonumber\\
&&=- \sum\nolimits_{i=n+1}^{n+p} \big( g_t( (\nabla^t_Z \mA)_{P_{\mH}^t \pi^* W_i } Y , U ) g_t(\mA_{X} P_{\mH}^t \pi^* W_i , U) \nonumber\\
&&\quad
 -  g_t( (\nabla^t_Z \mA)_{P_{\mH}^t \pi^* W_i}X, U) g_t( \mA_{Y} P_{\mH}^t \pi^* W_i , U ) \big) .
\end{eqnarray*}
Using \eqref{dtAXYeta}, \eqref{dtnablaAXYZ} and \eqref{nablaNAXYZU} we obtain for all $X,Y,Z \in \mathfrak{X}_N$
\begin{eqnarray} \label{dtQZXYfinal}
&& g_t( (\nabla^t_{P_{\mH}^t \pi^* Z} {\tilde Q}_t ) P_{\mH}^t \pi^* X ,P_{\mH}^t \pi^* Y ) \nonumber\\
&&=  
-\sum\nolimits_{i=n+1}^{n+p} \big( 
( g_0( (\nabla^0_{\pi^* Z} \mA)_{\pi^* W_i }  \pi^* Y , U ) - t \cdot (\nabla^N_{Z} {\rm d} \alpha )(W_i, Y) ) \nonumber\\
&&\quad \cdot ( g_0(\mA_{ \pi^* X} \pi^* W_i , U) - t \cdot {\rm d}\alpha(X, W_i)  ) \nonumber\\
&&\quad + ( g_0( (\nabla^0_{\pi^* Z} \mA)_{\pi^* W_i }  \pi^* X , U ) - t \cdot (\nabla^N_{Z} {\rm d} \alpha )(W_i, X) ) \nonumber\\
&&\quad \cdot ( g_0(\mA_{ \pi^* Y} \pi^* W_i , U) - t \cdot {\rm d}\alpha(Y, W_i)  ) \big) .
\end{eqnarray}
From \eqref{UReeb} and \eqref{iotaReeb} it follows that $\eta_0$ is a basic form, hence there exists a unique $1$-form $\theta_0$ such that ${\rm d} \eta_0 = \pi^* ( {\rm d} \theta_0)$. Using
\[
g_0(\mA_{ \pi^* X} \pi^* W_i , U) = -{\rm d} \eta_0( \pi^* X ,  \pi^* W_i ) = -{\rm d} \theta_0( X , W_i )
\]
in \eqref{dtQZXYfinal} and comparing terms of a given order in $t$, we obtain \eqref{weakcms0}-\eqref{weakcms2}.

We have 
\[
{\tilde Q}_t P_{\mH}^t \pi^*X =  -\phi_t^2  P_{\mH}^t \pi^*X + g_t(U,  P_{\mH}^t \pi^*X)U - P_{\mH}^t \pi^*X = -\phi_t^2  P_{\mH}^t \pi^*X - P_{\mH}^t \pi^*X.
\]
From \cite[5.37e]{Tondeur}, we obtain for $X,Y,Z \in \mathfrak{X}_N$:
\begin{eqnarray} \label{Rtweakcms1}
&& g_t( R_t( {\tilde Q}_t P_{\mH}^t \pi^*X  , P_{\mH}^t \pi^*Y )  P_{\mH}^t \pi^*Z , U ) =
g_t((\nabla^t_{ P_{\mH}^t \pi^*Z  } \mA^t)_{ {\tilde Q}_t P_{\mH}^t \pi^*X} P_{\mH}^t \pi^*Y, U) \nonumber\\
&&= - g_t((\nabla^t_{ P_{\mH}^t \pi^*Z  } \mA^t)_{ \phi_t^2 P_{\mH}^t \pi^*X} P_{\mH}^t \pi^*Y, U) 
- g_t((\nabla^t_{ P_{\mH}^t \pi^*Z  } \mA^t)_{ P_{\mH}^t \pi^*X} P_{\mH}^t \pi^*Y, U)  \nonumber\\
&&= - \sum\nolimits_{l=n+1}^{n+p} g_t(\phi_t^2 P_{\mH}^t \pi^*X ,  P_{\mH}^t \pi^* W_l ) g_t((\nabla^t_{ P_{\mH}^t \pi^*Z  } \mA^t)_{ P_{\mH}^t \pi^* W_l } P_{\mH}^t \pi^*Y, U) \nonumber\\
&&\quad 
- g_t((\nabla^t_{ P_{\mH}^t \pi^*Z  } \mA^t)_{ P_{\mH}^t \pi^*X} P_{\mH}^t \pi^*Y, U) . 
\end{eqnarray}
From \eqref{dtAWiU1WjU2}, \eqref{d2tAWiU1AWlU2} and \eqref{d3tAWiU1AWlU2} it follows that
\begin{eqnarray} \label{gtphit2weakcms}
&& -g_t(\phi_t^2 P_{\mH}^t \pi^* W_i ,  P_{\mH}^t \pi^* W_l ) =
 g_t(\phi_t P_{\mH}^t \pi^* W_i , \phi_t P_{\mH}^t \pi^* W_l ) \nonumber\\
&&=  g_t(\mA^t_{ P_{\mH}^t \pi^* W_i } U , \mA^t_{ P_{\mH}^t \pi^* W_l } U ) \nonumber\\
&&= g_0(\mA^0_{  \pi^* W_i } U , \mA^0_{ \pi^* W_l } U ) - \frac{1}{2}
t \cdot \sum\nolimits_{j=n+1}^{n+p} \big( ( g_0( \mA^0_{  \pi^*W_l  } \pi^*W_j , U  ) {\rm d} \alpha (W_i, W_j) ) \nonumber\\
&&\quad
+  ( g_0( \mA^0_{  \pi^*W_i  } \pi^*W_j , U  ) {\rm d} \alpha (W_l, W_j) )
\big) 
+ t^2 \cdot \sum\nolimits_{j=n+1}^{n+p} {\rm d} \alpha (W_l, W_j) \cdot {\rm d} \alpha (W_i, W_j) . \nonumber\\ 
\end{eqnarray}
Using \eqref{gtphit2weakcms}, \eqref{dtnablaAXYZ} and \eqref{nablaNAXYZU} in \eqref{Rtweakcms1}, we obtain
\begin{eqnarray*} 
&& g_t( R_t( {\tilde Q}_t P_{\mH}^t \pi^*W_i  , P_{\mH}^t \pi^*W_k )  P_{\mH}^t \pi^*W_r , U ) =
g_t((\nabla^t_{ P_{\mH}^t \pi^*W_r  } \mA^t)_{ {\tilde Q}_t P_{\mH}^t \pi^*W_i} P_{\mH}^t \pi^*W_k, U) \nonumber\\
&&= - \sum\nolimits_{l=n+1}^{n+p} g_t(\phi_t^2 P_{\mH}^t \pi^*W_i ,  P_{\mH}^t \pi^* W_l ) g_t((\nabla^t_{ P_{\mH}^t \pi^*W_r  } \mA^t)_{ P_{\mH}^t \pi^* W_l } P_{\mH}^t \pi^*W_k, U) \nonumber\\
&&\quad 
- g_t((\nabla^t_{ P_{\mH}^t \pi^*W_r  } \mA^t)_{ P_{\mH}^t \pi^*W_i} P_{\mH}^t \pi^*W_k, U)  \nonumber\\
&&= \sum\nolimits_{l=n+1}^{n+p} \big(  g_0(\mA^0_{  \pi^* W_i } U , \mA^0_{ \pi^* W_l } U ) - \frac{1}{2}
t \cdot \sum\nolimits_{j=n+1}^{n+p} \big( ( g_0( \mA^0_{  \pi^*W_l  } \pi^*W_j , U  ) {\rm d} \alpha (W_i, W_j) ) \nonumber\\
&&\quad
+  ( g_0( \mA^0_{  \pi^*W_i  } \pi^*W_j , U  ) {\rm d} \alpha (W_l, W_j) )
\big) 
+ t^2 \cdot \sum\nolimits_{j=n+1}^{n+p} {\rm d} \alpha (W_l, W_j) \cdot {\rm d} \alpha (W_i, W_j) \big) \nonumber\\
&&\quad \cdot ( g_0( (\nabla^0_{ \pi^*W_r  } \mA^0)_{  \pi^* W_l } \pi^*W_k, U) - t \cdot (\nabla^N_{ W_r } {\rm d} \alpha )( W_l , W_k ) ) \nonumber\\
&&\quad - g_0((\nabla^0_{ \pi^*W_r  } \mA^0)_{ \pi^*W_i} \pi^*W_k, U) + t \cdot (\nabla^N_{ W_r } {\rm d} \alpha )( W_i , W_k ) .
\end{eqnarray*}
Comparing terms of a given order in $t$, we obtain that 
\[ 
g_t( R_t( {\tilde Q}_t P_{\mH}^t \pi^*W_i  , P_{\mH}^t \pi^*W_k )  P_{\mH}^t \pi^*W_r , U )=0
\]
if and only if \eqref{Rweakcmst0}-\eqref{Rweakcmst3} hold.
\end{proof} 

\begin{corollary}
Let $\{ g_t, t \in (-\epsilon, \epsilon) \} \subset \Riem(M, \mV, g_0)$ be a 
variation such that $\pi : (M, g_t) \rightarrow (N, g_N)$ is a Riemannian submersion for all $t \in (-\epsilon, \epsilon)$ with geodesic fibers spanned by a unit vertical field $U$, which is Killing on $(M, g_0)$, such that $\mA^t_X U \neq 0$ for all $0 \neq X \in\mathfrak{X}_{\mH(t)}$. 
Let $(\phi_0, U, \eta_0, g_0)$ be a Sasaki structure and 
let ${\rm d} \theta_0$ be the only $2$-form on $N$ such that ${\rm d} \eta_0 = \pi^* ( {\rm d} \theta_0)$. Suppose that \eqref{Bteta} holds with basic $\alpha_t = \alpha$ independent of $t$. 

Then $(\pi,g_t,U)$ defines a weak contact metric structure on $M$, such that for all $t \in (-\epsilon, \epsilon)$ condition \eqref{vrcondition1} holds 
if and only if for all $X,Y,Z \in \mathfrak{X}_N$
\begin{eqnarray}
\label{weakcms1s}
&&
\langle \iota_{Y} (\nabla^N_{Z} {\rm d} \alpha ) , \iota_X {\rm d} \theta_0 \rangle_N + \langle \iota_{X} (\nabla^N_{Z} {\rm d} \alpha ) , \iota_Y {\rm d} \theta_0 \rangle_N =0 , \\
\label{weakcms2s}
&& \langle \iota_Y (\nabla^N_{Z} {\rm d}\alpha) , \iota_X {\rm d} \alpha \rangle_N + \langle \iota_X (\nabla^N_{Z} {\rm d}\alpha) , \iota_Y {\rm d} \alpha \rangle_N =0 ,
\end{eqnarray}
and \eqref{vrcondition2} holds if and only if for all $X,Y,Z \in \mathfrak{X}_N$ and every local orthonormal frame $\{ W_{n+1}, \ldots, W_{n+p} \}$ on $N$ we have
\begin{eqnarray}
\label{Rweakcmst2s}
&&\sum\nolimits_{l=n+1}^{n+p} ( \langle \iota_{W_l} {\rm d} \theta_0 , \iota_{X}  {\rm d} \alpha \rangle_N + \langle \iota_{X} {\rm d} \theta_0 , \iota_{W_l}  {\rm d} \alpha \rangle_N ) \nonumber\\
&&\quad \cdot
(\nabla^N_{ Z } {\rm d} \alpha )( W_l , Y ) =0, \\
\label{Rweakcmst3s}
&& \sum\nolimits_{l=n+1}^{n+p} \langle \iota_{W_l} {\rm d} \alpha , \iota_{X} {\rm d} \alpha \rangle_N \cdot (\nabla^N_{ Z } {\rm d} \alpha )( W_l , Y )=0.
\end{eqnarray}
\end{corollary}
\begin{proof}
If $(\phi_0, U, \eta_0, g_0)$ is a contact metric structure, then  $g_0(\mA^0_{  \pi^* W_i } U , \mA^0_{ \pi^* W_l } U ) = \delta_{il}$. If it is a Sasaki structure, then by \eqref{Sasakiphi} and \eqref{nablaphiandA} we also have 
\[
g_0((\nabla^0_{ \pi^* Z  } \mA^0)_{  \pi^*X} \pi^*Y, U)=0
\]
for all $X,Y,Z \in \mathfrak{X}_N$. The proof then follows from Theorem \ref{thweakcms}.
\end{proof}
We note that in particular \eqref{weakcms1s}-\eqref{Rweakcmst3s} hold for $\nabla^N {\rm d} \alpha=0$.


\section{Higher dimensional fibers} \label{sechighdim}

\subsection{Fat bundles} \label{subsecfatbundles}

We say that a Riemannian submersion $\pi: (M,g_t) \rightarrow (N,g_N)$ with totally geodesic fibers is fat if for all non-zero $X \in \mathfrak{X}_{\mH(t)}$, $U \in \mathfrak{X}_{\mV}$ we have $\sec_{(M,g_t)}(X,U)>0$, i.e., every vertizontal curvature is positive. Then $\pi : (M,g_t) \rightarrow (N,g_N)$ is called a fat bundle \cite{ZillerFatness}. Most of known examples of fat submersions have all vertizontal curvatures constant along the fibers \cite[Problem 1]{ZillerFatness}. We show how variations of metric preserving Riemannian submersion with totally geodesic fibers 
allow to 
produce fat bundles with vertizontal curvature non-constant along fibers.

\begin{proposition} \label{corpressecUXexboth}
Let $\dim N \geq 2$, let $\pi : (M, g_0) \rightarrow (N, g_N)$ be a Riemannian submersion with totally geodesic fibers. 
Let $y \in N$. If there exist $g_N$-orthonormal $X,Y \in T_y N$, $U \in \mathfrak{X}_{\mF_y}$ and a vector field $\xi \in \mathfrak{X}_{\mV}$ whose restriction to fiber $\mF_x$ is a Killing vector field on $(\mF_x , g_0 \vert_{\mF_x})$ for all $x \in N$, such that function
\[
g_0( [ \pi^* X,  \pi^* Y], U ) \cdot g_0( \xi ,  U )
\]
is not constant on the fiber $\mF_y$,
then there exists a metric $g_{s}$ such that $\pi : (M, g_s) \rightarrow (N, g_N)$ is a Riemannian submersion with totally geodesic fibers and sectional curvature $\sec_M( P_{\mH}^s \pi^* X , U)$ non-constant on the fiber $\mF_y$.
\end{proposition}
\begin{proof}
Let ${\cal O}$ be an open neighbourhood of $y$ on which there exist $g_N$-orthonormal vector fields $\{W_{n+1}, \ldots , W_{n+p}\}$ such that at the point $y$ we have: $W_{n+1} = X$, $W_{n+2} = Y$ and $[W_i, W_j]=0$ for all $i,j \in \{n+1, \ldots , n+p\}$.
Let $f \in C^\infty(N)$ be such that 
$Y(f) \neq 0$ and $W_{j}(f) = 0$ at $y$ for all $j \in \{n+1, n+3, \ldots, n+p\}$.

Let $V_{n+1} = (f \circ \pi) \cdot \xi$ and let $V_{n+2} = \ldots = V_{n+p} =0$ on ${\cal O}$.
Let $\{ {\tilde g}_t, t \in (-\epsilon, \epsilon) \} \subset \Riem(M, \mV, g_0)$ be a 
variation such that $\pi : (\pi^{-1}({\cal O}), {\tilde g}_t) \rightarrow ({\cal O}, g_N)$ is a Riemannian submersion for all $t \in (-\epsilon, \epsilon)$ and \eqref{bsharpinv} holds on $\pi^{-1}({\cal O})$, with 
$\{W_{n+1}, \ldots , W_{n+p}\}$ and $\{V_{n+1}, \ldots , V_{n+p} \}$ as above.
Let $\rho \in C^\infty(N)$ be a non-zero, compactly supported in ${\cal O}$ function, such that 
$\rho=1$ on some neighbourhood of $y$. From now on we will consider the variation $g_t$ with $\dt g_t = (\rho \circ \pi) \cdot \dt {\tilde g}_t$. Then $\pi : (M, g_t) \rightarrow (N, g_N)$ is a Riemannian submersion with totally geodesic fibers for all $t \in (-\epsilon, \epsilon)$, by Theorem \ref{corbsharpglobal}. 

We note that
\begin{eqnarray*}
g_0( [ \xi , P_{\mH}^0 \pi^* W_j  ] , U ) = (\mathcal{L}_{\xi} g_0)(P_{\mH}^0 \pi^* W_j , U) ,
\end{eqnarray*}
since $\xi( g_0 ( P_{\mH}^0 \pi^* W_j  , U ) )=0$ by $U$ being vertical and
$g_0([\xi, U] , P_{\mH}^0 \pi^* W_j )=0$ by the vertical distribution being integrable.

From \eqref{dtsecXURSGF} we obtain on $\mF_y$:
\begin{eqnarray} \label{dtsecXURSGFex}
&& \dt \sec_M (P_{\mH}^t \pi^* X , U) \vert_{t=0} \nonumber \\ &&= - \frac{1}{2} \sum\nolimits_{j=n+2}^{n+p} g_0( [P_{\mH}^0 \pi^* X , P_{\mH}^0 \pi^* W_j], U ) \cdot g_0( [V_{n+1} , P_{\mH}^0 \pi^* W_j  ] , U ) \nonumber \\
&&= - \frac{1}{2} \sum\nolimits_{j=n+2}^{n+p} g_0( [ \pi^* X,  \pi^* W_j], U ) \cdot g_0( [ (f \circ \pi) \cdot \xi ,  \pi^* W_j  ] , U ) \nonumber \\
&&= - \frac{1}{2} (f \circ \pi) \cdot \sum\nolimits_{j=n+2}^{n+p} g_0( [ \pi^* X,  \pi^* W_j], U ) \cdot g_0( [ \xi ,  \pi^* W_j  ] , U ) \nonumber \\
&& \quad + \frac{1}{2} \sum\nolimits_{j=n+2}^{n+p} g_0( [ \pi^* X,  \pi^* W_{j}], U ) \cdot g_0( \xi ,  U ) \cdot ( \pi^* W_{j} ) (f \circ \pi) \nonumber \\
&& = - \frac{1}{2} (f \circ \pi) \sum\nolimits_{j=n+2}^{n+p} g_0( [ \pi^* X,  \pi^* W_j], U ) \cdot (\mathcal{L}_{\xi} g_0)( \pi^* W_j , U) \nonumber \\
&& \quad + \frac{1}{2} g_0( [ \pi^* X,  \pi^* Y], U ) \cdot g_0( \xi ,  U ) \cdot Y (f ) .
\end{eqnarray}
With the assumptions about $\xi$, by adjusting $Y(f)$, we can make the right-hand side of
\eqref{dtsecXURSGFex} 
not constant on $\mF_y$.
Then there exists $s \in (-\epsilon, \epsilon)$ for which function $\sec_{(M,g_s)} (P_{\mH}^s \pi^* X, U )$ is not constant on $\mF_y$.
\end{proof}

\begin{theorem} \label{thpressecUXexboth}
Let $\dim N \geq 2$, let $\pi : (M, g_0) \rightarrow (N, g_N)$ be a Riemannian submersion with totally geodesic fibers. 
Let $y \in N$. If there exist $g_N$-orthonormal $X,Y \in T_y N$ and $U \in \mathfrak{X}_{\mF_y}$ such that either
\begin{itemize} 
\item[1$^\circ$] $U$ is a unit Killing field on $(\mF_y, g_0 \vert_{\mF_y})$ and $g_0( [ \pi^* X, \pi^* Y ] , U )$ is not constant on $\mF_y$, or
\item[2$^\circ$] $g_0( [ \pi^* X, \pi^* Y ] , U ) \neq 0$ is constant 
on the fiber $\mF_y$ and there exists a Killing vector field $\zeta$ on $(\mF_y, g_0 \vert_{\mF_y})$ such that 
$g_0(U, \zeta)$ is not constant on $\mF_y$,
\end{itemize} 
then there exists a metric $g_{s}$ such that $\pi : (M, g_s) \rightarrow (N, g_N)$ is a Riemannian submersion with totally geodesic fibers and sectional curvature $\sec_M( P_{\mH}^s \pi^* X , U)$ non-constant on the fiber $\mF_y$.
\end{theorem}
\begin{proof}
If case $1^{\circ}$ holds, let $\xi_y = U$ on $\mF_y$ and if case $2^{\circ}$ holds, let $\xi_y = \zeta$ on $\mF_y$. 

Flows of horizontal lifts of vector fields from $N$ preserve the vertical distribution and, since fibers are totally geodesic, induce isometries between fibers \cite{GromollWalschap}. Using these flows, 
we extend $\xi_y$ to a vertical field $\xi_0$ on $\pi^{-1}({\cal O})$, for some open set ${\cal O} \subset N$, such that restriction of $\xi_0$ to a fiber $\mF_x$ is a Killing field on $(\mF_x, g_0 \vert_{\mF_x})$ for all $x \in \pi^{-1}({\cal O})$. For a function $\rho \in C^\infty(N)$ that has compact support in ${\cal O}$, we extend $(\rho \circ \pi) \cdot \xi_0$ by zero to $M \setminus \pi^{-1}({\cal O})$, obtaining a smooth vertical field $\xi$ on $M$, whose restriction to every fiber is a Killing field on that fiber.
Then
we apply Proposition \ref{corpressecUXexboth}.
\end{proof} 

\begin{example} \label{exampleHopf}
Each of the assumptions of Corollary \ref{thpressecUXexboth} can be satisfied 
for the 
Hopf fibration
$\pi : (S^7, g_0) \rightarrow (S^4, g_N)$, 
with $g_0$ being the round metric on $S^7$, and orthonormal vertical Killing vector fields $E_1, E_2, E_3$ on $(S^7, g_0)$. 

Let $X,Y \in T_y S^4$ be orthonormal, then the function $g(  [ \pi^* X, \pi^* Y ] , E_a )$ is not constant for any $a \in \{1,2,3\}$. Hence, $U=E_a$ satisfies assumption $1^\circ$ of Theorem \ref{thpressecUXexboth}.

On the other hand, vector field $P_{\mV}^t [ \pi^* X, \pi^* Y ]$ is a Killing field of constant norm on the fiber $(\mF_y, g_0 \vert_{\mF_y})$, and hence $U=P_{\mV}^t [ \pi^* X, \pi^* Y ]$ together with $\zeta = E_a$, where $a \in \{1,2,3\}$, satisfies assumption $2^\circ$ of Theorem \ref{thpressecUXexboth}. 

Clearly, similar choices can be made for all Hopf fibrations $\pi_m : (S^{4m+3}, g_0) \rightarrow (\mathbb{H}P^m, g_N)$, $m>1$, and hence we can deform round metric on every sphere $S^{4m+3}$ to obtain a fat bundle with non-constant vertizontal curvature.
We note that manifolds obtained this way are all 
not Sasakian.
\end{example}

\subsection{Homogeneous submersions} \label{subsechomogenous}

From Theorem \ref{thKillingremains} it follows that in general, for variations discussed in Example \ref{exampleHopf}, vector fields $E_a$ will not be Killing for $t \neq 0$. However, we can find also variations that keep all vertical fields $\{E_1,E_2,E_3\}$ Killing, while making some vertizontal curvatures non-constant. We will do it in a more general setting of isometric group action on a manifold.

Let $\mu : G \times M \rightarrow M$ be an isometric action of a compact Lie group $G$ on a Riemannian manifold $(M,g_0)$, whose orbits are principal (i.e., there exists a $G$-equivariant diffeomorphism between any two of them).
Let $N$ be the space of orbits of the action $\mu$ and let $\pi : (M, g_0) \rightarrow (N, g_N)$ be the induced Riemannian submersion. We recall some results from \cite{GromollWalschap} about such submersions.

For $X,Y \in \mathfrak{X}_N$ the vector field $P_{\mV}^0 [\pi^* X, \pi^* Y]$ is a Killing field on every fiber $(\mF_x,  g_0 \vert_{\mF_x} )$, that is \emph{left-invariant}, i.e., $j_{\alpha*} P_{\mV}^0 [\pi^* X, \pi^* Y] = P_{\mV}^0 [\pi^* X, \pi^* Y] \circ j_\alpha$ for all $\alpha \in G$, where $j_\alpha(x) = \alpha \cdot x \equiv \mu(\alpha,x)$. 

The action $\mu$ induces 
\emph{fundamental} vertical Killing fields on $(M,g_0)$, being images of \emph{right-invariant} vector fields on 
$G$. More precisely, for $U(e) \in T_e G$, let ${\tilde U}(x) = \iota_{x*} U(e)$, for all $x \in M$, where $\iota_x(\alpha) = \mu(\alpha,x)$ for all $\alpha \in G$. Then \cite[Proposition 2.3.3]{GromollWalschap}, such ${\tilde U}$ satisfies
\[
{\tilde U}( \mu(\alpha,x) ) = \iota_{x*} {\bar U}(\alpha)
\]
where ${\bar U}(e) = U(e)$ and ${\bar U}$ is a right-invariant field on $G$.

\begin{theorem} \label{thGaction}
Let $\mu : G \times M \rightarrow M$ be an isometric action on $(M,g_0)$ of a compact, connected Lie group $G$  with trivial center, whose orbits are principal and totally geodesic.
Let $N$ be the space of orbits of the action $\mu$ and let $\pi : (M, g_0) \rightarrow (N, g_N)$ be the induced Riemannian submersion, with totally geodesic fibers and 
non-integrable horizontal distribution.
Then there exists a one-parameter family of metrics $\{ g_t, t \in (-\epsilon, \epsilon) \} \subset \Riem(M, \mV, g_0)$ such that $\pi : (M, g_t) \rightarrow (N, g_N)$ is a Riemannian submersion with totally geodesic fibers for all $t \in (-\epsilon, \epsilon)$, and sectional curvature $\sec_{(M,g_s)}( P_{\mH}^s \pi^* X , U)$ non-constant on the fiber $\mF_y$ for some $s \in (-\epsilon,\epsilon)$, $X \in T_y N$ and $U \in \mathfrak{X}_{\mF_y}$. Moreover, $\mu$ is an isometric action of $G$ on $(M, g_t)$ for all $t \in (-\epsilon, \epsilon)$.
\end{theorem}
\begin{proof}
Let $X,Y \in T_y N$ be such that $P_{\mV}^0 [\pi^* X, \pi^* Y] \neq 0$ at some $x \in \pi^{-1}(y)$.
Recall that $j_{\alpha*} P_{\mV}^0 [\pi^* X, \pi^* Y] = P_{\mV}^0 [\pi^* X, \pi^* Y] \circ j_\alpha$ for all $\alpha \in G$.
Since for every $\alpha \in G$, $j_\alpha$ is an isometry, it follows that $g_0 (P_{\mV}^0 [\pi^* X, \pi^* Y] , P_{\mV}^0 [\pi^* X, \pi^* Y])$ is constant along fibers. 

If $G$ has trivial center, no vector field on it is both left-invariant and right-invariant. Therefore there exists a fundamental field $U$ on $(M,g_0)$ such that function $g_0( P_{\mV}^0 [\pi^* X, \pi^* Y] , U )^2$ is not constant on a fiber $\mF_y$ (otherwise $P_{\mV}^0 [\pi^* X, \pi^* Y]$ would be a combination of fundamental fields with constant coefficients, hence a fundamental field itself, and would correspond to both a left- and a right-invariant field on $G$).
Taking the fundamental field $U$ and $\xi = P_{\mV}^0 [\pi^* X, \pi^* Y]$ and applying construction from the proof of Proposition \ref{corpressecUXexboth}, 
we find a variation $g_t \in \Riem(M, \mV, g_0)$ that makes $\sec_{(M,g_t)} (P_{\mH}^t \pi^* X , U )$ non-constant on $\mF_y$.
We will show that for all metrics $g_t$ of this variation all fundamental fields of action $\mu$ remain Killing fields on $(M,g_t)$.

Let ${\tilde U}$ be a fundamental field on $(M,g_0)$.
Since $P_{\mV}^0 [\pi^* X, \pi^* Y](x) = \iota_{x*} \circ V \circ \iota^{-1}_x$ for the left-invariant field $V$ on $G$ such that $V(e) = \iota_{x*}^{-1} P_{\mV}^0 [\pi^* X, \pi^* Y]$ \cite[below Definition 2.3.2.]{GromollWalschap}, we have for all $x \in M$
\begin{eqnarray*}
[ {\tilde U} , P_{\mV}^0 [\pi^* X, \pi^* Y] ] (x) &=& [ \iota_{x*} {\bar U}(e) , \iota_{x*} \circ V \circ \iota^{-1}_x (x) ] \\
&=& \iota_{x*} [ {\bar U}(e) , V(e) ] ,
\end{eqnarray*}
where ${\bar U}$ is the right-invariant field on $G$ with ${\bar U}(e) = \iota_{x*}^{-1} {\tilde U}(x)$.
Since left-invariant and right-invariant vector fields on $G$ commute, it follows that $[ {\tilde U} , P_{\mV}^0 [\pi^* X, \pi^* Y] ]=0$. 
%
Since  $V_{n+1} = (f \circ \pi) \cdot P_{\mV}^0 [\pi^* X, \pi^* Y]$ for some function $f \in C^\infty(N)$, and $V_{n+2}= \ldots = V_{n+p}=0$, we have $[V_i , {\tilde U}] =0$ for all $i \in {n+1, \ldots, n+p}$.
By Theorem \ref{thKillingremains}, ${\tilde U}$ remains a vertical Killing field on all $(M, g_t)$, where $t \in (-\epsilon, \epsilon)$.
By the proof of \cite[Proposition 2.3.3.]{GromollWalschap}, the flow $\varphi_t$ of vertical field ${\tilde U}$ corresponding to $U \in \mathfrak{g}$ is given by $j_{\exp(tU)}$. Since ${\tilde U}$ is Killing, $j_{\exp(tU)}$ is one-parameter group of isometries. Since every element of a compact, connected Lie group is a finite product of elements from $\exp(\mathfrak{g})$, it follows that $j_\alpha$ is an isometry of $(M, g_t)$ for every $\alpha \in G$ and hence $\mu$ is an isometric action on $(M,g_t)$. 
\end{proof} 
\begin{remark}
In particular case of Hopf fibrations $\pi: (S^{4n+3}, g_0) \rightarrow (N , g_N)$, with $g_0$ being a round metric on the sphere, Theorem \ref{thGaction} indicates existence of families of positive sectional curvature metrics on a sphere that make the Hopf fibration a homogeneous Riemannian submersion with totally geodesic fibers. A metric obtained as in Theorem \ref{thGaction} has a vertizontal curvature non-constant along a fiber, hence is not of the form $f^* g_0$ for a diffeomorphism $f$ of $S^{4n+3}$.
\end{remark}

\subsection{SU(2) actions and 3-Sasaki structure} \label{subsecsu2actions}

For a particular case of isometric $SU(2)$ action, in this section we obtain the 
variational formula for 
components of $\mA^t$ tensor, which
determine horizontal and vertizontal curvatures.

In what follows, $\eps_{abc}=0$ if $a=b$, or $b=c$, or $c=a$, otherwise $\eps_{abc}$ is the sign of permutation $(a,b,c)$ of $(1,2,3)$.

\begin{lemma}
Let $\pi : (M,g_0) \rightarrow (N,g_N)$ be a Riemannian submersion with totally geodesic fibers, spanned by $g_0$-orhtonormal vector fields $E_1,E_2,E_3$ such that $[E_a,E_b] = 2 \sum\nolimits_{c=1}^n \eps_{abc} E_c$ for all $a,b \in \{1,2,3\}$.
Let $\{ g_t, t \in (-\epsilon, \epsilon) \} \subset \Riem(M, \mV, g_0)$ be a 
variation such that $\pi : (M, g_t) \rightarrow (N, g_N)$ is a Riemannian submersion with totally geodesic fibers, and \eqref{Viomega} holds with 
$\dt \omega^a_t=0$, i.e., $\omega^a_t = \omega^a$, for all $a \in \{1,2,3\}$, and for all $t \in (-\epsilon, \epsilon)$. Then
\begin{eqnarray} \label{dtAXYomega}
&& 2 \dt g_t( \mA^t_{ P_{\mH}^t \pi^* W_l } P_{\mH}^t \pi^* W_j , U ) \nonumber\\
&&=\sum\nolimits_{a=1}^n g_0(U, E_a) \big(
- 2 {\rm d} \omega^a ( P_{\mH}^t \pi^* W_l , P_{\mH}^t \pi^* W_j) \nonumber\\
&&\quad + 4t \sum\nolimits_{b,c=1}^n \eps_{bca} \omega^b \wedge \omega^c ( \pi^* W_l , \pi^* W_j ) \nonumber\\
&&\quad +t \sum\nolimits_{b=1}^n \omega^b( \pi^* W_l)
E_b( \omega^a( \pi^* W_j ) ) \nonumber\\
&&\quad
 - t\sum\nolimits_{b=1}^n \omega^b( \pi^* W_j) E_b( \omega^a( \pi^* W_l ) )
\big) ,
\end{eqnarray}
where
\[
\omega^b \wedge \omega^c (X , Y) = \frac{1}{2} ( \omega^b( X ) \omega^c( Y ) - \omega^b( Y ) \omega^c( X ) )
\]
for all $X,Y \in \mathfrak{X}_M$.
\end{lemma}
\begin{proof}
Let \eqref{Viomega} hold. 
By Theorem \ref{thBomega}, for all $X,Y \in \mathfrak{X}_N$, from $\omega^a( P_{\mH}^t \pi^* X ) = \omega^a( \pi^* X )$ and, as $\mV$ is integrable and $g_0$-horizontal lifts are projectable, we have
\[
P_{\mH}^t [P_{\mH}^t \pi^* X , P_{\mH}^t \pi^* Y] = P_{\mH}^t [ \pi^* X - P_{\mV}^t \pi^* X ,  \pi^* Y - P_{\mV}^t \pi^* Y ] = P_{\mH}^t [ \pi^* X ,  \pi^* Y].
\]
Hence, it follows that ${\rm d} \omega^a(  P_{\mH}^t \pi^* X,   P_{\mH}^t \pi^* Y ) = {\rm d} \omega^a( \pi^* X, \pi^* Y )$.

For $g_N$-orthonormal set of local fields $\{W_{n+1} , \ldots, W_{n+p}\}$ and $U \in \mathfrak{X}_\mV$ we have
\begin{eqnarray} \label{dtAXYomega1}
&&\sum\nolimits_{m=n+1}^{n+p} g_N( [ W_l , W_j ] , W_m  )g_0(V_m ,U)  - g_t( [V_l , P_{\mH}^t \pi^* W_j ] + [P_{\mH}^t \pi^* W_l, V_j ] , U ) \nonumber \\
&& = \sum\nolimits_{a=1}^n g_0(U, E_a) \big( \sum\nolimits_{m=n+1}^{n+p} g_t( [ P_{\mH}^t \pi^* W_l , P_{\mH}^t \pi^* W_j ] , P_{\mH}^t \pi^* W_m  ) \omega^a(P_{\mH}^t \pi^* W_m) \nonumber\\
&& - g_t( [ \sum\nolimits_{b=1}^n \omega^b(P_{\mH}^t \pi^* W_l) E_b , P_{\mH}^t \pi^* W_j ] + [P_{\mH}^t \pi^* W_l, \sum\nolimits_{b=1}^n \omega^b(P_{\mH}^t \pi^* W_j) E_b  ] , E_a ) \big) \nonumber\\
&& = \sum\nolimits_{a=1}^n g_0(U, E_a) \big( \sum\nolimits_{m=n+1}^{n+p} g_t( [ P_{\mH}^t \pi^* W_l , P_{\mH}^t \pi^* W_j ] , P_{\mH}^t \pi^* W_m  ) \omega^a(P_{\mH}^t \pi^* W_m) \nonumber\\
&& +
P_{\mH}^t \pi^* W_j ( \omega^a( P_{\mH}^t \pi^* W_l ) ) - P_{\mH}^t \pi^* W_l ( \omega^a( P_{\mH}^t \pi^* W_j )) \nonumber\\
&& - \sum\nolimits_{b=1}^n \omega^b(P_{\mH}^t \pi^* W_l)g_t([E_b , P_{\mH}^t \pi^* W_j] , E_a) \nonumber\\
&& - \sum\nolimits_{b=1}^n \omega^b(P_{\mH}^t \pi^* W_j)g_t([ P_{\mH}^t \pi^* W_l , E_b] , E_a)
\big) \nonumber\\
&&= \sum\nolimits_{a=1}^n g_0(U, E_a) \big( \omega^a( [ P_{\mH}^t \pi^* W_l , P_{\mH}^t \pi^* W_j ] ) - P_{\mH}^t \pi^* W_j ( \omega^a( \pi^* W_l ) ) \nonumber\\
&& - P_{\mH}^t \pi^* W_l (  \omega^a( \pi^* W_j ) ) 
+ \sum\nolimits_{b=1}^n \omega^b(P_{\mH}^t \pi^* W_l) (\mathcal{L}_{E_b} g_t)( P_{\mH}^t \pi^* W_j,  E_a ) \nonumber\\
&& - \sum\nolimits_{b=1}^n \omega^b(P_{\mH}^t \pi^* W_j)(\mathcal{L}_{E_b}
g_t)( P_{\mH}^t \pi^* W_l,  E_a ) \big) \nonumber\\
&&= \sum\nolimits_{a=1}^n g_0(U, E_a) \big(
- 2 {\rm d} \omega^a (P_{\mH}^t \pi^* W_l , P_{\mH}^t \pi^* W_j) \nonumber\\
&& + \sum\nolimits_{b=1}^n \omega^b(P_{\mH}^t \pi^* W_l) (\mathcal{L}_{E_b} g_t)( P_{\mH}^t \pi^* W_j,  E_a ) \nonumber\\
&& - \sum\nolimits_{b=1}^n \omega^b(P_{\mH}^t \pi^* W_j)(\mathcal{L}_{E_b}
g_t)( P_{\mH}^t \pi^* W_l,  E_a )
\big)
\end{eqnarray}
From \eqref{Viomega} and $\dt \omega^a_t=0$, we obtain $\dt V_i =0$ for all $i \in \{n+1, \ldots, n+p\}$. Then, by Theorem \ref{thKillingremains} and Remark \ref{remd2tLKg} we have, using $[E_b,E_c] = 2 \sum\nolimits_{a=1}^n \eps_{bca}E_a$,
\begin{eqnarray} \label{LEbgtWjEa}
&& (\mathcal{L}_{E_b} g_t)( P_{\mH}^t \pi^* W_j,  E_a ) =
(\mathcal{L}_{E_b} g_0)( P_{\mH}^0 \pi^* W_j,  E_a ) + g_t([E_b, V_j], E_a) \cdot t \nonumber\\
&&= g_t([E_b, \sum\nolimits_{c=1}^n \omega^c( P_{\mH}^t \pi^* W_j ) E_c], E_a) \cdot t \nonumber\\
&&= \sum\nolimits_{c=1}^n \omega^c( P_{\mH}^t \pi^* W_j ) g_t([E_b, E_c] ,E_a) t + E_b( \omega^a( P_{\mH}^t \pi^* W_j ) ) t \nonumber\\
&& = 2 \sum\nolimits_{c=1}^n \eps_{bca} \omega^c( P_{\mH}^t \pi^* W_j ) t + E_b( \omega^a( P_{\mH}^t \pi^* W_j ) ) t.
\end{eqnarray}
From \eqref{LEbgtWjEa}, \eqref{dtAWiWjU} and \eqref{dtAXYomega1} follows \eqref{dtAXYomega}.
\end{proof}

We consider a family of variations $\{ g_t, t \in (-\epsilon, \epsilon) \} \subset \Riem(M, \mV, g_0)$ defined by fundamental fields on $(M,g_0)$, noting that in general these vector fields 
stop being Killing for $t \neq 0$ by Theorem \ref{thKillingremains}. We examine some properties of the variations on $3$-Sasaki manifolds.

Recall \cite{Blair} that $\{ (\phi^a_t, \xi_a, \eta^a_t,g_t), a \in \{1,2,3\} \}$, where $g_t$ is a Riemannian metric, $\phi^a_t$ is a $(1,1)$-tensor, $\xi_a$ is a nowhere vanishing vector field and $\eta^a_t$ is a $1$-form on $M$ is called an almost contact metric 3-structure if the following hold for all $a,b \in \{1,2,3\}$:
\begin{eqnarray} \label{almostcontact3structure1}
&& \iota_{\xi_a} \eta^b_t = \delta_{ab} , \\&&
\label{almostcontact3structure2}
\iota_{\xi_a} {\rm d} \eta^a_t =0 , \\&&
\label{almostcontact3structure3}
\phi^a_t \phi^a_t = -{\rm Id}_{TM} + \eta^a_t \otimes \xi_a,  \\&&
\label{almostcontact3structure4}
g(\phi^a_t X, \phi^a_t Y) = g_t(X,Y) - \eta^a_t(X)\eta^a_t(Y)
\end{eqnarray} 
and for an even permutation $(a,b,c)$ of $(1,2,3)$ we have:
\begin{eqnarray} \label{almostcontact3structure5}
&& \xi_c = \phi^a_t \xi_b = -\phi^b_t \xi_a,  \\
\label{almostcontact3structure6}
&&
\phi^c_t = \phi^a_t \phi^b_t - \eta^b_t \otimes \xi_a = -\phi^b_t \phi^a_t +  \eta^a_t \otimes \xi_b, \\&&
\label{almostcontact3structure7}
\eta^c_t = \eta^a_t \circ \phi^b_t = -\eta^b_t \circ \phi^a_t
\end{eqnarray}
By \cite{Kashiwada} (\cite[Theorem 14.1]{Blair}), if additionally every $(\phi^a_t, \xi_a, \eta^a_t , g_t)$ is a contact metric structure, i.e., ${\rm d}\eta^a_t(X,Y) = g_t(X , \phi^a Y)$ for all $X,Y \in \mathfrak{X}_M$, then every $(\phi^a_t, \xi_a, \eta^a_t , g_t)$ a Sasaki structure and $\{ (\phi^a_t, \xi_a, \eta^a_t,g_t), a \in \{1,2,3\} \}$ is called a 3-Sasaki structure.

\begin{proposition}
Let $(M,g_0)$ be a $3$-Sasaki manifold. Let $\{ g_t, t \in (-\epsilon, \epsilon) \} \subset \Riem(M, \mV, g_0)$ be a variation preserving Riemannian submersion with totally geodesic fibers $\pi : (M, g_t) \rightarrow (N, g_N)$, such that \eqref{Btomega} holds with all $\omega^a$ basic.

Let $\phi^a_t X = -\nabla^t_{X} \xi_a$ and $\eta^a_t(X) = g_t(X,\xi^a)$ for all $X \in \mathfrak{X}_M$ and all $a \in \{1,2,3\}$.
Then $\{ (\phi^a_t, \xi_a, \eta^a_t,g_t), a \in \{1,2,3\} \}$ is an almost contact metric 3-structure if and only if all $\omega^a$ are closed and for all $b,c \in \{1,2,3\}$ we have $\omega^b \wedge \omega^c =0$.
Moreover, $\{ (\phi^a_t, \xi_a, \eta^a_t,g_t), a \in \{1,2,3\} \}$ is not a 3-Sasaki structure for some $0 \neq t \in (-\epsilon, \epsilon)$.
\end{proposition}
\begin{proof}
That $\{ (\phi^a_t, \xi_a, \eta^a_t,g_t), a \in \{1,2,3\} \}$ is not a 3-Sasaki structure for some $t \neq 0$, follows from Theorem \ref{thKillingremains}, by which 
there exist $t \neq 0$, $X \in \mathfrak{X}_{\mH(t)}$ and $a,b \in \{1,2,3\}$ such that
\[
{\rm d}\eta^a_t( \xi_b, X ) = -\frac{1}{2} g_t(\xi^a , [\xi_b, X]) = \frac{1}{2} (\mathcal{L}_{\xi_b}g_t)(X, \xi^a) \neq 0.
\]
On the other hand, for all $a,b \in \{1,2,3\}$ and all $X \in \mathfrak{X}_{\mH(t)}$ we have
\begin{eqnarray} \label{xibphiaX}
&& g_t(\xi_b, \phi^a_t X) = -g_t( X , \phi_a^t \xi_b) = g_t( X, \nabla^t_{\xi_b} \xi_a) \nonumber\\
&&= \frac{1}{2} g_t( [\xi_b , \xi_a] , X) +
\frac{1}{2} g_t( \nabla^t_{\xi_b} \xi_a + \nabla^t_{\xi_a} \xi_b , X ) =0, 
\end{eqnarray}
since $\mV$ is integrable and totally geodesic on $(M,g_t)$ for all $t \in (-\epsilon, \epsilon)$.

For all $a,b,c \in \{1,2,3\}$ and for all $X \in \mathfrak{X}_{\mH(t)}$ we have $\dt g_t( \nabla^t_{\xi_a} \xi_b , \xi_c )=0$, which together with \eqref{xibphiaX} implies that \eqref{almostcontact3structure1} and \eqref{almostcontact3structure5} are satisfied,  \eqref{almostcontact3structure3}-\eqref{almostcontact3structure4} and \eqref{almostcontact3structure6} hold for all vertical vectors,
and \eqref{almostcontact3structure7} holds for all vectors.

Since $\mV$ is totally geodesic on $(M,g_t)$ for all $t \in (-\epsilon, \epsilon)$ and $\{ g_t, t \in (-\epsilon, \epsilon) \} \subset \Riem(M, \mV, g_0)$ we have
\begin{eqnarray*}
\iota_{\xi_a} {\rm d} \eta^a_t (X) &=& -\frac{1}{2} g_t( \xi_a , [\xi_a, X] ) = -\frac{1}{2} g_t( \xi_a,  \nabla^t_{\xi_a} X - \nabla^t_X \xi_a  ) \nonumber\\ &=& \frac{1}{2} g_t(\nabla^t_{\xi_a} \xi_a , X ) + \frac{1}{4} X( g_0(\xi_a,\xi_a) ) =0
\end{eqnarray*}
and since $\iota_{\xi_a} {\rm d} \eta^a_t (\xi_b) = -\frac{1}{2} g_0(\xi_a , [\xi_a, \xi_b]) =0$, we see that \eqref{almostcontact3structure2} holds.

From \eqref{xibphiaX} we obtain
\begin{equation} \label{defphiSU2H}
\phi^a_t X = - P_{\mH}^t \nabla^t_{X} \xi_a = - \mA^t_{X} \xi_a \quad {\rm for \; all \;} X \in \mathfrak{X}_{\mH(t)}.
\end{equation}  

Let $\{W_{n+1}, \ldots,  W_{n+p}\}$ be a local $g_N$-orthonormal frame on $N$. We will show that if \eqref{almostcontact3structure4} holds for all $t \in (-\epsilon,\epsilon)$, then 
$\omega^b \wedge \omega^c=0$ for all $a,b,c \in \{1,2,3\}$.
If \eqref{almostcontact3structure4} holds for all $t \in (-\epsilon,\epsilon)$, then for all $i,j \in \{n+1, \ldots, n+p\}$ and $a \in \{1,2,3\}$ we have
\begin{eqnarray} \label{dtphi2almost3contact}
&& \dt  g_t( \phi^a_t {P_{\mH}^t \pi^* W_i } , \phi^a_t {P_{\mH}^t \pi^* W_l} ) \nonumber\\
&& =
\dt g_t(P_{\mH}^t \pi^* W_i , P_{\mH}^t \pi^* W_l ) - \dt ( \eta^a_t( P_{\mH}^t \pi^* W_i ) \eta^a_t(P_{\mH}^t \pi^* W_l ) ) \nonumber\\
&& = \dt g_N( W_i, W_l ) - \dt ( g_t( \xi_a , P_{\mH}^t \pi^* W_i ) g_t( \xi_a , P_{\mH}^t \pi^* W_l ) ) =0.
\end{eqnarray} 
From \eqref{Viomega} and $\omega^a$ being basic, it follows that for all $i,j \in \{n+1, \ldots, n+p\}$ and $a \in \{1,2,3\}$
\begin{eqnarray} \label{SU2ViVjEa}
&& g_t([V_i, V_j] , E_a) = \sum\nolimits_{b,c=1}^n g_t( [ \omega^b( P_{\mH}^t \pi^* W_i ) E_b , \omega^c( P_{\mH}^t \pi^* W_j ) E_c  ] , E_a ) \nonumber\\
&& = 2 \sum\nolimits_{b,c=1}^n \eps_{bca} \omega^b( P_{\mH}^t \pi^* W_i ) \omega^c( P_{\mH}^t \pi^* W_j ) \nonumber\\
&& = \sum\nolimits_{b,c=1}^n \eps_{bca} \omega^b( P_{\mH}^t \pi^* W_i ) \omega^c( P_{\mH}^t \pi^* W_j ) + \sum\nolimits_{b,c=1}^n \eps_{cba} \omega^c( P_{\mH}^t \pi^* W_i ) \omega^b( P_{\mH}^t \pi^* W_j )  \nonumber\\
&& = \sum\nolimits_{b,c=1}^n \eps_{bca} ( \omega^b( P_{\mH}^t \pi^* W_i ) \omega^c( P_{\mH}^t \pi^* W_j ) - \omega^c( P_{\mH}^t \pi^* W_i ) \omega^b( P_{\mH}^t \pi^* W_j ) ) \nonumber\\
&& = 2 \sum\nolimits_{b,c=1}^n \eps_{bca} (\omega^b \wedge \omega^c)( P_{\mH}^t \pi^* W_i  , P_{\mH}^t \pi^* W_j ).
\end{eqnarray}
From \eqref{defphiSU2H}, \eqref{d4tAWiU1AWlU2} and \eqref{SU2ViVjEa} we obtain
\begin{eqnarray*}
&& 4 \frac{\partial^4}{\partial t^4} g_t( \phi^a_t {P_{\mH}^t W_i } , \phi^a_t {P_{\mH}^t W_l} ) = 4 \frac{\partial^4}{\partial t^4} g_t( \mA^t_{P_{\mH}^t W_i } E_a , \mA^t_{P_{\mH}^t W_l} E_a ) \nonumber\\
&&= 24 \sum\nolimits_{j=n+1}^{n+p} g_t( [V_l, V_j], E_a  ) g_t( [V_i, V_j], E_a  ) \nonumber\\
&&= 96 \sum\nolimits_{j=n+1}^{n+p} \sum\nolimits_{b,c=1}^n
\eps_{bca}  (\omega^b \wedge \omega^c)( P_{\mH}^t \pi^* W_l  , P_{\mH}^t \pi^* W_j ) \nonumber\\
&&\cdot \sum\nolimits_{d,e=1}^n \eps_{dea}  (\omega^d \wedge \omega^e)( P_{\mH}^t \pi^* W_i  , P_{\mH}^t \pi^* W_j ) .
\end{eqnarray*}
Using $\sum\nolimits_{a=1}^n \eps_{bca} \eps_{dea} = \delta_{bd} \delta_{ce} - \delta_{be} \delta_{cd}$, we obtain
\begin{eqnarray*} 
&& 4 \sum\nolimits_{a=1}^n \frac{\partial^4}{\partial t^4} g_t( \phi^a_t {P_{\mH}^t W_i } , \phi^a_t {P_{\mH}^t W_l} )  \nonumber\\
&&= 96 \sum\nolimits_{b,c=1}^n \sum\nolimits_{j=n+1}^{n+p} \big(  (\omega^b \wedge \omega^c)( P_{\mH}^t \pi^* W_l  , P_{\mH}^t \pi^* W_j ) (\omega^b \wedge \omega^c)( P_{\mH}^t \pi^* W_i , P_{\mH}^t \pi^* W_j ) \nonumber\\
&&\quad - (\omega^b \wedge \omega^c)( P_{\mH}^t \pi^* W_l  , P_{\mH}^t \pi^* W_j ) (\omega^c \wedge \omega^b)( P_{\mH}^t \pi^* W_i , P_{\mH}^t \pi^* W_j )   \big) \nonumber\\
&&= 192 \sum\nolimits_{b,c=1}^n (\omega^b \wedge \omega^c)( P_{\mH}^t \pi^* W_l  , P_{\mH}^t \pi^* W_j ) (\omega^b \wedge \omega^c)( P_{\mH}^t \pi^* W_i , P_{\mH}^t \pi^* W_j ) .\nonumber\\
&&
\end{eqnarray*}
Let $\Omega^{bc}$ by a $p \times p$ matrix with entries $\Omega^{bc}_{ij} =  (\omega^b \wedge \omega^c)( P_{\mH}^t \pi^* W_i , P_{\mH}^t \pi^* W_j )$, then $\Omega^{bc}_{ij} = - \Omega^{bc}_{ji}$ and 
\begin{equation} \label{d4tphi}
4 \sum\nolimits_{a=1}^n \frac{\partial^4}{\partial t^4} g_t( \phi^a_t {P_{\mH}^t W_i } , \phi^a_t {P_{\mH}^t W_l} ) = - 192 \sum\nolimits_{b,c=1}^n \sum\nolimits_{j=n+1}^{n+p} \Omega^{bc}_{ij} \Omega^{bc}_{jl}.
\end{equation}
Since $n=3$ and $\Omega^{aa}=\Omega^{bb}=0$, from \eqref{d4tphi} we obtain
\begin{equation} \label{d4tphi2}
4 \sum\nolimits_{a=1}^n \frac{\partial^4}{\partial t^4} g_t( \phi^a_t {P_{\mH}^t W_i } , \phi^a_t {P_{\mH}^t W_l} ) = - 384 \sum\nolimits_{j=n+1}^{n+p} \Omega^{\beta \gamma}_{ij} \Omega^{\beta \gamma}_{jl},
\end{equation}
where $\beta<\gamma$ are the two indices from  $\{1,2,3\}$ that are different from $a$.
Since for an antisymmetric matrix $\Omega^{\beta \gamma}$ we have $(\Omega^{\beta \gamma})^2=0$ if and only if $\Omega^{\beta \gamma}=0$, from \eqref{dtphi2almost3contact} and \eqref{d4tphi} it follows that 
\begin{equation} \label{omegabc0}
\omega^b \wedge \omega^c=0 \quad {\rm for \; all} \; b,c \in \{1,2,3\}.
\end{equation}

Similarly, we will show that if \eqref{almostcontact3structure4} holds for all $t \in (-\epsilon,\epsilon)$, then 
$\omega^a$ is closed for all $a \in \{1,2,3\}$. Indeed, from \eqref{d2tAWiU1AWlU2}, \eqref{dtAXYomega}, \eqref{SU2ViVjEa} and \eqref{omegabc0} we obtain
\begin{eqnarray} \label{dtphiadomega}
&& 4\ddt g_t( \phi^a_t P_{\mH}^t W_i , \phi^a_t P_{\mH}^t W_l ) = 4 \ddt g_t( \mA^t_{P_{\mH}^t W_i } \xi_a , \mA^t_{P_{\mH}^t W_l} \xi_a ) \nonumber\\
&&=
-8 \sum\nolimits_{i=n+1}^{n+p} {\rm d} \omega^a(P_{\mH}^t W_i , P_{\mH}^t W_j ) \cdot
{\rm d} \omega^a(P_{\mH}^t W_l , P_{\mH}^t W_j ).
\end{eqnarray}
Defining antisymmetric $p \times p$ matrix $\Omega^a$ with entries $\Omega^a_{ij}= {\rm d} \omega^a(P_{\mH}^t W_i , P_{\mH}^t W_j )$, from \eqref{dtphi2almost3contact} and \eqref{dtphiadomega}  we obtain $(\Omega^a)^2 =0$ and hence ${\rm d}\omega^a=0$ for all $a \in \{1,2,3\}$.

On the other hand, if ${\rm d}\omega^a=0$ and $\omega^b \wedge \omega^c=0$ for all $a,b,c \in \{1,2,3\}$, from 
\eqref{dtAXYomega} it follows that $g_t(\phi^a_t P_{\mH}^t \pi^* W_i , P_{\mH}^t \pi^* W_j )= -g_t( \mA^t_{ P_{\mH}^t \pi^* W_i }  \xi_a, P_{\mH}^t \pi^* W_j ) = g_t( \mA^t_{ P_{\mH}^t \pi^* W_i } P_{\mH}^t \pi^* W_j,  \xi_a ) =g_0( \mA^0_{  \pi^* W_i }  \pi^* W_j,  \xi_a )= g_0(\phi^a_0  \pi^* W_i , \pi^* W_j )$. Hence, for an even permutation $\{a,b,c\}$ of $\{1,2,3\}$, we have
\begin{eqnarray*}
&& - g_t( \phi^b_t \phi^a_t P_{\mH}^t \pi^* W_i , P_{\mH}^t \pi^* W_l ) =
g_t( \phi^a_t P_{\mH}^t \pi^* W_i, \phi^b_t P_{\mH}^t \pi^* W_l ) \nonumber\\
&& = \sum\nolimits_{j=n+1}^{n+p} g_t( \phi^a_t P_{\mH}^t \pi^* W_i,  P_{\mH}^t \pi^* W_j) g_t( \phi^b_t P_{\mH}^t \pi^* W_l  ,  P_{\mH}^t \pi^* W_j) \nonumber\\
&& =  \sum\nolimits_{j=n+1}^{n+p} g_0( \phi^a_0 \pi^* W_i,  \pi^* W_j) g_0( \phi^b_0 \pi^* W_l  ,  \pi^* W_j)  \nonumber\\
&& = g_0( \phi^a_0 \pi^* W_i, \phi^b_0 \pi^* W_l ) = -g_0( \phi^b_0 \phi^a_0 \pi^* W_i,  \pi^* W_l ) \nonumber\\
&&=  g_0( \phi^c_0 \pi^* W_i,  \pi^* W_l ) = g_t( \phi^c_t P_{\mH}^t \pi^* W_i , P_{\mH}^t \pi^* W_l ).
\end{eqnarray*}
Therefore, $\phi^b_t \phi^a_t P_{\mH}^t \pi^* W_i = \phi^c_t P_{\mH}^t \pi^* W_i$ and it follows that \eqref{almostcontact3structure6} holds on $\mH(t)$, which is spanned by $\{ P_{\mH}^t \pi^* W_{n+1}, \ldots , P_{\mH}^t \pi^* W_{n+p} \}$. Similarly we obtain that \eqref{almostcontact3structure3}-\eqref{almostcontact3structure4} hold on $\mH(t)$.
\end{proof}

\begin{theorem} \label{thdifeo3Sasaki}
Let $(M,g_0)$ be an analytic, simply connected $3$-Sasaki manifold. Let $\{ g_t, t \in (-\epsilon, \epsilon) \} \subset \Riem(M, \mV, g_0)$ be a variation preserving Riemannian submersion with totally geodesic fibers $\pi : (M, g_t) \rightarrow (N, g_N)$, such that \eqref{Btomega} holds with all $\omega^a$ basic, closed and analytic, and for all $b,c \in \{1,2,3\}$ we have $\omega^b \wedge \omega^c =0$.

Then there exists a family of diffeomorphisms $f_t : M \rightarrow M$ such that $g_t = f^* g_0$ for all $t \in (-\epsilon, \epsilon)$.
\end{theorem}
\begin{proof}
Since for all $b,c \in \{1,2,3\}$ we have $\omega^b \wedge \omega^c =0$, all forms $\omega^b$ are linearly dependent. We assume that not all those forms are zero. 
Then there exists an orthonormal frame of vertical Killing fields $\xi_1, \xi_2 , \xi_3$ on $M$, where $\xi_1$ is unit and orthogonal to $\bigcap_{a=1}^3 \ker \omega^a$, and $[\xi_a, \xi_b] = 2 \sum\nolimits_{c=1}^n \eps_{abc} \xi_c$. In this frame, the following analogue of \eqref{Btomega} holds:
\[
B_t(\xi_a ,X)  = \omega^a(X) , \quad {\rm for \, all \,} X \in \mathfrak{X}_M, a \in \{1,2,3\}
\]
with $\omega^2=\omega^3=0$, and $\omega^1$ basic and closed.

Since $\xi_1$ is a Killing field on $(M,g_t)$ for all $t \in (-\epsilon, \epsilon)$, for all $x\in M$ there exists an open set $M_x \subset M$ and a Riemannian submersion $\pi_1 :(M_x , g_t) \rightarrow (N_1,g_{N_1})$, with geodesic fibers along $U=\xi_1$. 
Clearly, variation $g_t$ preserves Riemannian submersion with geodesic fibers  $\pi_1 :(M_x , g_t) \rightarrow (N_1,g_{N_1})$, we will show that for this variation \eqref{Bteta} holds with ${\rm d} \alpha_t=0$. To do this, we will show that $\omega^1$, which by assumption is basic with respect to $\pi$, is basic also with respect to $\pi_1$. 

Since $\omega^1$ is basic with respect to $\pi$, we have $\iota_{\xi_1}\omega^1=0$. We have
\begin{eqnarray*}
2 \iota_{\xi_1} {\rm d} \omega^1 (\xi_2) &=& 2 {\rm d} \omega^1 (\xi_1,\xi_2) = \xi_1( \omega^1 (\xi_2) ) - \xi_2( \omega^1 (\xi_1) ) - \omega^1([\xi_1,\xi_2]) = 0,
\end{eqnarray*}
since $\omega^1$ vanishes on $\xi_1,\xi_2,\xi_3$. Similarly, $\iota_{\xi_1} {\rm d} \omega^1 (\xi_3)=0$.
Let $y \in M_x$ and let $X_y \in T_y M$ be horizontal with respect to $\pi$. Then there exists a vector field $X \in \mathfrak{X}_N$ such that $\pi^* X = X_y$ at $y$ and we have ${\rm d} \omega^1 (\xi_1, \pi^* X) = 0$ by $\omega^1$ being basic with respect to $\pi$. It follows that $\iota_{\xi_1} {\rm d} \omega^1=0$, and hence $\omega^1$ is basic with respect to $\pi_1$. By the assumption $\omega^1$ is also closed, hence $\omega_1 = \pi_1^* \alpha_t$ for some $1$-form $\alpha_t$ on $N_1$ such that ${\rm d} \alpha_t=0$. 
From Theorem \ref{thdiffequivalence} follows existence for every $t \in (-\epsilon, \epsilon)$ of a local diffeomorphism, that extends to a global one $f_t : M \rightarrow M$, such that $g_t = f_t^* g_0$.
\end{proof}

\end{document}